%% file: time_changes_unipotents.tex
\documentclass[10pt,a4paper]{article}

\usepackage{nag}    
\usepackage{iftex,ifthen,ifpdf}  %
\makeatletter%
\newif\ifdvi\dvitrue
\@ifundefined{pdfoutput}{}{\ifnum\pdfoutput>0 \dvifalse\fi}%
\makeatother%
\ifdvi
  \usepackage[dvips]{graphicx}   
\else
  \usepackage[pdftex]{graphicx}  
  \usepackage{epstopdf}          
\fi
\ifPDFTeX                        
    \usepackage[T1]{fontenc}     
    \usepackage{lmodern}
    \usepackage[utf8]{inputenc}  
    \wlog{T1-fontenc}
\else                            
    \usepackage{fontspec}
    \setromanfont[Ligatures=TeX]{Latin Modern Roman}
    \usepackage{unicode-math}
     \defaultfontfeatures{Ligatures=TeX}
\fi%
\ifXeTeX                                  
    \usepackage{polyglossia}       
    \setdefaultlanguage{english}
\else                            
    \usepackage[english]{babel}%
\fi%
\usepackage{mathtools}
\mathtoolsset{mathic=true}
\usepackage{amsfonts,amssymb}
\usepackage{amsthm}
\ifPDFTeX                        
 \ifdvi
    \relax
 \else
    \usepackage[protrusion=true,expansion, tracking,spacing,kerning]{microtype}
 \fi
\else
\usepackage[protrusion=true]{microtype}
\fi
\usepackage[usenames,dvipsnames,svgnames,table]{xcolor}
\usepackage{hyperref}
\hypersetup{%
  citecolor=NavyBlue, 
  colorlinks=true, 
  filecolor=magenta, 
  linkcolor=Maroon, 
  pdfcreator={pdfLaTeX}, 
  pdffitwindow=false, 
  pdfmenubar=true, 
  pdfnewwindow=true, 
  pdfproducer={Producer}, 
  pdfstartview={FitH}, 
  pdfsubject={Mathematics}, 
  pdftoolbar=true, 
  unicode=true, 
  urlcolor=cyan, 
  anchorcolor=cyan 
}
\usepackage{titlesec}
\titleformat{\section}
{\centering \normalfont\normalsize\scshape}{
\S\thesection.}{0pt}{\enspace}[]
\titleformat{\subsection}[runin]
{\normalfont\normalsize\bfseries}{\thesubsection.}{1em}{}[]
\titleformat{\subsubsection}[runin]
{\normalfont\normalsize\bfseries}{\thesubsubsection.}{1em}{}
\usepackage{calc}
\usepackage{lflam}
\usepackage{listings}\usepackage[capitalise,noabbrev]{cleveref}
\usepackage{thm_env_en_ante}

\title{On rigidity properties of time-changes of unipotent flows}
\author{Mauro Artigiani\thanks{School of Engineering, Science and Technology,
  Universidad del Rosario, Bogot\'a, Colombia. \textit{Email address:}
  \href{mauro.artigiani@gmail.com}{mauro.artigiani@gmail.com}} 
    \and
    Livio Flaminio\thanks{Universit\'e Lille, CNRS, UMR 8524 - Laboratoire Paul
  Painlev\'e, F-59000 Lille, France. \textit{Email address:}
  \href{livio.flaminio@univ-lille.fr}{livio.flaminio@univ-lille.fr}} 
    \and 
  Davide Ravotti\thanks{Universit\"at Wien, Department of Mathematics, 
  Oskar-Morgenstern-Platz 1, 1090 Wien, Austria. \textit{Email address:}
  \href{davide.ravotti@gmail.com}{davide.ravotti@gmail.com}}}

\def\xb{\mathbf x}
\def\yb{\mathbf y}
\def\ub{\mathbf u}
\def\eb{\mathbf e}
\def\ab{\mathbf a}

\def\gb{\mathbf g}
\def\hb{\mathbf h}
\def\gbtilde{\widetilde{\gb}}
\def\bel{E} 
\def\Ctau{C_{\alpha}} 
\def\Lip{\operatorname{Lip}} 
\DeclareMathOperator{\IC}{IC} 
\DeclareMathOperator{\FBR}{FBR} 
\DeclareMathOperator{\Good}{\mathcal G} 
\DeclareMathOperator{\conv}{Conv} 
\DeclareMathOperator{\Stab}{Stab} 
\DeclareMathOperator{\Sub}{SubLev}
\DeclareMathOperator{\graph}{Gr} 
\def\sublev{\Sub(\xb,\yb)}
\def\sublevi{\Sub(\xb_i,\yb_i)}
\def\sublevone{\Sub(\xb_1,\yb_1)}
\let\oldmarginpar\marginpar
\renewcommand{\marginpar}[1]{\oldmarginpar{\footnotesize #1}}

\def\eps{\varepsilon}
\renewcommand{\epsilon}{\varepsilon}
\renewcommand{\emptyset}{\varnothing}
\newcommand*{\diff}{\mathop{}\!\mathrm{d}}
\DeclareMathOperator{\Id}{Id}
\newcommand{\one}{{\rm 1\mskip-4mu l}}

\begin{document}
\date{July 18, 2024}
\maketitle{}

\begin{abstract}
  We study time-changes of unipotent flows on finite volume quotients of
  semisimple linear groups, generalising previous work by Ratner on time-changes
  of horocycle flows. 
  Any measurable isomorphism between time-changes of unipotent flows gives rise 
  to a non-trivial joining supported on its graph. 
  Under a spectral gap assumption on the groups, we show the
  following rigidity result: either the only limit point of this graph joining under 
  the action of a one-parameter renormalising subgroup is the trivial joining, 
  or the isomorphism is \lq\lq affine\rq\rq, namely it is obtained composing an 
  algebraic isomorphism with a (non-constant) translation along the centraliser.
\end{abstract}

\input{01_introduction}
\input{02_background}
\input{03_geometry}
\input{04_basic_lemma}
\input{05_normaliser}
\input{06_conclusion}
\input{07_sl2xG}

\appendix
\input{Chevalley.tex}

\bibliographystyle{amsalpha}
\bibliography{time_changes_unipotents.bib}

\end{document}

%% file: 01_introduction.tex
\section{Introduction}

\subsection{Parabolic and unipotent flows.}
Parabolic flows are dynamical systems characterised by a ``slow'' divergence of
nearby points, usually at a polynomial rate. They presents an intermediate
chaotic behaviour, in the sense that they have some properties, like mixing,
which are typical of highly chaotic systems, but also have zero entropy, a
feature of regular (i.e., non chaotic) systems. Fundamental examples of
(homogeneous) parabolic flows are unipotent flows on quotients of Lie groups and
nilflows on nilmanifolds. For more details and examples of parabolic systems we
refer to~\cite[Chapter 8]{HasselblattKatok:Principal} and the introduction
of~\cite{AFRU}.

Let $G$ be a connected semisimple Lie group, with Lie algebra $\mathfrak{g}$. We
recall that an element $U\in\mathfrak{g}\setminus\{0\}$ is \emph{unipotent} if
$\ad(U) = [U,\cdot]$ is a nilpotent linear operator on $\mathfrak{g}$. Given any
quotient $M=\Gamma \backslash G$ of $G$ by a discrete subgroup $\Gamma$, the
unipotent flow $\{\phi_U^t\}_{t \in \R}$ on $M$ generated by $U$ is defined by
$\phi_U^t(\Gamma\xb)=\Gamma\xb\exp(tU)$, with $\xb\in G$. One of the prime
examples of unipotent flow is the horocycle flow on quotients of $\SL_2(\R)$.

Unipotent flows on quotients of Lie groups have been heavily studied by many
authors. Their properties are useful in many number theoretic problems, as
Margulis' seminal proof of Oppenheim conjecture~\cite{Margulis:Oppenheim}.
Ratner's proof of Raghunathan
conjectures~\cite{Ratner:Raghunathan1,Ratner:Raghunathan2,Ratner:Raghunathan3}
at the beginning of the 1990s showed that unipotent flows are well-behaved in
many ways. In particular, their probability invariant measures are of algebraic
nature.

\subsection{Rigidity in unipotent dynamics.}

One remarkable feature of unipotent dynamics is the presence of several rigidity
phenomena. Roughly speaking, rigidity occurs when a weak form of equivalence
implies a stronger one. In the homogeneous setting, it translates to the
surprising fact that, under only measure-theoretic assumptions, one can, in some
cases, deduce strong, algebraic (in particular, smooth) conclusions. Several
results of this type were proved in the 1980s by Marina Ratner.
In~\cite{Ratner:RigidityHorocycles}, she showed that any mesurable isomorphism
$\psi$ between horocycle flows $\{h_i^t\}_{t \in \R}$ on quotients
$M_i = \Gamma_i \backslash \SL_2(\R)$ for $i=1,2$ of $\SL_2(\R)$ is in fact
algebraic; namely, there exist $\xb_0 \in \SL_2(\R)$ and $s_0 \in \R$ such that
$\xb_0 \Gamma_1 \xb_0^{-1} \subset \Gamma_2$ and
$\psi(\Gamma_1\xb) = h_2^{s_0}(\Gamma_2 \xb_0 \xb)$. This result was generalised
by Witte to unipotent flows in~\cite{Witte:Rigidity,Witte:Affine}, and can now
be seen as a corollary of the main results in~\cite{Ratner:Raghunathan2}.

A strengthening of this isomorphism result for horocycle flows was obtained by
providing a complete classification of all possible joinings between them. Let
us recall that a \emph{joining} between two probability preserving flows
$\phi_i^t \colon (X_i, \mu_i) \to (X_i, \mu_i)$ for $i=1,2$ is a probability
measure $\nu$ on the product $X_1 \times X_2$ which is invariant by the product
flow $\phi_1 \times \phi_2$ and projects onto the measures $\mu_i$ under the
canonical projections onto the factors $X_1$ and $X_2$. Clearly, the set of
joinings is not empty, since the product measure $\nu = \mu_1 \otimes \mu_2$ is
always a joining. In~\cite{Ratner:RigidityJoinings}, Ratner proved that all
joinings of horocycle flows are algebraic. More precisely, using the same
notation as above, she showed that, if $\nu$ is an ergodic joining of
$\{h_1^t\}_{t \in \R}$ and $\{h_2^t\}_{t \in \R}$, then either $\nu$ is the
product joining or the two lattices are commensurable in the sense that there
exist $\xb_0 \in \SL_2(\R)$ and $\Gamma_0$ such that
$\Gamma_0 = \Gamma_1 \cap \xb_0 \Gamma_2 \xb_0^{-1}$ and the product flow
$\{h_1^t \times h_2^t\}_{t \in \R}$ is isomorphic to the horocycle flow on
$\Gamma_0 \backslash \SL_2(\R)$. This joining result is stronger than the
aforementioned one, since any isomorphism between two flows defines a joining
supported on the graph, which we call the \emph{graph joining}.

The classification of the ergodic invariant measures for horocycle flows by Dani
and Smillie~\cite{DaniSmillie} can also be seen as a further rigidity
phenomenon. Ratner's proof of Raghunatan conjecture~\cite{Ratner:Raghunathan2}
extends this classification to any unipotent action on quotients
$\Gamma \backslash G$ of Lie groups $G$, for which ergodic invariant probability
measures are always algebraic; that is, supported on ``affine'' homogeneous
submanifolds (namely, translates of closed subgroups of $G$ intersecting
$\Gamma$ in a lattice).

Remarkably, some of these results extend beyond the homogeneous setting, as we
are going to explain.

\subsection{Time-changes of unipotent flows.}
Since there is not a precise definition of parabolic flows, it is not clear what
properties should be considered typical of the parabolic class, and which should
not. In order to gain some insight on this problem, it is natural to look for
more examples, especially for non-homogeneous ones. However, it turns out that
it is not easy to produce examples of non-homogeneous parabolic flows, since
perturbations usually break down the fragile parabolicity and one typically
obtains a hyperbolic flow. A simple class of parabolic flows consists of
(smooth) \emph{time-changes} of unipotent ones. In this case points move along
the same orbits of the original flow, only with a different speed, determined by
a smooth function. In particular, time-changes of ergodic flows remain ergodic
with respect to an equivalent invariant measure. However, finer dynamical
properties can change drastically. For instance, nilflows are never weak-mixing,
yet, given any ergodic nilflow, non-trivial time-changes within a class of
\lq\lq trigonometric polynomials\rq\rq\ on the nilmanifold are
mixing~\cite{AFRU,AFU,ForniKanigowski,Ravotti:nilflows}.  For Heisenberg
nilflows, a stronger dichotomy holds: either a time-change is trivial or is
weak-mixing \cite{ForKan2}.  In the case of the horocycle flow, building on work
by Kuschnirenko~\cite{Kuschnirenko}, Marcus proved in~\cite{Marcus:Mixing} that
smooth time-changes are mixing of all orders. The rate of mixing was studied by
Forni and Ulcigrai in~\cite{ForniUlcigrai}. In the same paper, it was shown that
the spectrum is Lebesgue. Independently, at the same time, absolute continuity
of the spectrum was proven by Tiedra de Aldecoa in~\cite{Tiedra}.

Much less is known for time-changes of general unipotent flows. Simonelli showed
that the spectrum is absolutely continuous for unipotent flows on semisimple Lie
groups~\cite{Simonelli}, and the polynomial rate of mixing was studied by the
third-named author in~\cite{ravotti2020polynomial}.

Finally, let us mention that a new kind of parabolic perturbation of unipotent
flows on (compact) quotients of $\SL_n(\R)$, \emph{not} given by a time-change,
was defined and studied by the third-named author
in~\cite{Ravotti:Perturbations}.

\subsection{Rigidity of parabolic perturbations of unipotent flows.}

A natural question is to ask when a smooth parabolic perturbation of a unipotent
flow is measurably isomorphic to the homogeneous unperturbed flow. In the case
of time-changes, if the time-change function is measurably cohomologous to a
constant, it is easy to see that the two flows are isomorphic and the regularity
of the isomorphism is given by the regularity of the transfer function. Ratner
proved a rigidity result for time-changes of horocycle flows~\cite{Ratner:Acta}
which, in particular, implies the converse statement. She showed that any
measurable isomorphism $\psi$ between two time-changes
$\{\widetilde h_i^t\}_{t \in \R}$ of horocycle flows on quotients $M_i$ is in
fact algebraic in the sense that there exists $\xb_0 \in \SL_2(\R)$ and a
measurable function $\sigma \colon M_2 \to \R$ such that
$\xb_0 \Gamma_1 \xb_0^{-1} \subset \Gamma_2$ and
$\psi(\Gamma_1\xb) = h_2^{\sigma(\Gamma_2 \xb_0 \xb)}(\Gamma_2 \xb_0 \xb)$.
Hence, the problem of establishing the triviality of time-changes is equivalent
to solving the cohomological equation for the horocycle flow. By the work of the
second author and Forni~\cite{FlaminioForni:Cohomological}, measurably trivial
time-changes are \emph{rare}; namely, they form a closed subspace of countable
codimension.

A similar statement holds for time-changes of Heisenberg nilflows. Avila, Forni
and Ulcigrai proved that a function is a measurable coboundary if and only if it
is a smooth coboundary~\cite{AFU}. Therefore, they were able to provide explicit
examples of non-trivial (as a matter of fact, mixing) time-changes.

A weaker form of rigidity than the one discussed above has been showed to hold
for the perturbations constructed in~\cite{Ravotti:Perturbations}. As for
time-changes, whenever a certain cocycle associated to the flow is a coboundary,
it is easy to see that the perturbation is trivial. The third author showed that
assuming that the perturbation is smoothly trivial implies the former statement.
Therefore, analogously as for time-changes, the (smooth) triviality of the
perturbation is equivalent to a cohomological statement. It would be interesting
to extend this result to the measurable setting.

Coming back to time-changes of horocycle flows, their factors and joinings are
also algebraic, by a further work of Ratner~\cite{Ratner:Reparametrizations}.
Interestingly, non-trivial horocycle time-changes have disjoint rescalings, as
shown by Kanigowski, Lemanczyk and Ulcigrai~\cite{KLU}, albeit this property is
clearly false for the horocycle flow~itself.

A partial generalisation of the works of Ratner to the Lorentz group has been
recently achieved by Tang~\cite{Tang:timechanges}. He shows that the existence
of a measurable isomorphism between a unipotent flow in $SO(n,1)$ and a
time-change of itself generated by a smooth function $\tau$ implies that $\tau$
and the composition of $\tau$ with any element in the centraliser of the
unipotent flow are cohomologous. He then deduces a full analougue of Ratner's
result under the additional assumption of the transfer function being in $L^1$.
This seems, however, difficult to check in most concrete cases. Under similar
assumptions, in~\cite{Tang:joining}, he investigates the extension of Ratner's
joining result for time-changes of unipotent flows in $SO(n,1)$.

\subsection{Statement of results.}

In this paper, we generalise the aforementioned works of Ratner and Tang to
general semisimple linear groups satisfying a natural assumption on the spectral
gap. Given any two isomorphic time-changes of unipotent flows on any
finite-volume quotients, we show that the isomorphism is of a special algebraic form.

In order to state our result precisely, let us introduce some notation. Let
$G_1$ and $G_2$ be two connected, semisimple, linear groups. Consider two
lattices $\Gamma_1< G_1$ and $\Gamma_2 < G_2$, and consider the quotient
manifolds $M_i = \Gamma_i\backslash G_i$ with the probability measure $\mu_i$
inherited from the Haar measure on $G_i$, for $i=1,2$. Let
$\{\ub^t_i\}_{t\in\R}$ be two one-parameter unipotent subgroups of $G_i$. These
define a unipotent flow $\phi_{U_i}^t=\phi_i^t$ on $M_i$ as before. Given two
positive measurable functions $\alpha_i \colon M_i\to\R_{>0}$, we consider the
time-changes of the unipotent flows $\phi_i^t$ denoted $\widetilde{\phi}_i^t$,
where for simplicity we omit the dependence on $\alpha_i$.

The Jacobson-Morozov Theorem (see, e.g.,~\cite[Theorem~10.3]{knapp:beyond}),
ensure the existence of a $\sl_2(\R)$-sub-algebra
$\mathfrak{s}_i=\langle U_i, A_i, \overline{U_i}\rangle$ inside of
$\mathfrak{g}_i$. Exploiting this, we let $\ab_i^t=\exp(tA_i)$ be the Cartan
one-parameter subgroup which renormalises the unipotent flow defined by
$\ub_i^t$, meaning that $\ab_i^t\ub_i^s = \ub_i^{se^t}\ab_i^t$, for all
$s,t\in\R$. Denote by
$\phi_{A_1\times A_2}^t\colon M_1 \times M_2 \to M_1 \times M_2$ the diagonal
action $\phi_{A_1\times A_2}^t(\xb_1, \xb_2) = (\xb_1 \ab_1^t, \xb_2 \ab_2^t)$.

Having introduced the necessary notation, we are now ready to state our main
result. The precise assumptions about the time-change functions $\alpha_i$ are
described in \S\ref{sec:preliminaries}, and the definition of \emph{good
  time-change} is given in~\cref{def:good_timechange}. The \emph{strong spectral
  gap assumption} is also explained in \S\ref{sec:preliminaries}. For the
definition of graph joining $\mu_\psi$, we refer to \S\ref{sec:conclusion}.

\begin{bigtheorem}\label{thm:main} 
  Let $G_i$ be connected, semisimple, linear groups and $\Gamma_i$ be
  irreducible lattices in $G_i$. Assume that the manifolds
  $M_i=\Gamma_i\backslash G_i$ satisfy the strong spectral gap condition. Let
  $\phi_i^t$ be the unipotent flows on $M_i$ given by $\ub_i^t$, $i=1,2$. Let
  $\widetilde{\phi}_i^t$ be the time-changes of the unipotent flows $\phi_i^t$
  obtained by good time-changes $\alpha_i$. Assume that there is a measurable
  conjugacy $\psi\colon M_1\to M_2$ between $\widetilde{\phi}_1^t$ and
  $\widetilde{\phi}_2^t$, and let $\mu_\psi$ be the graph joining defined by
  $\psi$.
  
  Suppose that $(\phi_{A_1\times A_2}^t)_*\mu_\psi$ does \emph{not} converge to
  the trivial joining $\mu_1\otimes\mu_2$ as $t\to +\infty$. Then, there exist
  $\overline{\gb}\in G_2$, an isomomorphism $\varpi\colon G_1\to G_2 $, such
  that $\varpi(\ub_1)=\ub_2$,
  $\varpi(\Gamma_1)\subset\overline{\gb}^{-1}\Gamma_2\overline{\gb}$, and, up to
  passing to a finite quotient, we have that
  \[
    \psi(\Gamma_1 \xb) = \Gamma_2\overline{\gb}\varpi(\xb) \mathbf{c}(\Gamma_1
    \xb)\ub_2^{t(\Gamma_1 \xb)},
  \]
  for $\mu_1$-a.e.\ $\Gamma_1 \xb\in M_1$, where $\mathbf{c}(\Gamma_1 \xb)$
  commutes with $\ub_2$, $t(\Gamma_1 \xb)\in\R$, and both depend measurably on
  $\Gamma_1 \xb$.
\end{bigtheorem}

If we restrict to a special class of groups and a specific kind of unipotent
element, we obtain a rigidity result without any assumption on the graph
joining. In particular, this allows us to recover Ratner's original result on
time-changes of the horocycle flow~\cite{Ratner:Acta}.

\begin{bigtheorem}\label{thm:sl2xG} 
  Let $G_i'$ be connected, semisimple, linear groups with finite centre and
  without compact factors. Consider the groups $G_i = \SL_2(\R) \times G_i'$,
  for $i=1,2$ and let $\Gamma_i$ be irreducible lattices in $G_i$. Let $\phi_i^t$
  be the unipotent flows on $M_i$ given by $\ub_i^t = \left(\begin{smallmatrix}
  1 & t \\ 0 & 1 \end{smallmatrix}\right) \times \eb_i$, where $\eb_i\in G_i'$
  is the identity.  Let $\widetilde{\phi}_i^t$ be the time-changes of the
  unipotent flows $\phi_i^t$ obtained by good time-changes $\alpha_i$. Assume
  that there is a measurable conjugacy $\psi\colon M_1\to M_2$ between
  $\widetilde{\phi}_1^t$ and $\widetilde{\phi}_2^t$.
  
  Then, there exist $\overline{\gb}\in G_2$,an isomomorphism
  $\varpi\colon G_1\to G_2 $, such that $\varpi(\ub_1)=\ub_2$,
  $\varpi(\Gamma_1)\subset\overline{\gb}^{-1}\Gamma_2\overline{\gb}$, and we
  have that
  \[
    \psi(\Gamma_1 \xb) = \Gamma_2\overline{\gb}\varpi(\xb) \mathbf{c}(\Gamma_1
    \xb)\ub_2^{t(\Gamma_1 \xb)},
  \]
  for $\mu_1$-a.e.\ $\Gamma_1 \xb\in M_1$, where $\mathbf{c}(\Gamma_1 \xb)$
  commutes with $\ub_2$, $t(\Gamma_1 \xb)\in\R$, and both depend measurably on
  $\Gamma_1 \xb$.
\end{bigtheorem}

Finally, as a consequence of our main technical result, we obtain the following
cohomological statement, which generalises~\cite[Theorem~1.1]{Tang:timechanges}.

\begin{bigtheorem}\label{thm:cohomologycentraliser} Let $G_1$ and $G_2$ be two
  connected, semisimple, linear groups. Let $\Gamma_i$ be irreducible lattices
  in $G_i$ and assume that the manifolds $M_i = \Gamma_i\backslash G_i$ satisfy
  the strong spectral gap assumption, for $i=1,2$. Let $\widetilde{\phi}_1^t$ be
  the time-change of the unipotent flow obtained from $\phi_1^t$ by a good
  time-change $\alpha$. Assume that there is a measurable conjugacy
  $\psi\colon M_1\to M_2$ between $\widetilde{\phi}_1^t$ and the unperturbed
  unipotent flow $\phi_2^t$. Then $\alpha(x)$ and $\alpha(x\mathbf{c})$ are
  \emph{measurably cohomologous} for all $\mathbf{c}$ in the centraliser of
  $\ub_1$.
\end{bigtheorem}

For the precise formulation and the proof of the above result, we refer the
reader to~\cref{cor:cohomologyofalphas}.

\begin{notabene}
After this article was finished, Lindenstrauss and Wei announced a similar and
stronger rigidity result for unipotent flows~\cite{LindenstraussWei}. 
Notably, they are able to prove that any measurable isomorphism $\psi$ 
of time-changes is in fact cohomologous to an algebraic isomorphism 
(as in our \Cref{thm:main}, but without the assumption on the graph joining). 
Furthermore, their result does not require any smoothness assumption 
on the time-change function. 

In both our and Lindenstrauss and Wei's work, the fundamental tool is a version 
of Ratner's Basic Lemma, which controls the relative distance of the images 
of points under $\psi$ in terms of the time their orbits 
stay close. In our proof, we employ geometric arguments which are closer to 
Ratner's work, whereas Lindenstrauss and Wei' approach is based on 
the study of Kakutani-Bowen balls under Kakutani equivalence. 

The main difference we can see with the arguments outlined in \cite{LindenstraussWei} 
is that they manage to ensure the convergence 
of the conjugation of $\psi$ by the subgroups generated by $A_1$ and $A_2$ to 
a well-defined measurable map (which will be an isomorphism of the homogeneous 
flows). This crucial step relies on a $\SL_2(\R)$ ergodic theorem 
\cite[\S6]{LindenstraussWei}.
\end{notabene}

\subsection*{Outline of the paper.}
The paper is organized as follows.

We begin in \S\ref{sec:preliminaries} by giving all the precise definitions of
the objects we deal with. In particular, we explain our assumptions on both the
time-changes and the groups we consider.

In \S\ref{sec:geometryU}, using the Lie group structure, we study the geometry
of the unipotent flow, defining certain polynomials which measure the divergence
of the orbit in the $\sl_2(\R)$-sub-algebra given by $U_1$ and in the remaining
part of the Lie algebra. In this section we explain carefully the construction
of the blocks, which will be crucial later on. We study the blocks exploiting
the polynomial nature of the unipotent flow.

Section \ref{sec:basiclemma} is dedicated to the proof of the key technical
result we use: Ratner's Basic Lemma (\cref{lemma:basic}). This result is applied
in \S\ref{sec:centraliser} to show that the isomorphism maps leaves tangent to
the normaliser of $U_1$ to leaves tangent to the normaliser of $U_2$.  At end of
this section, we prove a more general version of
\cref{thm:cohomologycentraliser}.

In \S\ref{sec:conclusion} we define the graph joining and exploit Ratner's
classification of joinings of unipotent flows in~\cite{Ratner:Raghunathan2} to
prove \cref{thm:main}. Finally, \cref{thm:sl2xG} is proven in
\S\ref{sec:generalisingRatner} by adapting Ratner's original strategy
from~\cite{Ratner:Acta}. The \cref{sec:appendix_Chevalley}, contains, for
completeness, a proof of a consequence of Chevalley's Lemma.

\subsection*{Acknowledgments}
M. A. was supported by the \textit{Fondo de Arranque} project ``Rigidez en unos
flujos parab\'olicos'' of the Universidad del Rosario.

We thank Giovanni Forni for many comments on a previous version of this article,
and Nimish Shah for useful discussions around~\cite{Ratner:Raghunathan2}.


%% file: 02_background.tex

\section{Preliminaries}\label{sec:preliminaries}

In this section we introduce the class of homogeneous manifolds and times
changes of unipotent flow that we shall consider and we highlight their
properties. We also introduce two technical conditions to take into account the
possibility that the manifold $M_2$ is not compact.

\subsection{Time-changes and spectral condition.}\label{sec:time-changes}

In general will have groups~$G_1$ and~$G_2$, manifolds~$M_1$ and~$M_2$, flows
$(\phi_1^t)$ and~$(\phi_2^t)$, \dots. In this section we deal with properties
common to these objects.  Thus, in order to lighten the notation, we shall drop
the indices $1$ and $2$, so that~$G$, $M$, $(\phi^t)$ \dots will refer to both
groups, manifolds, flows, \dots.

\subsubsection{}\label{sec:sl2_triple}
Recall from the introduction that~$G$ is a \emph{connected semisimple linear
  group} and that the manifold~$M= \Gamma\backslash G$ is obtained as a quotient
of this group by a lattice~$\Gamma$. The manifold $M$ is endowed with a
probability measure~$\mu$ locally defined by the Haar measures of~$G$.  A point
in the group~$G$ will be written in boldface characters:~$\mathbf x \in G$.

We also recall that~$\{\ub^t\}_{t \in \R}$ is a one-parameter unipotent subgroup
of~$G$, generated by an element~$U\in \mathfrak g=\Lie(G)$. The
subgroup~$\{\ub^t\}_{t \in \R}$ defines a measure-preserving smooth flow on~$M$
by letting
\[
  \phi^t(x) = x \ub^t,\qquad (x =\Gamma\gb\in M).
\]
We identify the vector field on~$M$ generating the flow~$(\phi^t)$ with~$U$. In
fact, elements of the Lie algebra~$\mathfrak g$ can be considered as left
invariant vector fields on the group~$G$ and thus they project to the quotient
space~$M$.

We fix, once and for all,
a~$\sl_2(\R)$-sub-algebra~$\mathfrak s=\langle U,A,\bar U\rangle$
containing~$U$, that is a triple satisfying the commutation relations
\begin{equation}\label{eq:blocks:7}
  [A, U] = U, \quad [A, \bar U] = -\bar U,
  \quad [U, \bar U]=2A.
\end{equation}
The existence of such a sub-algebra~$\mathfrak s$ is ensured by the
Jacobson-Morozov Theorem (see, e.g., \cite[Theorem 10.3]{knapp:beyond}).  We
denote $\ab^t =\exp(tA)$ and $\phi_A^t(x) = x\ab^t$ the induced flow on $M$.

\begin{definition}
  Let~$\alpha \colon M \to \R_{>0}$ be a strictly positive measurable function
  on~$M$. The \emph{time-change of the unipotent flow~$(\phi^t)$ with
    generator~$U$ determined by $\alpha$} is the flow~$(\widetilde \phi^t)$
  on~$M$ generated by the vector field~$\alpha^{-1} U$.
\end{definition}

The definition above implies that the time-changed flow~$(\widetilde \phi^t)$
preserves the measures~$\alpha \mu$.

For~$x\in M$, let~$w(x,t)$ be defined by the equality
\begin{equation}\label{eq:01_background:1}
  t = \int_0^{w(x,t)} \alpha (x \ub^r) \D r.
\end{equation}
Then
\begin{equation}\label{eq:01_background:2a}
  \widetilde \phi^t(x) = x \ub^{w(x,t)}.
\end{equation}

It is easily verified by its definition that the function~$w$ is a
\emph{cocycle} for the time-changed flow $(\widetilde \phi^t)$, i.e., it
satisfies the \emph{cocycle equation}
\begin{equation}\label{eq:01_background:4}
  w(x,s+t)= w(x,s) +  w\big(\widetilde \phi^s(x),t\big).
\end{equation}
We also define~$\xi(x, t)$ to be the inverse function of~$w(x,t)$ with respect
to the second variable, so that
\[
  t=w(x, \xi(x,t)),\qquad \forall (x,t)\in M \times \R.
\]
The identities~\eqref{eq:01_background:1} and~\eqref{eq:01_background:2a} may be
rewritten as
\[
  \xi(x, t) = \int_0^{t} \alpha (x \ub^r) \D r\quad \text{and}\quad {\widetilde
    \phi}^{\xi(x, t)} (x) = x \ub^t.
\]
By symmetry, the function~$\xi$ is a cocycle for the~$U$-flow $(\phi^t)$.

\subsubsection{}
A measurable cocycle $w(x,t)$ for the time-changed flow $(\widetilde \phi^t)$ is
a measurable \emph{coboundary} if there exists a function $f\colon M\to\R$,
called the \emph{transfer function} such that
\[
  w(x,t) = f\circ\widetilde{\phi}^t(x) - f(x),
\]
for $\mu$-a.e.\ $x\in M$ and all $t\in\R$. Two cocycles are measurably
\emph{cohomologous} if their difference is a measurable coboundary.

Cohomologous cocycles yield isomorphic time-changes, and the isomorphism is
exactly flowing along the orbit of the time-changed flow for time $f(x)$.
See~\cite[\S9]{Katok:combinatorial} or~\cite[\S 2.1]{AFRU} for more details.

\subsubsection{}
Next we state some of the assumptions on the time-change function $\alpha$
indicating, for justification, their consequences later needed in the proofs.

\begin{assumption}\label{ass:1}
  We will assume that~$\alpha$ is \emph{uniformly bounded}, and we let
  \[
    \Ctau = \max \{ \alpha(x), \alpha^{-1}(x) : x \in M\} >1.
  \]
\end{assumption}

This hypothesis easily implies the following \emph{rough bound} on the
cocycle~$w$ and its inverse:
\begin{equation}\label{eq:01_background:3}
  \Ctau^{-1} \leq \Lip_t [w( \cdot ,t)] \leq \Ctau, \quad\text{and}\quad
  \Ctau^{-1} \leq \Lip_t[\xi( \cdot ,t)] \leq \Ctau,
\end{equation}
where the symbol~$\Lip_t$ stands for the Lipschitz constant of a function with
respect to the variable denoted by~$t$.

Another immediate consequence of the Assumption~\ref{ass:1}, is that the
measure~$\widetilde \mu :=\alpha\mu$ which is preserved by the time-changed
flows~$(\widetilde \phi^t)$ is equivalent to the measures~$\mu$:
\begin{equation}\label{eq:01_background:5}
  \Ctau^{-1} \mu \le \widetilde  \mu \le \Ctau \mu.
\end{equation}

\begin{assumption}
  \label{ass:mormalisation}
  The time-change function~$\alpha$ has average~$1$ with respect to the Haar
  measure~$\mu$.
\end{assumption}

This assumption is not restrictive, as we can always rescale the function by its
average and compose the isomorphism $\psi$ with $\phi_A^t$ for an appropriate
amount.

\subsubsection{} The rough bound~\eqref{eq:01_background:3} is insufficient for
good estimates of the cocycles~$w(\cdot,t)$ and~$\xi(\cdot,t)$ for large values
of~$t$. For precise estimates in this range we need an effective version of the
ergodic theorem, which will be derived from a quantitative mixing result.  The
study of quantitative mixing results is the same of the study of decay of matrix
coefficients of the regular representation of $G$ on $L^2 (M,\mu)$. This goes
back to the work of
Harish-Chandra~\cite{Harish-Chandra:spherical1,Harish-Chandra:spherical2} and
Warner~\cite{WarnerG:Harmonic_analysis1,WarnerG:Harmonic_analysis2}. For more
details we refer the reader to the introduction of~\cite{BjEiGo:multiple_mixing}
and the references therein.

We say that~$M$ satisfies the \emph{strong spectral gap assumption} if the
restriction of the regular representation of~$G$ on~$ L^2(M,\mu)$ to every
noncompact simple factor of $G$ is isolated, in the Fell topology, from the
trivial representation.

\begin{assumption}\label{ass:2}
  The manifolds~$M_1$ and~$M_2$ satisfy the strong spectral gap assumption.
\end{assumption}

This assumption holds in a number of cases, for example if $G$ has the Kazhdan
property (T)~\cite{Nevo:Spectral}, or if $G$ is a semisimple group with finite
centre and no compact factors and $\Gamma$ is
irreducible~\cite[p.~285]{KeSa:spectral_gaps} and~\cite{KleinbockMargulis}.  We
will use the following quantitative mixing result, stated
in~\cite{BjEiGo:multiple_mixing}.

\begin{theorem}\label{thm:spectral_gap}
  Assume that~$M$ satisfies the strong spectral gap assumption.  There exist
  constants~$C_{M}, \widetilde \eta >0$ and a Sobolev norm~$\mathcal{S}$ for
  functions on~$M$ such that for all sufficiently smooth~$f,g\colon M \to \C$ of
  average zero and for all~$t \geq 1$ we have
  \[
    \abs{ \langle f\circ \phi^t, g \rangle } \leq C_{M} \mathcal{S}(f)
    \mathcal{S}(g) t^{- \widetilde \eta}.
  \]
\end{theorem}

Without loss of generality we may suppose that
$ \sup_M\abs{f} \le \mathcal{S}(f)$ for any $f\colon M\to \C$ with
$\mathcal{S}(f)<\infty$.

\begin{assumption}\label{ass:3}
  The time-change function~$\alpha$ has finite Sobolev
  norm~$\mathcal S(\alpha)$.
\end{assumption}

\begin{definition}[Good time-change]\label{def:good_timechange}
  A time-change function $\alpha\colon M\to\R_{>0}$ is \emph{good} if it
  satisfies Assumption~\ref{ass:1},\ref{ass:mormalisation}, and~\ref{ass:3}.
\end{definition}

\Cref{ass:3} is needed to ensure we can apply \Cref{thm:spectral_gap} to
$\alpha$. In the case of $\SL_2(\R)$, the polynomial decay in 
\Cref{thm:spectral_gap} holds under weaker assumptions (namely, it 
is sufficient to require H\"{o}lder regularity along the rotation subgroup). 
In this sense, our assumptions on the time-change function are a natural 
generalization of Ratner's ones in \cite{Ratner:Acta}.

Under the assumptions of this quantitative mixing result, we deduce a polynomial
estimate on the ergodic averages of~$(\phi^t)$ on a set of large measure.

\begin{lemma}\label{thm:mixing_bound}	
  There exists constants~$\eta' \in (0,1)$ and $C_{M}'$, depending only on~$M$,
  such that the following holds.  Let $f\colon M \to \C$ be a function with zero
  average and finite Sobolev norm~$\cS (f)$.  For all~$\omega >0$ there exists a
  set $Y=Y(\omega, f) \subset M$ of measure~$\mu(Y) \geq 1-\omega$ and a number
  $m'= m'(\omega) \geq 1$ such that for all~$x \in Y$ and all~$t \geq m'$ we
  have
  \[
    \abs[\bigg]{\int_0^t f(x \ub^s)\, \D s}\le C_{M}' \cS (f) t^{1-\eta'}.
  \]
\end{lemma}

\begin{proof}
  Let us denote
  \[
    s_t(x)=\int_0^t f(x\ub^s) \, \D s.
  \]
  Using \cref{thm:spectral_gap}, we have
  \[
    \begin{split}
      \norm{s_t}_2^2 &= \int_M \left( \int_0^t f(x\ub^s)\, \D s\right)^2 \, \D
      \mu = \int_M \int_0^t\int_0^t f(x\ub^r) f(x\ub^s)\,
      \D r \D s \, \D\mu\\
      &= \int_0^t\int_0^t \langle f(x),f(x\ub^{s-r})\rangle \, \D r \D s
      \le \int_0^t\int_{-s}^{t-s} \abs{\langle f(x),f(x\ub^r\rangle}\, \D r \D s \\
      &\le 2t \int_0^t\abs{\langle f(x),f(x\ub^r\rangle}\, \D r \le
      \frac{2C_M\,\mathcal{S} (f)^2}{1-\widetilde \eta} t^{2-\tilde \eta} = C_1
      \,\mathcal{S} (f)^2 \,t^{2-\tilde \eta}.
    \end{split}
  \]
  By Chebyshev's Inequality, we have that
  \[
    \mu\{ x\in M: \abs{s_t(x)}\ge \mathcal{S} (f) \,t^{1-\frac{\tilde
        \eta}{4}}\}\le C_1\, t^{-\frac{\tilde \eta}{2}}.
  \]
  For~$t>0$, we consider the
  sets~$Y_t = \{ x\in M: \abs{s_t(x)}<\mathcal{S} (f) t^{1-\tilde \eta/ 4} \}$.
  Choosing~$t_n=n^{4/\tilde \eta}$, we have that
  \[
    \mu(Y_{t_n})\ge 1- \frac{C_{1} }{n^2}.
  \]
  Given~$\omega>0$ we define~$\bar{k}$ so
  that~$C_{1} \sum_{k\ge \bar{k}} \frac{1}{k^2}<\omega$.  Hence
  \[
    \mu(Y)> 1-\omega, \qquad \text{where} \qquad Y=\bigcap_{k\ge \bar{k}}
    Y_{t_k}.
  \]
  Any point~$x\in Y$
  satisfies~$\abs{s_{t_k}(x)}<\mathcal{S} (f)t_k^{1-\tilde \eta/4}$ for
  all~$k\ge\bar k$.  We need to prove a similar estimate for the remaining
  times~$t\ge t_{\bar k}$.

  Let~$k\ge \bar k$ such that~$t_k<t\le t_{k+1}$. There exists a
  constant~$C_2=C_2(\tilde\eta)>0$ such
  that~$t_{k+1}-t_k=C_2t_k^{1-\tilde \eta/4}$.  Writing~$t=t_k+q$,
  with~$0<q\le C_{2} t_k^{1-\tilde \eta/4}$, we obtain, for $x\in Y$,
  \[
    \abs{s_t(x)}\le \abs{s_{t_k}(x)} + \abs[\bigg]{\int_{t_k}^{t_k+q}
      f(x\ub^s)\, \D s} \le C_{M}'\,\mathcal{S} (f)\, t_k^{1-\frac{\tilde
        \eta}{4}}.
  \]
  with~$C_{M}'=C_{1}+C_{2}$. Setting $m'=t_{\bar k}$
  and~$\eta' = \widetilde \eta /4$, we have completed the proof.
\end{proof}

\begin{corollary}\label{thm:mixing_condition}
  There exists~$\eta \in (0,1/4)$ such that for all~$\omega >0$ there exist a
  set $Y \subset M$ of measure~$\mu(Y) \geq 1-\omega$ and a constant $m \geq 1$
  such that for all~$x \in Y$ and all~$t \geq m$,
  \[
    \abs{\xi(x,t) - t} \leq \frac{1}{\Ctau^4} t^{1-\eta}, \quad \abs{w(x,t) - t}
    \leq \frac{1}{\Ctau^4} t^{1-\eta}.
  \]
\end{corollary}

\begin{proof}
  We begin with the estimate for~$\xi$.  We have, for all $t\ge m_0$,
  \[
    \abs{\xi(x, t)-t} = \abs[\bigg]{\int_0^t (\alpha(x\ub^r) -1) \, \D r} \le
    C'_{M} \cS(\alpha) t^{1- \eta'},
  \]
  where we applied \cref{thm:mixing_bound} to the zero mean function
  $\alpha-1$. In order to conclude, it is enough to fix a smaller
  $\eta < \eta' \leq 1/4$ and possibly a larger~$m_0$.

  For the second estimate, using the bound on the Lipschitz constants of~$w$
  seen as functions of time, we have
  \[
    \abs{w(x,t)-t}=\abs{w(x,t)-w(x, \xi(x,t))} \leq \Ctau \abs{\xi(x, t)-t},
  \]
  and the estimate follows, possibly by again reducing $\eta$.
\end{proof}

\subsection{General setting.} We now return to the general setting of
\cref{thm:main}. Thus $G_1$ and $G_2$ are two groups as in
\S\ref{sec:sl2_triple} and all objects and constants defined in
\S\ref{sec:sl2_triple} have corresponding suffixes. We let
$\Ctau=\max\{C_{\alpha_1}, C_{\alpha_2}\}$, $C_M=\max\{C_{M_1}, C_{M_2}\}$,
$\eta= \min\{\eta_1,\eta_2\}$, $m=\max\{m_1, m_2\}$ so that all inequalities
stated in \S\ref{sec:time-changes} holds in $M_1$ and in $M_2$. Thus we now have

\medskip \noindent \textit{The map~$\psi \colon M_1 \to M_2$ is a measurable
  conjugacy between the time-changes~$\widetilde \phi_1$ and~$\widetilde \phi_2$
  defined by the good time-change functions~$\alpha_1$ and~$\alpha_2$. We
  suppose that~$\psi$ maps the measure~$\alpha_1\mu_1$ to the
  measure~$\alpha_2\mu_2$.}

\medskip

From the inequality \eqref{eq:01_background:5} we immediately have the following
bounds.
\begin{lemma}
  \label{lem:01_background:1}
  For any measurable set~$A\subset M_2$ we have
  \begin{equation}\label{eq:01_background:6}
    \Ctau^{-1} \widetilde \mu_2(A) \le \mu_1 (\psi^{-1} A ) \le \Ctau \widetilde \mu_2(A).
  \end{equation}
\end{lemma}

By definition of~$\psi$, for almost all~$x \in M_1$,
\[
  \psi(x \ub_1^t) = \psi ({\widetilde \phi_1}^{\xi_1(x, t)} (x)) = {\widetilde
    \phi_2}^{\xi_1(x, t)} (\psi (x)) = \psi(x) \ub_2^{w_2(\psi(x),\xi_1(x, t))}.
\]
We define
\[
  z(x,t) := w_2(\psi(x),\xi_1(x, t)),
\]
so that
\[
  \psi(x \ub_1^t) = \psi(x) \ub_2^{z(x,t) }, \quad \text{and} \quad\Ctau^{-2}
  \leq \Lip_t[z( \cdot ,t) ] \leq \Ctau^2.
\]
\begin{remark}
  The identity above exhibits~$\psi$ as a measurable isomorphism of the
  $U_1$-flow on $M_1$ and a reparametrization of the $U_2$-flow on
  $M_2$. Observe, however, that we cannot apply directly
  \cref{thm:mixing_condition} to the function $z(x,t)$ because, a priori, it is
  not sufficiently smooth.
\end{remark}

\begin{lemma}\label{lemma:z_cocycle}
  The map~$z$ is a \emph{cocycle over~$\{\ub_1^t\}_{t \in \R}$}, namely for
  almost all~$x \in M_1$ and all~$t,s \in \R$ we have
  \[
    z(x,t+s) = z(x,s) + z(x\ub_1^s,t).
  \]
  Moreover, if for some~$x \in M_1$,~$m >0$, and~$t'>t \geq 0$ we have
  \[
    z(x,t')-z(x,t) >m,
  \]
  then
  \[
    t'-t > m\Ctau^{-2}, \text{\ \ \ and\ \ \ } z(y,t')-z(y,t) > m\Ctau^{-4}
    \text{\ for all\ }y \in M_1.
  \]
\end{lemma}
\begin{proof}
  By definition of~$z$ we have
  \[
    \psi(x)\ub_2^{z(x,t+s)} = \psi(x \ub_1^{t+s}) = \psi(x\ub_1^s)
    \ub_2^{z(x\ub_1^s,t)} = \psi(x) \ub_2^{z(x,s)} \ub_2^{z(x\ub_1^s,t)},
  \]
  which implies the first claim.

  As for the second, by the cocycle relation and the Lipschitz bound on~$z$ we
  get
  \[
    m < z(x,t')-z(x,t) = z(x \ub_1^t, t'-t) \leq (t'-t)\Ctau^2,
  \]
  so that~$t'-t > m\Ctau^{-2}$. Then, for all other~$y \in M_1$,
  \[
    z(y,t')-z(y,t) = z(y\ub_1^t, t'-t) \geq \Ctau^{-2}(t'-t) > m\Ctau^{-4},
  \]
  and the proof is complete.
\end{proof}

\subsection{Two technical conditions.}
We now define two conditions which will be relevant later on: the Injectivity
Condition (IC) and the Frequently Bounded Radius Condition (FBR).

In general, we
 show that these conditions hold on a set of arbitrarily large measure.

\begin{definition}[The Injectivity Condition (IC)]
  \label{def:01_background:1}
  We say that a point~$x\in M_2$ satisfies the \emph{Injectivity
    Condition~$\IC(\rho,m)$} if, for any lift~$\xb\in G_2$ of~$x$ and
  any~$\yb\in B(\xb, \rho)$, we have
  \[
    d(\xb \ub_2^{t_1}, \gamma \yb \ub_2^{t_2}) > \rho \quad \text{for all} \quad
    \gamma \in \Gamma_2 \setminus \{\eb\} \quad \text{and all} \quad t_1,t_2 \in
    [-m,m].
  \]
\end{definition}

If the manifold $M_2$ is compact then, then there exists $\rho_0$ such that the condition $\IC(\rho,m)$ is satisfied for all $\rho < \rho_0$ and all $m>0$. Then in the following proposition, we may take $Y_{\IC}=M_2$ and $\zeta=0$.

\begin{proposition}\label{thm:IC}
  Let~$m >1$ and~$\zeta>0$.  There exist a compact
  set~$Y_{\IC}=Y_{\IC}(m,\zeta) \subset M_2$, with~$\mu_2(Y_{\IC}) > 1-\zeta$,
  and~$\rho = \rho(m, \zeta) \in (0,1)$ such that \emph{every}~$x\in Y_{\IC}$
  satisfies the Injectivity Condition~$\IC(\rho,m)$.
\end{proposition}

\begin{proof}
  Let~$K\subset M_2$ be a compact set of measure~$\mu_2(K) \geq 1-\zeta$.
  Let~$4r$ be the injectivity radius of~$K$. Thus, for any two points~$x\in K$
  and~$y\in M_2$ such that~$d(x,y)< r$ and for any lift~$\xb \in G_2$ of~$x$,
  there exists a unique lift~$\yb\in G_2$ of~$y$ such
  that~$d(\xb \ub^{t_1}_2, \yb \ub^{t_2}_2) < r$ for all~$t_1, t_2$
  with~$\abs{t_1}<r$ and~$\abs{t_2}<r$.

  We recall that, if~$\ab_2^s=\exp(sA_2)$, with~$A_2\in \mathfrak s_2$ as in
  \S\ref{sec:sl2_triple}, one has $\ab^s_2 \ub^t_2 \ab^{-s}_2 =
  \ub^{e^{s}t}_2$. Let~$s_0$ such that~$e^{s_0} r> m$ and
  set~$Y_{\IC} = K {\mathbf a}^{s_0}_2$.
  Clearly~$\mu_2( Y_{\IC}) = \mu_2(K) \geq 1-\zeta$.  Let~$L$ be the Lipschitz
  constant of the map~$\gb \mapsto \gb {\mathbf a}^{-s_0}_2$ on $G_2$.
  Set~$\rho = r/L$.

  Arguing by contradiction, let~$\xb \in \pi_2^{-1}(Y_{\IC})$
  with~$d(\xb, \yb) < \rho$ and suppose
  that~$d(\xb \ub_2^{t_1}, \gamma \yb \ub_2^{t_2}) \le \rho$ for
  some~$\gamma \in \Gamma_2\setminus \{\eb\}$ and some~$t_1,t_2 \in [-m,m]$.

  Then,~$d(\xb{\mathbf a}^{-s_0}_2 , \yb{\mathbf a}^{-s_0}_2)\le L\rho =r~$
  and~$d(\xb \ub_2^{t_1}{\mathbf a}^{-s_0}_2, \gamma \yb \ub_2^{t_2}{\mathbf
    a}^{-s_0}_2) \le L\rho = r$.  Set~$\xb':=\xb {\mathbf a}^{-s_0}_2$
  and~$\yb':=\yb {\mathbf a}^{-s_0}_2$.  On one hand we
  have~$d(\xb' , \yb')\le r~$ and
  \[
    d(\xb' \ub_2^{e^{-s_0}t_1}, \gamma \yb' \ub_2^{e^{-s_0}t_2})= d(\xb
    \ub_2^{t_1}{\mathbf a}^{-s_0}_2, \gamma \yb \ub_2^{t_2}{\mathbf
      a}^{-s_0}_2)\le r.
  \]
  On the other hand, since~$e^{-s_0}t_1, e^{-s_0}t_2\in [-r,r]$, we also have
  \[
    d(\xb' , \gamma \yb' ) \le d(\xb' , \xb'\ub^{e^{-s_0}t_1}_2 ) +
    d(\xb'\ub^{e^{-s_0}t_1}_2, \gamma\yb' \ub^{e^{-s_0}t_2}_2) + d( \gamma \yb'
    \ub^{e^{-s_0}t_2}_2 , \gamma \yb') < 3 r.
  \]
  This shows that both points~$\yb'$ and~$\gamma \yb'$ belong to the ball of
  radius~$3r$ centred in~$\xb'$. Since~$\xb'\in \pi_2^{-1}(K)$ and since the
  radius of injectivity of~$M_2$ at any point in~$K$ is strictly greater
  than~$3r$ we conclude that~$\gamma$ is the identity of~$G_2$, a contradiction.
\end{proof}

In the above proof we used the existence of the subgroup~$({\mathbf a}^s)$ to
rescale the orbits of the~$U$-flow. In fact the above proposition holds for any
volume preserving smooth flow on a finite volume manifold provided that the set
of periodic orbit has measure
zero. 

\subsubsection{}
We now introduce the second condition.

\begin{definition}[The Frequently Bounded Radius Condition ($\FBR$)]
  Given~$T_0 \geq 0$,~$r_0 >0$, and~$c \in [0,1]$, we say that a
  point~$x \in M_2$ satisfies the \emph{Frequently Bounded Radius
    Condition~$\FBR(T_0, c, r_0)$} if, for any lift~$\xb \in G_2$ of~$x$ and
  any~$T \geq T_0$, there exists~$t \in [c T, T]$ such that the injectivity
  radius at~$\xb \ab^t$ is bounded below by~$r_0$.
\end{definition}

If the space~$M_2$ is compact, then there is~$r_0 >0$ such that every point
$x \in M_2$ satisfies~$\FBR(T_0, c, r_0)$ for all~$T_0 \geq 0$ and
$c \in [0,1]$.  It is also clear that if a point satisfies~$\FBR(T_0, c, r_0)$,
then it also satisfies~$\FBR(T_0', c', r_0')$ for all~$T_0' \geq T_0$,
$0 \leq c' \leq c$, and~$0<r_0' \leq r_0$.

\begin{lemma}\label{lemma:FBR}
  For every~$c \in (0,1)$  there exists~$r_0 >0$ and $T_0>0$ such
  that almost every point~$x \in M_2$ satisfies~$\FBR(T_0, c, r_0)$.

  For every~$c \in (0,1)$ and every~$\omega>0$, there exists~$r_0 >0$
  and~$T_0 \geq 0$ such that the measure of the set of points~$x\in M_2$ which
  satisfy~$\FBR(T_0,c,r_0)$ is at least~$1-\omega$.
\end{lemma}

\begin{proof}
  It easy to see from the definition that if $x$ satisfies the condition
  $\FBR(T_0,c,r_0)$, then the point $x\ab^{-s}$ satisfies $\FBR(T_0+s,c,r_0)$,
  for all $s\ge 0$. Thus, for~$c\in(0,1)$ and~$r_0>0$ fixed, the set
  \[
    \bigcup_{T'=0}^\infty\{x\in M_2 : x \text{ satisfies } \FBR(T',c,r_0)\}
  \]
  is invariant under~$\ab^t$.  By ergodicity, if~$r_0$ is small enough, it has
  full measure.

  Let~$A_n = \{x\in M_2 : x \text{ satisfies } \FBR(n,c,r_0)\}$.  By the above,
  the~$A_n$ are an increasing sequence of sets whose union has full
  $\mu_2$-measure.  Hence, given~$\omega$, there exists an integer $T_0$ such that
  \[
    \mu_2 \Biggl( \bigcup_{n=0}^{T_0} A_n\Biggr) \ge 1-\omega. \qedhere
  \]
\end{proof}


%% file: 03_geometry.tex

\section{The geometry of the~$U$-flow}\label{sec:geometryU}

In this section, to ease the notation, we again drop the indices $1$ and $2$, so
that the group $G$ refers to both groups $G_1$ and $G_2$.

\subsection{Irreducible~$\sl_2(\R)$-modules and local transverse manifolds.}

\subsubsection{}\label{sec:basis_module}
We recall that we have fixed~$\sl_2(\R)$-sub-algebras in \S\ref{sec:sl2_triple}.

Considering~$G$ as the connected component of the identity of the real points of
an algebraic linear~$\R$-group~$\mathbf G$ we have that~$S=\exp \mathfrak s$ is
the connected component of the identity of a Zariski closed subgroup
$\mathbf S<\mathbf G$ with Lie algebra~$\mathfrak s$.

By Chevalley's lemma (\cite[7.9]{Borel:IntroGAri}, \cite[3.1.4]{Zimmer:Erg_th})
and the simplicity of~$\mathbf S$ there exists a linear representation~$\phi$
of~$\mathbf G$ on a finite vector space~$V$ and a vector~$v_0\in V$ such that
$\mathbf S= \Stab_{\mathbf G}(v_0)$.

The group~$S$ acts on~$\mathfrak g$ by the adjoint action.  By the simplicity of
$S$, the subspace~$\mathfrak s$ of~$\mathfrak g$ has an~$\Ad(S)$ invariant
complementary subspace~$\mathfrak m$. Thus there exists a neighbourhood
$\mathcal O\subset G$ of the identity such that any~$\gb \in \mathcal O$ may be
written as a product~$\gb = \gb_{\mathfrak m} \gb_{\mathfrak s}~$ of elements
$\gb_{\mathfrak m}\in \exp \mathfrak m$ and
$\gb_{\mathfrak s}\in \exp \mathfrak s$. The following lemma, whose proof is in
\cref{sec:appendix_Chevalley}, proves the uniqueness of such a decomposition.

\begin{lemma}\label{lemma:consequence_Chevalley}
  For every neighbourhoods~$\mathcal O_{\mathfrak m}$,
  $\mathcal O_{\mathfrak s}$ of the identity, respectively in~$\exp \mathfrak m$
  and in~$\exp \mathfrak s$, there exists a neighbourhood~$\mathcal O\subset G$
  of the identity such that, if~$\gb\in \mathcal O~$
  and~$\gb=\gb_{\mathfrak m} \gb_{\mathfrak s}~$,
  with~$\gb_{\mathfrak m}\in \exp \mathfrak m$
  and~$\gb_{\mathfrak s}\in \exp \mathfrak s~$,
  then~$\gb_{\mathfrak m}\in \mathcal O_{\mathfrak m}$
  and~$ \gb_{\mathfrak s}\in \mathcal O_{\mathfrak s}$.
\end{lemma}

\subsubsection{}
As a consequence of this lemma, if $\gb\in \mathcal O$ and
\begin{equation}\label{eq:01_background:7}
  \gb = \gb_{\mathfrak m} \gb_{\mathfrak s},\quad\text{with } \gb_{\mathfrak m}\in \exp \mathfrak m,\quad
  \gb_{\mathfrak s}\in \exp \mathfrak s,
\end{equation}
we may set
\[
  d ( \eb, \gb) = \max \{ d_{\mathfrak m} ( \eb, \gb_{\mathfrak m}),
  d_{\mathfrak s} ( \eb, \gb_{\mathfrak s})\},
\]
where $d_{\mathfrak m}$,~$d_{\mathfrak s}$ are distances on
$\exp \mathfrak m$ and~$\exp \mathfrak s$ --- (see~\eqref{eq:03_geometry:1}
and~\eqref{eq:01_background:def_dist_s} for their definitions). Thus there exists $\epsilon_0$ such that $ d ( \eb, \gb)< \epsilon_0$ implies $\gb \in \mathcal O$ and the decomposition~\eqref{eq:01_background:7} is uniquely determined.

\subsubsection{}\label{def_distance}
For any~$\gb \in G$ and any $\epsilon\in (0,\epsilon_0)$, let
\[
  \begin{split}
    W(\gb, \epsilon)= \Big\{ \gb \,\gb_{\mathfrak m} &\exp (a A)\exp (\bar
    u\overline U
    )\mid \\
    &\gb_{\mathfrak m}\in \exp(\mathfrak m), \ d_{\mathfrak
      m}(\eb,\gb_{\mathfrak m})\le \epsilon,\ \abs{a}\le \epsilon,\ \abs{\bar
      u}\le \epsilon\Big\}.
  \end{split}
\]
The set~$W(\gb, \epsilon)~$ is the \emph{local leaf through~$\gb$} transversal
to the $U$-flow.  We also define the \emph{(global) leaf through~$\gb$}
transversal to the $U$-flow by
\[
  W(\gb)= \{ \gb \,\gb_{\mathfrak m} \exp (a A)\exp (\bar u\overline U )\mid
  \gb_{\mathfrak m}\in \exp(\mathfrak m), \ a, \bar u \in \R\}.
\]
Given~$\xb, \yb\in G$ with~$d(\xb, \yb) \le \varepsilon$, we define the
\emph{continuous parametrization}~$q(s)$ of the~$U$-orbit of~$\yb$ by the
condition~$\yb \ub^{q(s)}\in W(\xb \ub^s)$.  The parametrization~$q(s)$ is
well-defined at least for all sufficiently small values of~$s\geq 0$ (see also
\cref{lemma:int_max_track} below).

\subsubsection{}\label{sec:geom_U_flow_1}
Let now~$\xb, \yb\in G$, $\epsilon\in (0,\epsilon_0)$, and assume that
\begin{equation}
  \label{eq:reparam}
  d(\xb, \yb) \le \varepsilon\quad \text{ and }
  \quad d( \xb\ub^s , \yb\ub^{t} ) \le \epsilon.
\end{equation} for some~$s>0$, and $t>0$.
This condition depends only on the relative position of~$\xb$ and~$\yb$: if we
Setting~$\gb = \xb^{-1}\yb$ we may rewrite \eqref{eq:reparam} as
\begin{equation}
  \label{eq:reparam2}
 d(\eb, \gb) \le \varepsilon\quad \text{ and }
  \quad d( \eb, \ub^{-s} \gb\ub^{t} ) \le \epsilon.
\end{equation}
or, with~$\gb = \gb_{\mathfrak m} \gb_{\mathfrak s}$ as in
\eqref{eq:01_background:7},
\begin{equation}
  \label{eq:track1}
 d_{\mathfrak m}( \eb, \ub^{-s} \gb_{\mathfrak m}\ub^{s} ) \le \epsilon
  \quad \text{and}\quad
  d_{\mathfrak s}( \eb, \ub^{-s} \gb_{\mathfrak s}\ub^{t} ) \le \epsilon.
\end{equation}
Thus the condition \eqref{eq:reparam} imposes strong restrictions on both $t(s)$
and on the relative position of~$\xb$ and~$\yb$.  If \eqref{eq:reparam} holds,
with an appropriate definition of the distance~$d_{\mathfrak s}$, there exists
$q(s)\in \R$ such that
\begin{equation}\label{eq:01_background:8}
  \abs{t-q(s)}\le \epsilon, \quad\text{ and } \quad\yb \ub^{q(s)}\in W(\xb \ub^s, \epsilon)
\end{equation}
With~$\xb^{-1}\yb=\gb_{\mathfrak m} \gb_{\mathfrak s}$ as in
\eqref{eq:01_background:7}, this condition may be rewritten as
\begin{equation}\label{eq:01_background:9}
  \ub^{-s}\gb_{\mathfrak s} \ub^{q(s)} =\exp (a(s) A)\exp (\bar u(s) \overline U ).
\end{equation}
Thus, the second inequality in formula~\eqref{eq:track1} is equivalent the the
requirement that \eqref{eq:01_background:8} holds true and that
$ \abs{a(s)}\le \epsilon/2$ and~$\abs{\bar{u}(s)}\le \epsilon/2$.

It follows that in order to exploit the consequences of the
condition~\eqref{eq:reparam} we can analyse separately the
terms~$\ub^{-s} \gb_{\mathfrak m}\ub^{s}$ and
$\ub^{-s} \gb_{\mathfrak s}\ub^{q(s)}$.
\subsubsection{}\label{sec:about_m}
We start by considering the term~$\ub^{-s} \gb_{\mathfrak m}\ub^{s}$.
Since~$ \ad (U)$ is a nilpotent endomorphism of~$\mathfrak m$, there exists a
basis of the vector space~$\mathfrak m$ which is a Jordan basis for~$\ad
(U)$. More precisely, the subspace~$\mathfrak m$ splits into a direct sum of
$\Ad(S)$-invariant subspaces~$\mathfrak m_\iota$, ($\iota=1, \dots, I$), of
dimension~$d_\iota+1$, on which~$\Ad(S)$ acts irreducibly. By the elementary
representation theory of~$\sl_2(\R)$, each subspace~$\mathfrak m_\iota$ has a
basis~$(\bel_{0,\iota}, \dots, \bel_{d_\iota,\iota})$ such that
\begin{equation}\label{eq:Adjoint_at}
 \ad (U) \bel_{j,\iota}=\bel_{j-1,\iota}\quad\text{
and } \quad \ad (A) \bel_{j,\iota}= \left( \frac{d_\iota}{2}-j \right) \bel_{j,\iota}.
\end{equation}
We choose~$\mathcal B=\{ \bel_{j,\iota} : 0\le j\le d_\iota, 1\le \iota\le I\}$
as a convenient Jordan basis for~$\ad (U)$ and we let, for
$\xi = \sum c_{j,\iota} \bel_{j,\iota}$,
\begin{equation}\label{eq:03_geometry:1}
  \norm{\xi}= \max_{\iota,j}{\abs{ c_{j,\iota}}},\quad \text{and}\quad
  d_{\mathfrak m} (\eb, \exp(\xi)) = \norm{\xi}.
\end{equation}
For any~$\xi \in \mathfrak m$,~$\Ad(\ub^s)\xi = \exp (s\ad (U))\xi$ is a
$\mathfrak m$-valued polynomial in the variable~$s$. Indeed, the usual
computation of the exponential of a Jordan block tells us that if
$\xi = \sum _{\iota,j} c_{j,\iota}\bel_{j,\iota}$ then
\[
  \exp (s\ad (U))\xi = \sum_{\iota,j}c_{j,\iota}(s)\bel_{j,\iota}, \quad
  c_{j,\iota}(s)=\sum_{k=0}^{d_\iota-j} c_{k+j,\iota}s^{k}/k!\,.
\]
Since the highest degree polynomials are given by the coefficients of the
terms~$\bel_{0,\iota}$, which together with~$U$ span the centraliser
$\mathfrak z(U)$ of~$U$ in~$\mathfrak g$ we are led to define for
$\gb_{\mathfrak m}= \exp \xi$ and
$\xi = \sum _{\iota,j} c_{j,\iota}\bel_{j,\iota}$ the polynomials
\[
 r_{\mathfrak
    m_\iota}(\gb_{\mathfrak m}, s):= \sum_{k=0}^{d_\iota}
  c_{k,\iota}(-1)^ks^{k}/k!\, ,
\]
and the vector valued polynomial
\[
  r_{\mathfrak m}(\gb_{\mathfrak m}, s) :=(r_{\mathfrak m_1}(\gb_{\mathfrak m},
  s), \dots, r_{\mathfrak m_I}(\gb_{\mathfrak m}, s)) \in \R^I\otimes \R[s]
\]
which controls the growth of the term
$\ub^{-s} \gb_{\mathfrak m}\ub^{s}= \exp \Ad(\ub^{-s})\xi$. Clearly
\begin{equation}\label{eq:01_background:rm_leq_dm}
  \abs{ r_{\mathfrak m}(\gb_{\mathfrak m}, s)} := \sup_{1\le\iota\le I}
  \abs{r_{\mathfrak m_\iota}( \gb_{\mathfrak m}, s)} \le d_{\mathfrak m}(\eb,
  \ub^{-s} \gb_{\mathfrak m}\ub^{s}).
\end{equation}

\subsubsection{}\label{comm_sl2}
We now turn to the term
$\ub^{-s}\gb_{\mathfrak s} \ub^{q(s)} \in \exp \mathfrak s$, where we recall
that the parametrisation~$q(s)$ is defined by the
identity~\eqref{eq:01_background:9}.

For~$\xb_{\mathfrak s}, \yb_{\mathfrak s} \in S$ and
$\gb_{\mathfrak s} = \xb^{-1}_{\mathfrak s} \yb_{\mathfrak s}= \exp (a A)\exp
(\bar u\overline U ) \ub^u~$ it is convenient to define the ``distance''
\begin{equation}\label{eq:01_background:def_dist_s}
  d_{\mathfrak s} (\xb_{\mathfrak s}, \yb_{\mathfrak s})
  = d_{\mathfrak s} (\eb_{\mathfrak s}, \gb_{\mathfrak s})
  =\abs{a} + \abs{\bar u} +\abs{u}.
\end{equation}

\noindent Writing
\begin{equation}\label{eq:01_background:2}
  \gb_{\mathfrak s} = \exp (a A)\exp (\bar u\overline U ) \ub^u,
\end{equation} the formula~\eqref{eq:01_background:9}, after an easy computation, yields
\begin{equation}\label{eq:sl2R:1}
  e^{a(s)/2}=e^{-a/2}(e^{a}-\bar u s) , \qquad \bar u(s)=e^{-a}\bar u
  (e^{a}-\bar us)
\end{equation}
\begin{equation}\label{eq:sl2R:2}
  q(s) = s - \frac{\bar u s^2 + s(1-e^a)}{ e^a-\bar u s} -u.
\end{equation}
Henceforth~$q(s)$,~$a(s)$,~$\bar{u}(s)$ denote the functions of~$s$ defined by
the formulas above (depending on the initial parameters~$a$,~$\bar u~$ and~$u$).

The following lemma/definition expresses the
condition~$\yb \ub^{q(s)}\in W(\xb \ub^s)~$ in terms of the initial parameters~$a$,
$\bar u$.

\begin{lemma}\label{lemma:int_max_track}
  Let~$\xb, \yb \in G$ with~$ \xb^{-1} \yb= \gb_{\mathfrak m}\gb_{\mathfrak s}~$
  and~$ \gb_{\mathfrak s} = \exp (a A)\exp (\bar u\overline U ) \ub^u$. Then
  the condition~$\yb \ub^{q(s)}\in W(\xb \ub^s)~$ is satisfied on the connected
  interval $[0,s]\cup[s,0]$ if,
  and only if,~$s$ belongs to the \emph{interval of maximal tracking} defined as
  \[ \mathcal I( \xb, \yb):=\mathcal I( \gb_{\mathfrak s}) :=\{ s\in \R :
    e^{a}-\bar u s >0\}.
  \]
  In particular if~$0< \eps< \min\big(\eps_0, 1/(2M)\big)$ and $d(\xb, \yb) \le \eps$, the
  condition~$\yb \ub^{q(s)}\in W(\xb \ub^s)~$ can be fulfilled for
  all~$\abs{s}< M$.
\end{lemma}

\begin{lemma}
  \label{lemma:I_max_track}
  Let~$\xb, \yb \in G$ with~$ \xb^{-1} \yb= \gb_{\mathfrak m}\gb_{\mathfrak s}~$
  and~$ \gb_{\mathfrak s} = \exp (a A)\exp (\bar u\overline U )
  \ub^u$. Assume~$\delta_1,\delta_2<\min(\eps_0,\ln 2 )$.
  If~$d (\xb, \yb)\le \delta_1$
  and~$d (\xb \ub^{s_0} , \yb \ub^{t_0})\le \delta_2$ for some~$s_0>0$
  and~$t_0\in \R$. Then~$s_0$ belongs to the interval of maximal
  tracking~$ \mathcal I(\xb,\yb)$ and
  \[
    \abs{q(s_0)-t_0}\le \delta_2\quad\text{and}\quad \abs{\bar u s}\le 3(\delta_1+\delta_2),\quad \forall s\in [0,s_0].
  \]
  Furthermore
  \begin{equation}\label{eq:03_geometry:2}
  1- 8\max(\delta_1, \delta_2)\le e^{a} -\bar u s < 1+ 8\max(\delta_1, \delta_2)
  ,\qquad  \forall s\in [0,s_0].
\end{equation}

Setting~$u_0:=q(s_0)-t_0$ we have
  \[
    \yb \ub^{t} = (\xb \ub^{s_0})\exp (a(s_0) A)\exp (\bar u(s_0)\overline U )
    \ub^{u_0}.
  \]
  Furthermore, for all~$s\in [0,s_0]$, we have
  \[
    \yb \ub^{q(s)} = (\xb \ub^{s })\exp (a(s) A)\exp (\bar u(s)\overline U
    )
  \]
  with
  \[
 \abs{a(s)}\le 11 \delta_1+ 6 \delta_2, \qquad \abs{\bar u(s)}\le 13\delta_1.
  \]
\end{lemma}
\begin{proof} In the course of the proof we use the estimates
  $1-\abs{x}/2 \le e^x \le 1 +2 \abs{x}$ for $\abs{x} \le \log 2 < 1$.

  By \S\ref{def_distance}, the condition
  $d (\xb \ub^{s_0} , \yb \ub^{t_0})\le \delta_2$ implies by
  $d_{\mathfrak s} (\eb, \ub^{-s_0} \gb_{\mathfrak s} \ub^{t_0})\le
  \delta_2$. It follows that the interval of
  orbit~$\{\yb \ub^{t_0+t'}\}_{\abs{t'}\le \delta_2}$ crosses the
  leaf~$W(\xb \ub^{s_0})$ in a point $\yb \ub^{q}$ such
  that~$\abs{q-t_0}\le \delta_2$. Then
  \[
    \ub^{-s_0}\gb_{\mathfrak s} \ub^{q} =\exp (a' A)\exp (\bar{u}'\, \overline U
    ),
  \]
  with~$\abs{a'}+\abs{\bar{u}'} \le \delta_2$.

  From \S\ref{comm_sl2} the only solution of this identity is given
  by~$q=q(s_0)$,~$a'=a(s_ 0)$,~$\bar{u}'=\bar{u}(s_0)$,
  where~$q(s)$,~$a(s)$,~$\bar{u}(s)$ are the functions defined by the
  formulas~\eqref{eq:sl2R:1}--\eqref{eq:sl2R:2}.

  Since~$\abs{a}+\abs{\bar{u}}+ \abs{u}\le \delta_1$, using the first
  identity~\eqref{eq:sl2R:1} we have
  \[
    e^a -\bar u s_0 = e^{(a'+a)/2}\ge 1 -\frac{\abs{a'+a}}{4} > \frac 1 2.
  \]
  proving that $s_0\in \mathcal I(\xb,\yb)$. Furtermore, since
  $ e^a\le 1+2\abs{a}\le 3$ and $e^{(a'-a)/2} -1 \le \abs{a'-a}$, we have
  \[
    \abs{\bar u s_0} = e^a\abs{1 -e^{(a'-a)/2}} \le 3(\delta_1+\delta_2)
  \]
  and consequently~$ \abs{\bar u s}\le 3(\delta_1+\delta_2)$ for all
  $s\in [0,s_0]$. Then, the estimate \eqref{eq:03_geometry:2} follows easily.

  Finally formula \eqref{eq:sl2R:1} and the
  inequalities~$\abs{\bar u s}\le 3(\delta_1+\delta_2)\le 6$, for
  all~$s\in [0,s_0]$, and $e^{a}-\bar u s\le 1+2\abs{a} + \abs{\bar u s}$ imply
  \[\abs{a(s)}\le \abs{a}+ 2\abs{\log(e^{a}-\bar u s)}\le 11 \delta_1 +6\delta_2
  \]
  and
  \[
    e^{(a(s)-a)/2} = 1 - e^{-a} \bar u s \le 13.
  \]
  Consequently
  \[
    \abs{\bar u(s)}= e^{(a(s)-a)/2}\abs{\bar u} \le 13 \delta_1 \] for
  all~$s\in [0,s_0]$.
\end{proof}

\subsubsection{}\label{sublinear_trackings}
In the sequel we shall deal with parametrisations~$\tau(s)$ of the orbit
of~$\yb$ satisfying, for sufficiently large~$s$, the inequality
\[
  \abs{\tau(s) -s}\le \frac{1}{2}(s^{1-\eta}+ \eps),
\]
and such that
\[
  d(\xb \ub^{s}, \yb \ub^{\tau(s)} )\le \frac{\eps}{2}
\]
for some~$s>0$  and $\eps < \min(\eps_0, \log 2)$. By the previous lemma if this second condition occurs we have
$s\in \mathcal I(\xb , \yb )$ and~$\abs{\tau(s)-q(s)} \le \eps/2$. Then by the
first condition~$ \abs{q(s) -s}\le 1/2 s^{1-\eta}+ \eps$. From
formulas~\eqref{eq:03_geometry:2} and~\eqref{eq:sl2R:2}, the quantity $\bar u s^2
+ s(1-e^a)$ up to a factor close to~$1$ is the delay $s-q(s)$ of the two orbits.
Thus we are led to define
\begin{equation}\label{eq:01_background:rs_leq_s}
  r_{\mathfrak s}(\gb_{\mathfrak s},s) := \bar u s^2 + s(1-e^a),
\end{equation}
and study the sub-level sets where
$\abs{r_{\mathfrak s}(\gb_{\mathfrak s},s)} \le s^{1-\eta} + \eps$.  When these
intervals are not too spaced from each other the following proposition comes to
aid.
\begin{proposition}
  Let~$(V, \abs{\,\cdot\,}_V)$ be a finite dimensional normed vector space
  and~$P(s)= c_0+ c_1s + \dots+c_k s^k$ be a~$V$-valued polynomial of degree~$k$
  bounded by~$\max\{\eps, C s^{1-\eta} \}$ on consecutive
  intervals~$J_1=[s_1=0,\bar s_1]$, \dots,~$J_h=[s_h,\bar s_h]$. (We shall
  assume~$\eps \le C \bar s_1^{1-\eta}$.)\enspace Suppose that the
  intervals~$J_1,\dots, J_h$ are \emph{~$b$-close} in the sense that, for
  some~$b > 0$
  \[
    d(J_j, J_{j+1}):= s_{j+1}- \bar s_{j} \le \bar s_{j}^{1+b}, \quad \forall
    j=1,\dots, h-1 .\] Then, for some constants~$\kappa_k$,~$\kappa_k'$,
  depending only on~$k$, we have
  \[
    \abs{c_i}_V\le \kappa_k (\kappa_k')^{h-1} \bar s_h^{-i + 1-\eta + (h-1)kb}.
  \]
  Thus if~$ (h-1)kb < \eta/2~$ we have
  \[
    \abs{c_i}_V\le \kappa_k (\kappa_k')^{h-1} \bar s_h^{-i + 1-\eta/2}.
  \]

\end{proposition}
\begin{proof}
  Since the space of polynomials on~$[0,1]$ of degree lesser or equal than $k$
  with values in a finite vector space~$V$ has dimension~$(k+1)\times\dim V$,
  the norms~$\norm{P}_1=\sup_{s\in [0,1]} \abs{P(s)}_V$
  and~$\norm{P}_2= \sup_{0\le i \le k} \abs{c_i}_V$ are equivalent.

  It follows there exists a constant~$\kappa_k>1$ such that, if a polynomial of
  degree~$k$ such as~$P$ is bounded by a constant~$K$ on a interval~$[0,L]$,
  then~$\abs{c_i}_V\le \kappa_k K \,L^{-i}~$. Thus, with our assumptions, if~$h=1$
  we have
  \[
    \abs{c_i}_V\le \kappa_k C \bar s_1^{-i + 1-\eta} .
  \]
  This inequality implies that on the
  interval~$[\bar s_1,\bar s_1 + \bar s_1^{1+b}]$, the polynomial~$P(s)$ is
  bounded
  by~$\abs{\sum_{i=0}^k c_i \bar s_1^i (1+ \bar s_1^{b})^i }_V\le C \sum_{i=0}^k
  \kappa_k (\bar s_1)^{1-\eta} (1+ \bar s_1^{b})^i\le C \kappa_k' \bar
  s_1^{1-\eta+ k b}$.

  In the case~$\bar s_2^{1-\eta} \le \kappa_k'\bar s_1^{1-\eta+ k b}$ we deduce
  \begin{equation}\label{eq:sl2R:3}
    \abs{c_i}_V\le  C \kappa_k\kappa_k'   \bar s_1^{1-\eta+ k b} s_2^{-i}\le
    \kappa_k\kappa_k' \bar s_2^{-i+1-\eta+ k b}.
  \end{equation}
  If~$\bar s_2^{1-\eta} > \kappa_k'\bar s_1^{1-\eta+ k b}$ then we deduce
  \[
    \abs{c_i}_V\le C \kappa_k \bar s_2^{-i+1-\eta}.
  \]
  Thus the worse case scenario is given by the estimate~\eqref{eq:sl2R:3}. By
  induction we obtain that
  \[
    \abs{c_i}_V\le C \kappa_k (\kappa_k')^{j-1} \bar s_j^{-i+1-\eta +
      (j-1)kb}.\qedhere
  \]
\end{proof}

We shall apply this proposition to the~$\R^I$-valued polynomial
$r_{\mathfrak m}(\gb_{\mathfrak m}, s)$ bounded by~$\eps$ (in this case in the
proposition we choose~$C=\eps$ and~$\eta=1$), and to the real valued polynomial
$r_{\mathfrak s}(\gb_{\mathfrak s}, s)$ bounded by
$\max\{\eps, C s^{1-\eta} \}$. We remark that the set
$\{ s\ge 0 : r_{\mathfrak m}(\gb_{\mathfrak m}, s)\le\eps , r_{\mathfrak
  s}(\gb_{\mathfrak s}, s) \le \max\{\eps, C s^{1-\eta} \} \}$ consists of at
most $D= 2 \prod_{\iota=1}^I d_\iota$ connected components.

\begin{proposition}
  \label{prop:01_background:1}
  There exists a constant~$\kappa>0$ such that the following holds true.
  Let~$\xb, \yb \in G$ and~$ \xb^{-1} \yb= \gb_{\mathfrak m}\gb_{\mathfrak s}~$
  with~$ \gb_{\mathfrak s} = \exp (a A)\exp (\bar u\overline U ) \ub^u~$
  and~$\gb_{\mathfrak m}= \sum_{j,\iota} c_{j,\iota} E_{j,\iota}$.

  Suppose we have~$h\le D$ intervals~$J_1=[s_1=0,\bar s_1]$,
  \dots~$J_h=[s_h,\bar s_h=L]$ such that
  \begin{enumerate}
  \item~$r_{\mathfrak m}(\gb_{\mathfrak m}, s)\le \eps$
    and~$r_{\mathfrak s}(\gb_{\mathfrak s} ,s)\le \max\{\eps, C s^{1-\eta} \}$
    for all~$s \in \bigcup_{j=1}^h J_j~$.
  \item the intervals~$J_1$,\dots~$J_h$ are~$b$-close for
    some~$b < \eta/(2 D\times \max_\iota d_\iota)$.
  \end{enumerate}
  Then,
  \[
    \abs{a} \le \kappa L^{-\eta/2}, \quad \abs{\bar u}\le \kappa L^{-1-\eta/2},
  \]
  and
  \[
    \abs{c_{j,\iota}} \le \kappa \epsilon L^{-j +\eta/2}.
  \]
\end{proposition}

The role of the intervals~$J_1, \dots, J_h$ will become clear in the next
subsection.

\subsection{Blocks.}\label{Sec:blocks_def_lg}

\subsubsection{}
We now explain how we are going to apply Proposition
\ref{prop:01_background:1}. In order to do it, we are going to introduce the
notion of $(\rho,\eps)$-\emph{blocks} (or simply \emph{blocks}), which will play
a crucial role in the proof of Ratner's Basic Lemma in the next Section.

In this section, for simplicity, we will write~$r_{\mathfrak m}(\gb, s)$ and
$r_{\mathfrak s}(\gb, s)$ instead of~$r_{\mathfrak m}(\gb_{\mathfrak m}, s)$
and~$r_{\mathfrak s}(\gb_{\mathfrak s}, s)$. We define

\begin{equation*}
  \begin{split}
    l_\eps(\gb_{\mathfrak m})&:= \sup \{ s>0 : \abs{r_{\mathfrak m}(\gb, t)}
    \le \eps \text{\ for all\ } t \in [0,s]\}, \\
    l_\eps(\gb_{\mathfrak s})&:= \sup \{ s>0 : \abs{r_{\mathfrak s}(\gb, t)}
    \le t^{1-\eta}+\eps \text{\ for all\ } t \in [0,s]\}, \\
    l_\eps(\gb) &:= \min\{ l_\eps(\gb_{\mathfrak m}),l_\eps(\gb_{\mathfrak
      s})\}.
  \end{split}
\end{equation*}

Let $\rho \leq \eps/2$.  Let~$\xb, \yb \in G$ with~$d(\xb,\yb) \le \rho$ be
fixed, and let~$\tau(s)$ be a parametrization of the~$U$-orbit of~$\yb$
satisfying~$\tau(0)=0$ and
$\abs{\tau(s) - s} \le \frac{1}{2} (s^{1-\eta} + \eps)$ for
all~$s \ge l_\eps(\gb)~$.  Let now~$\xb_i = \xb \ub^{s_i}$
and~$\yb_i = \yb \ub^{\tau(s_i)}$ be two points on the~$U$-orbits of~$\xb$
and~$\yb$ satisfying $d(\xb_i,\yb_i) \le \rho$.  A $(\rho,\eps)$-\emph{block}
for~$\xb_i, \yb_i$ is an interval
\[
  J_i = [s_i, \bar s_i], \quad \text{ with } \quad \abs{J_i} = \bar s_i - s_i
  \le l_\eps(\gb_i),
\]
where~$\gb_i = \xb_i^{-1}\yb_i$, such that
\[
  d(\xb \ub^{\bar s_i}, \yb \ub^{\tau(\bar s_i)}) \le \eps,
\]
in other words, if not only the initial points but also the final points are at
distance not larger than~$\rho$.  By definition, the translated
interval~$J- s_i = [0, \bar s_i - s_i]$ is contained in a sub-level set of the
polynomials $r_{\mathfrak m}(\gb_i, s)$ and~$r_{\mathfrak s}(\gb_i, s)$.  Since
$\rho$ and $\eps$ are fixed, we will simply say \emph{block}.

Let~$J_i = [s_i, \bar s_i]$ be a block for~$\xb_i, \yb_i$, and assume that there
exists~$s_j \geq s_i + l_\eps(\gb_i) \geq \bar s_i$ such that the points
$\xb_j = \xb \ub^{s_j}$ and~$\yb_j = \yb \ub^{\tau(s_j)}$ are at
distance~$d(\xb_j, \yb_j) \le \eps$.  Let~$J_j$ be a block for~$\xb_j, \yb_j$,
hence we know that~$J_j-s_j$ is an interval contained in a sub-level set
for~$r_{\mathfrak m} (\gb_j, s)$ and~$r_{\mathfrak s}(\gb_j, s)$ (notice that
these polynomials are defined by~$\gb_j = \xb_j^{-1} \yb_j$ and not
by~$\gb_i = \xb_i^{-1} \yb_i$).  The following lemma shows that~$J_j-s_i$ is
contained in a (slightly larger) sub-level set for~$r_{\mathfrak m}(\gb_i, s)$
and~$r_{\mathfrak s}(\gb_i, s)$.

\begin{lemma}\label{lemma:same_sub_level_sets}
  There exists an absolute constant~$\tilde \kappa >0$ such that the following
  holds.  Let~$J_j= [s_j, \bar s_j]$ be a block for~$\xb_j, \yb_j$ as above.
  Then,
  \[
    J_j - s_i = [s_j-s_i, \bar s_j- s_i] \subset \{ \abs{r_{\mathfrak m}(\gb_i,
      s)} \le \tilde \kappa \eps \} \cap \{\abs{r_{\mathfrak s}(\gb_i, s)} \le
    \tilde \kappa (s^{1-\eta}+\eps) \}.
  \]
\end{lemma}
\begin{proof}
  To simplify the notation, in this proof we will write~$\tilde \gb$ instead
  of~$\gb_j$ and~$\gb$ instead of~$\gb_i$.

  Denote by~$\abs{\tilde J}= \bar s_j - s_j$ the length of the block
  $J_j = \tilde J$.  By definition, for all~$0\leq r \leq \abs{\tilde J}$, we
  have that $\abs{r_{\mathfrak m}(\tilde \gb, r)} \le \eps$.  By Proposition
  \ref{prop:01_background:1} (applied with only one interval), we get that the
  coefficients~$\tilde c_{j,\iota}$ in the expression
  $\tilde \gb_{\mathfrak m} = \sum_{j,\iota} \tilde c_{j,\iota} \bel_{j,\iota}$
  satisfy~$\abs{\tilde c_{j,\iota}} \le \kappa \epsilon \abs{\tilde J}^{-j
    +\eta/2}$.  Hence, from \S\ref{sec:about_m}, it follows that
  $d_{\mathfrak m}(\eb, \ub^{-r} \tilde \gb_{\mathfrak{m}} \ub^r) \le \tilde
  \kappa \eps$ for some constant~$\tilde \kappa$ depending on~$\mathfrak{m}$
  only.  By \eqref{eq:01_background:rm_leq_dm}, we deduce that
  \[
    \abs{r_{\mathfrak m}(\gb, r+s_j-s_i)} \le d_{\mathfrak m}(\eb,
    \ub^{-(r+s_j-s_i)} \gb_{\mathfrak{m}} \ub^{r+s_j-s_i}) =d_{\mathfrak m}(\eb,
    \ub^{-r} \tilde \gb_{\mathfrak{m}} \ub^r) \le \tilde \kappa \eps,
  \]
  which shows
  that~$\tilde J - s_i \subset \{ \abs{r_{\mathfrak m}(\gb, s)} \le \tilde
  \kappa \eps \}$.

  Let us prove the other inclusion. By \cref{lemma:int_max_track}, the whole
  interval~$[0, \bar s_j]$ is contained in the interval of maximal
  tracking~$\mathcal{I}(\xb, \yb)$.  Moreover, by \cref{lemma:I_max_track} and
  the subsequent discussion, since~$d(\xb_j, \yb_j) \leq \rho \leq \eps/2$, we
  have that
  \[
    \abs{q(s_j - s_i) - (s_j - s_i) }\le (s_j - s_i)^{1-\eta} + \eps.
  \]
  On the other hand, by assumption and by \eqref{eq:sl2R:2}, we have that
  \[
    \abs{q(r)-r} \le \frac{\abs{r_{\mathfrak s}(\tilde \gb, r)}}{1-10\eps}+\eps
    \le 2 r^{1-\eta} + 2 \eps,
  \]
  for all~$0 \leq r \leq \abs{\tilde J}$.  Therefore, we get
  \begin{multline*}
    \abs{q(r+s_j - s_i) - (r+s_j - s_i)} \le \abs{q(r)-r}  + \abs{q(s_j - s_i) - (s_j - s_i)} \\
    \qquad \le 2 r^{1-\eta} +(s_j - s_i)^{1-\eta} +3 \eps \le 4(r+s_j -
    s_i)^{1-\eta} +4 \eps.
  \end{multline*}
  Finally, again by \eqref{eq:sl2R:2}, since
  \begin{equation*}
    \begin{split}
      \abs{q(r+s_j - s_i) - (r+s_j - s_i)} &\geq \frac{\abs{r_{\mathfrak s}(\gb, r+s_j - s_i)}}{1+10\eps} - \eps \\
      &\geq \frac{1}{2}\abs{r_{\mathfrak s}(\gb, r+s_j - s_i)} - \eps,
    \end{split}
  \end{equation*}
  combining the two inequalities above, we conclude
  that~$\tilde J - s_i \subset \{\abs{r_{\mathfrak s}(\gb, s)} \le
  8(s^{1-\eta}+\eps) \}$, which completes the proof.
\end{proof}

In layman's terms, the previous lemma says that all the blocks that we can find
along the future orbits of two fixed points~$\xb, \yb$ must be contained in a
certain sub-level set of a polynomial map determined only by the initial
points~$\xb, \yb$, whose degree is \emph{bounded independently of the points}.

For any given~$\xb_i, \yb_i$, let~$S_1, \dots, S_D$ be the connected components
of the sub-level set
\[
  \sublevi:=\{ \abs{r_{\mathfrak m}(\gb_i, s)} \le \tilde \kappa \eps \} \cap
  \{\abs{r_{\mathfrak s}(\gb_i, s)} \le \tilde \kappa (s^{1-\eta}+\eps) \}.
\]
Note again that~$D$ is bounded by a constant depending on~$\mathfrak g$ only.
By \cref{lemma:same_sub_level_sets}, for any block~$J_j$ as above, there
exists~$S_{i_j}$ such that~$J_j-s_i \subset S_{i_j}$.
We now use \cref{prop:01_background:1} to deduce the following corollary.

\begin{corollary}
  \label{cor:bounds_on_coeffs}
  There exists a constant~$\kappa>0$ such that the following holds.
  Let~$\xb, \yb \in G$ and~$\tau(s)$ be as above.  Let~$J_1=[s_1=0,\bar s_1]$ be
  a $(\rho,\eps)$-block for~$\xb,\yb$.  Assume that there
  exists~$s_2 > l_\eps(\gb_1)$ such that~$\xb_2= \xb\ub^{s_2}$
  and~$\yb_2 = \yb \ub^{\tau(s_2)}$ are at distance at most~$\rho\leq\eps/2$,
  and, as before, let~$J_2 = [s_2, \bar s_2]$ be a $(\rho,\eps)$-block
  for~$\xb_2, \yb_2$.  Assume that this procedure can be repeated~$\bar j$
  times, and let~$J_1$, \dots,~$J_{\bar j}=[s_{\bar j},\bar s_{\bar j}=L]$ be
  the resulting blocks.

  Let~$S_1, \dots, S_h$ be the connected components of~$\sublev$ as above,
  with~$h \le D$ minimal so that~$\bigcup_{i=1}^h S_i$
  covers~$\bigcup_{j=1}^{\bar j} J_j$.  If the intervals~$S_1,\dots, S_h$
  are~$b$-close (in the sense of \cref{prop:01_background:1}) for
  some~$b < \eta/(2 D\times \max_\iota d_\iota)$, then,
  \[
    \abs{u}\le \eps, \quad \abs{a} \le \kappa L^{-\eta/2} \quad \abs{\bar u}\le
    \kappa L^{-1-\eta/2}, \quad \text{ and } \quad \abs{c_{j,\iota}} \le \kappa
    \epsilon L^{-j +\eta/2}.
  \]
\end{corollary}
\begin{proof}
  We apply \cref{prop:01_background:1} to the
  intervals~$S_1, \dots, S_{h-1}, S_h \cap J_{\bar j}$: by assumption they
  are~$b$-close and by Lemma \ref{lemma:same_sub_level_sets} the assumption 1 of
  \cref{prop:01_background:1} is satisfied as well.
\end{proof}

If the assumptions of Corollary \ref{cor:bounds_on_coeffs} are satisfied, it is
possible to see that the~$\ab^t$-orbits of~$\xb$ and~$\yb$ stay close for all
times up to~$(1+\eta)\log L$, as the next lemma shows.

\begin{lemma}\label{lemma:dist_a_t}
  Under the assumptions of Corollary \ref{cor:bounds_on_coeffs}, for
  every~$0<r<1$ we have
  \[
    d(\xb \ab^{(1+r) \log L}, \yb \ab^{(1+r) \log L}) \leq 3\kappa \left(
      L^{-\frac{1}{2}+\frac{r+\eta}{2}} + L^{-\frac{\eta}{2}+r} \right).
  \]
\end{lemma}
\begin{proof}
  Call~$T= (1+r) \log l$.  As in \S\ref{sec:geom_U_flow_1}, let us
  write~$\xb^{-1} \yb = \exp(\xi) \gb_{\mathfrak s}$,
  where~$\xi \in \mathfrak{m}$ and~$\gb_{\mathfrak s} \in \mathfrak{s}$ satisfy
  \eqref{eq:track1}.  Then,
  \[
    d(\xb \ab^T, \yb \ab^T) = d(\eb, \ab^{-T} \xb^{-1} \yb \ab^T) = \max \{
    \norm{\ab^{-T} \xi \ab^T}, \ d_{\mathfrak s} ( \eb, \ab^{-T}\gb_{\mathfrak
      s}\ab^T)\}.
  \]
  We now focus on the first term corresponding
  to~$\xi = \sum c_{j,\iota} \bel_{j,\iota}$.  By \eqref{eq:Adjoint_at}, we have
  \begin{equation*}
    \begin{split}
      \norm{\ab^{-T} \xi \ab^T} &= \max_{j,\iota} | c_{j,\iota} | \cdot
      \norm{\ab^{-T} \bel_{j,\iota} \ab^T}
      = \max_{j,\iota} | c_{j,\iota} | \cdot \norm{\Ad(\ab^{-T}) \bel_{j,\iota}} \\
      &=\max_{j,\iota} | c_{j,\iota} | e^{-\left(\frac{d_{\iota}}{2}-j\right)T}.
    \end{split}
  \end{equation*}
  By Corollary \ref{cor:bounds_on_coeffs} and the trivial bounds
  $j \leq d_{\iota}$ and~$d_{\iota} \geq 1$, we obtain
  \begin{equation}\label{eq:lemma_backg_01}
    \norm{\ab^{-T} \xi \ab^T} \leq \kappa \epsilon
    L^{-j +\frac{\eta}{2} - (1+r)\left(\frac{d_\iota}{2} - j\right)} \leq \kappa \epsilon
    L^{-\frac{1}{2}+\frac{r+\eta}{2}}.
  \end{equation}

  We now look at the term corresponding to~$\gb_{\mathfrak s}$.  Writing it as
  in \eqref{eq:01_background:2}, by definition
  \eqref{eq:01_background:def_dist_s}, we have
  \[
    d_{\mathfrak s} ( \eb, \ab^{-T}\gb_{\mathfrak s}\ab^T) = \abs{e^T \, \bar u}
    + |a| + \abs{e^{-T} \, u}.
  \]
  By Corollary \ref{cor:bounds_on_coeffs}, we conclude
  \[
    d_{\mathfrak s} ( \eb, \ab^{-T}\gb_{\mathfrak s}\ab^T) \leq \kappa
    L^{(1+r)-1-\frac{\eta}{2}} + \kappa L^{-\frac{\eta}{2}} + \eps L^{-(1+r)}
    \leq 3\kappa L^{-\frac{\eta}{2}+r}.
  \]
  Together with \eqref{eq:lemma_backg_01}, this completes the proof.
\end{proof}


%% file: 04_basic_lemma.tex
\section{Ratner's Basic Lemma}\label{sec:basiclemma}

\subsection{Construction of the Good Set $\Good$.}

Here we construct a set of large measure which satisfies simultaneously the
various conditions proved or assumed in the previous section. The first
condition obviously guarantees that two orbits, that are geometrically close,
are mapped to two orbit segments that are geometrically close. By choosing in
the source space initial points sufficiently close, we obtain, both in the
source space and the target space, orbit segments of length larger than $m_0$.
Then, the second condition guarantees that the time change on these orbit
segments is sub-linear, which yields important consequences for the relative
position of the initial points of these orbit segments, thanks to the
Lemma~\ref{lemma:same_sub_level_sets}. The third condition assures that in the
target space the orbit segments with close initial points and close endpoints
have continuous lifts to $G_2$.

These three conditions are sufficient to deal with the conjugacy of the
centraliser of $\ub_1^t$. The fourth condition is used to control the conjugacy
of the normaliser. In that case we need to push orbits along the flow of a
normalising element. The fourth condition  guarantees that along these pushes 
we can keep track of the initial positions of orbit segments.

\subsubsection{}\label{sec:alpha_and_K_2}

Recall that $\eta \in (0,1/4)$ is given by Corollary \ref{thm:mixing_condition}.
Let us fix $b >0$ satisfying
\begin{equation}\label{eq:defalpha}
  b < \frac{\eta}{2D \cdot \max_{\iota} d_{\iota}}, \quad \text{ and define }
  \quad c(b):=\frac{1+\eta}{1+\eta-b}.
\end{equation}
Let $A_2 \in \mathfrak{g}_2$ be the semi-simple element of a $\sl_2(\R)$-triple
as defined in \S\ref{sec:sl2_triple} for the unipotent element $U_2\in \mathfrak{g}_2$.  As usual, let
$\{\ab_2^t\}_{t \in \R}$ the associated one-parameter subgroup.

\begin{proposition}[Construction of $\Good$]\label{prop:good_set}
  For every $\omega>0$, there exist $m_0>1$, $\rho \in (0,1)$ and a compact set
  $\Good \subset M_1$ with measure $\mu_1(\Good) > 1- \omega$, such that the following properties hold.
  \begin{enumerate}
  \item \emph{Uniform continuity:} the map $\psi$ is uniformly continuous on
    $\Good$.
  \item \emph{Sub-polynomial deviations:} for all $x \in \Good$ and all
    $t \geq m_0$ we have
    \[
      |\xi_1(x,t) - t | \leq \Ctau^{-4}t^{1-\eta} \text{\ \ \ and\ \ \ }
      |\xi_2(\psi(x),t) - t | \leq \Ctau^{-4}t^{1-\eta},
    \]
    where $\eta$ is given by \cref{thm:mixing_condition}.
  \item \emph{$\IC(\rho, 2\Ctau^4 m_0)$ on the image $\psi(\Good)$:} for
    $\xb \in \pi_2^{-1}( \psi(\Good) )$ and $d(\xb, \yb) < \rho$, then
    $d(\xb \ub_2^s, \gamma \yb \ub_2^t) > \rho$ for all
    $\gamma \in \Gamma_2 \setminus \{\eb\}$ and all
    $s,t \in [-2\Ctau^4m_0,2\Ctau^4m_0]$.
  \item \emph{The condition $\FBR$ is satisfied on the image $\psi(\Good)$:}
    for all $x \in \Good$, the point $\psi(x)$ satisfies the condition
    $\FBR(T_0,c(b),r_0)$ with
    \[
      T_0 \leq \frac{1}{2} \log \frac{m_0}{\Ctau}, \quad \text{ and } \quad r_0
      \geq 20 \kappa \Ctau m_0^{-\frac{\eta - 2 b}{12}}.
    \]
  \end{enumerate}
\end{proposition}
\begin{proof}
  The set $\Good$ will be given as the intersection of four sets of measure
  larger than $1-\omega/4$ on which the conditions 2)--5) hold. Without
  mentioning it, we use several times the estimate~\eqref{eq:01_background:6}.

  By Lusin's Theorem we can choose a compact set $K_1 \subset M_1$ of measure
  larger than $1-\omega/4$ on which $\psi$ is uniformly continuous.

  By Lemma \ref{lemma:FBR}, there exist $T_0 \geq 1$ and $r_0 >0$ such that the
  set $K_4 \subset M_2$ of points which satisfy the condition $\FBR(T_0, c(b), r_0)$ has
  measure greater than $1-\frac{\omega}{4 \Ctau}$.  Thus,
  $\mu_1( \psi^{-1}K_{4}) > 1- \omega/4$.   In order to satisfy the
  statement~(4) of the Proposition, we choose $m_0>1$ so large that $T_0 \leq \frac{1}{2} \log \frac{m_0}{\Ctau}$ and
  $r_0 \geq 20 \kappa \Ctau m_0^{-\frac{\eta - 2 b}{12}}$.

  By \cref{thm:mixing_condition} (applied twice with  $\omega$ replaced by $\omega/(4\Ctau)$), there
  exist $m' >1$ and a set $K_2\subset M_1$ of measure greater than $1-\omega/4$ such that
  for all $x_1 \in K_2, x_2 \in \psi(K_2)$ and all $t\geq m'$ we have
  \[
    \abs{\xi_1(x_1,t) - t } \leq \Ctau^{-4}t^{1-\eta} \quad \text{and}\quad
    \abs{\xi_2(x_2,t) - t } \leq \Ctau^{-4}t^{1-\eta}.
  \]
  Up to increasing $m_0$, we can assume $m_0>m'$.

  By Proposition \ref{thm:IC}, with $\zeta= \omega/(4\Ctau^2)$ and
  $m = 2C^4 m_0$, there exist $\rho >0$ and a compact set $K_3 \subset M_2$ with
  $\mu_2(K_3)\ge 1-\omega/(4\Ctau^2)$ on which the Injectivity Condition
  $\IC(\rho, 2C^4 m_0)$ holds for any point in $K_3$.

  Then
  $\mu_1(K_1 \cap K_2 \cap \psi^{-1}K_{3} \cap \psi^{-1} K_{4}) > 1-\omega$.
  We can choose a compact set
  \[
    \Good \subset K_1 \cap K_2 \cap \psi^{-1}K_3 \cap \psi^{-1} K_{4}
  \]
  of measure $\mu_1(\Good) \geq 1-\omega$.  The construction is complete, and
  the properties of $\Good$ follow from its definition.
\end{proof}

In the following, we will denote $m = 2 \Ctau^4 m_0 > m_0$.

\subsection{Statement and proof of Ratner's basic lemma.}
In this section we will prove the main technical result of this paper, which is
an adaptation of Ratner's Basic Lemma in~\cite{Ratner:Acta} for $G=\SL_2(\R)$. A
similar result was also proved by Tang in~\cite{Tang:timechanges} for $G =
\operatorname{SO}(n,1)$.

For a given $\omega>0$ (which will be fixed in the next section), let
$m_0, \rho, T_0, r_0$, and $\Good$ be given by Proposition \ref{prop:good_set}.

\subsubsection{Standing assumptions.} We fix $\eps < 1/10$ (the specific choice of~$\eps$ plays no role).  Up to
possibly taking a smaller $\rho \leq \eps /2$, we can assume that if
$d(\gb ,\eb) \le \rho$ then $l_\eps(\gb)>m$, where, we recall, $l_\eps(\gb)$ was
defined in \S\ref{Sec:blocks_def_lg}.  For any such $\gb$, we write
\[
  \gb = \exp \Big( \sum_{\iota, j} c_{j, \iota} \bel_{j,\iota} \Big) \, \exp(aA)
  \, \exp(\bar u \bar U) \, \ub^u
\]
as in the previous section.

\begin{lemma}[Ratner's Basic Lemma]\label{lemma:basic}
  There exists $\theta \in (0,1)$ such that the following holds.  Let
  $x, y \in M_2$.  Assume that there exist $\lambda>m$, a compact set
  $A=A(x,y) \subset[0,\lambda]$, and a strictly increasing Lipschitz function
  $\tau=\tau_{x,y}\colon[0,\lambda]\to\R_+$, with $\tau(0)=0$, such that
  \begin{enumerate}
  \item $0,\lambda_s\in A$;
  \item If $r \in A$ then the point $x \ub_2^r$ satisfies the Injectivity
    Condition $\IC(\rho,m)$ and the Frequently Bounded Radius Condition
    $\FBR(T_0, c(b), r_0)$.
  \item If $r \in A$ then
    \[
      d\left( x \ub_2^r, y \ub_2^{\tau(r)} \right) < \rho,
    \]
  \item \label{basiclemmaholder} if $r \in A$ and $r'-r>m$ or
    $\tau(r')-\tau(r) >m$, then
    \[
      \abs[\big]{ \big(\tau(r')-\tau(r)\big) -( r'-r) } \leq 4(r' - r)^{1-\eta},
    \]
    for some $\eta>0$.
  \end{enumerate}
  If the relative density of $A$ in $[0,\lambda]$ is greater than $1-\theta/8$,
  then there exists $\bar s \in A$ such that
  $x\ub_2^{\bar s} \gb = y\ub_2^{\tau(\bar s)}$, where $\gb$ satisfies
  \[
    |u|\leq \varepsilon, \quad |a|\leq \kappa \lambda^{-\eta/2}, \quad |\bar
    u|\leq \kappa \lambda^{-1-\eta/2}, \quad |c_{\iota,j}|\leq \kappa
    \varepsilon \lambda^{-j+\eta/2}.
  \]

  Moreover, if the assumptions above hold for all $\lambda \geq \lambda_0 >m$,
  then $y=x\gb$, where $\gb$ \emph{commutes} with $\ub_2^t$ for all $t \in \R$
  and satisfies
  $d(\eb, \ub_2^{-\bar s} \gb \ub_2^{\tau(\bar s)}) \leq \varepsilon$ for some
  $0\leq \bar s \leq \lambda_0$.
\end{lemma}

\subsection{Comments on the proof of the Basic Lemma~\ref{lemma:basic}}\label{sec:blocks}
Before diving into the details of the proof of the Basic
Lemma~\ref{lemma:basic}, let us sketch the general strategy. Let $x, y \in M_2$
two points at distance lesser than $\rho$, and fix two lifts $\xb, \yb \in G_2$
at the same distance.  From the assumption that the orbits of $x$ and $y$
(parametrized by $r$ and $\tau = \tau(r)$) stay close in $M_2$ for a large
portion of times, we construct a collection of what we call \emph{blocks}: lifts
to $G_2$ of segments of $\ub_2^t$-orbits of $x$ and $y$ which are close
together.  The lifts are chosen so that the one for $x$ lies on the
$\ub_2^t$-orbits of $\xb$, and the one for $y$ so that the two lifts are close
in $G_2$; these conditions uniquely determine such lifts.  Note that two cases
are possible: either the lift of the orbit segment of $y$ also lies in the orbit
of $\yb$, or not, see \cref{fig:blockrelation}. This type of dichotomy will result in a \emph{relation}
between blocks, which we define below.
\begin{figure}[tb]
  \centering
  \includegraphics{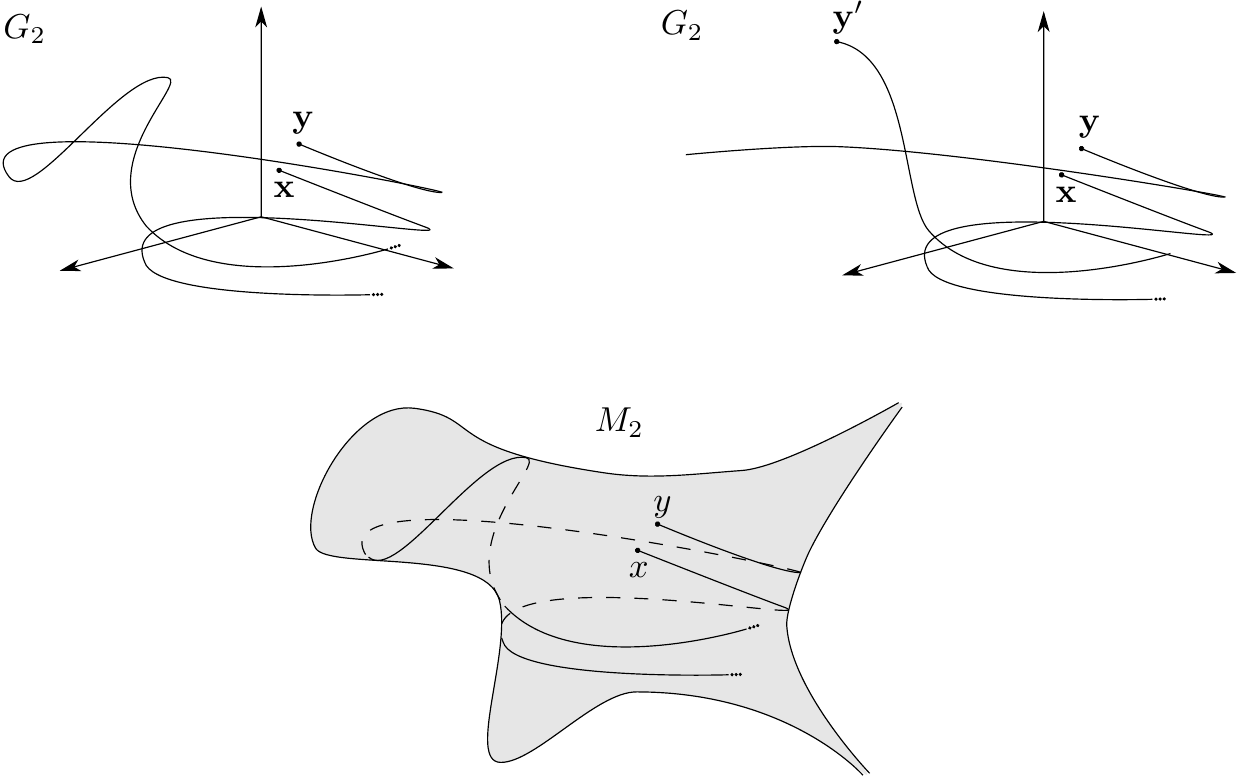}
  \caption{A schematic drawing of the relation between blocks. Below are the
  orbits, under $\ub_2^t$, of $x$ and $y$ on $M_2$, which are close for two
  blocks of times. The lifts of these orbits to $G_2$ could either be the orbit
  of $\xb$ and $\yb$ as in the top left figure, or be the orbit of $\xb$ and
  several lifts of $y$, as in the top right figure.}
  \label{fig:blockrelation}
\end{figure}

The goal is to apply Solovay's Lemma~\ref{lemma:Solovay} to the collection of
time intervals determined by the blocks (namely, the intervals of time for which
$\xb \ub_2^r$ belong to some block in the collection), in order to deduce that
the orbits of $\xb$ and $\yb$ stay close for all future times.  The main
difficulty is to check the first condition in the statement of
Lemma~\ref{lemma:Solovay}.  We will argue as follows: if two blocks \emph{are
not in relation}, then this condition is automatically satisfied, see Lemma
\ref{lemma:blocks_not_related} below.  On the other hand, if two blocks are
\emph{related}, this may not necessarily be true.  In this case, we \lq\lq
force\rq\rq\ the condition by \lq\lq gluing\rq\rq\ together related blocks which
are too close to each other.  We will then show that the \emph{new} collection
of \emph{superblocks} we obtain in this way satisfies the assumptions of
Lemma~\ref{lemma:Solovay}, and hence the proof of the Basic Lemma will follow.

\subsubsection{Solovay's Lemma}
In the proof of Ratner's Basic Lemma~\ref{lemma:basic}, we will use the
following result about a collection of intervals on the real line.  This result
first appeared in~\cite{Ratner:Cartesian}, the proof being attributed to
Solovay.

For an interval $I\subset \R$, we denote its length by $\abs{I}$.  Let
$\mathcal J= \{J_1,\dotsc,J_n\}$ be a collection of disjoint subintervals of
$I$. The \emph{density} of the collection $\mathcal J$ in the interval $I$ is
the ratio of the Lebesgue measures of $\bigcup_{i=1}^n J_i$ and $I$.  The
\emph{distance} of two intervals $I$ and $J$ is the number
$d(I,J)=\inf_{x\in I, y\in J}\abs{x-y}$.

\begin{lemma}[Solovay]\label{lemma:Solovay}
  Given a $b>0$ there exists a real number $\theta=\theta(b)\in\, (0,1)$ such
  that, for any collection $\mathcal J=\{J_1,\dotsc,J_n\}$ of disjoint
  subintervals of an interval $I\subset \R$, if the following conditions hold
  true
  \begin{enumerate}
  \item $d(J',J'')\geq \min\{\abs{J'}, \abs{J''}\}^{1+b}$, for all distinct
    $J',J''\in \mathcal J$;
  \item the density of the collection $\mathcal J$ in $I$ is larger than
    $1-\theta$;
  \item $\abs{C}\geq 1$ for each connected component $C$ of
    $I\setminus \bigcup_{i=1}^n J_i $;
  \end{enumerate}
  then there exists an interval $J\in \mathcal J$ with $\abs{J}\geq 3/4 |I|$.
\end{lemma}

\subsubsection{}
The proof of Lemma~\ref{lemma:basic} will be divided in several steps and
concluded on page~\pageref{sec_proof_basic}. For the rest of the section, again
to lighten the notation, we will again drop the indices 2 that refer to objects
in the target space.

\subsubsection{}
Fix $\omega>0$ and let $\theta \in (0,1)$ be given by Solovay's Lemma
(\cref{lemma:Solovay}) for the choice of $b$ in \eqref{eq:defalpha}.  Let now
$m_0, \rho, T_0, r_0$, and $\Good$ be given by \cref{prop:good_set}, and recall
that, if $d(\gb , \eb) \le \rho$ then $l_\eps(\gb)>m$. Observe also that if
$ d( \xb \ub^s , \yb \ub^{\tau(s)} ) \le \rho$ then there is
$t\in [\tau(s)-\rho,\tau(s)+\rho ]$ with
$ d( \xb \ub^s , \yb \ub^{t} ) \le \rho$ and
$\yb \ub^{t}\in W( \xb \ub^s, 2\rho) \subset W( \xb \ub^s, \eps)$.

Let now $x,y \in M$ be given as in the statement of the Basic Lemma.  We are
going to follow the $U$-orbits of $x$ and $y$ on $M$, and construct
$(\rho, \eps)$-\emph{blocks} for appropriate \emph{lifts} of $x$ and $y$ to $G$.
Again, since $\rho$ and $\eps$ are fixed, here and henceforth we will simply say
blocks.

\subsubsection{}\label{sssec:blocks}
Let $x$ and $y$ be as in the statement of the Basic Lemma. Since $0\in A_s$, and
$\tau(0)=0$ we have that $d(x,y)<\rho$. Let $\xb_1 $ and $ \yb_1$ be lifts of
$x_1$ and $y_1$ with $d( \xb_1 , \yb_1 )\le \rho$. Since the point $x$ satisfies
the Injectivity Condition $\IC(\rho,m)$, given $\xb_1$, the lift $\yb_1$ is
unique.

Set $s_1=0$. We define the interval $J_1=[0, \bar s_1]= [s_1, \bar s_1]$ where
\[
  \bar s_1 = \sup \big\{ r\in A_s\cap [0,s] \mid \enspace r\le
  l_\eps(\xb_{1}^{-1}\yb_{1}), \enspace d( \xb_{1} \,\ub^{r}, \yb_{1}
  \,\ub^{\tau(r)})\le \rho \big\}.
\]
Notice that $J_1$ is a $(\rho, \eps)$-\emph{block} for $\xb_1$ and $\yb_1$, in
particular the points $\xb_1\ub^{\bar s_1}$ and $\yb_1 \ub^{\tau(\bar s_1)}$ are
at distance not larger than $\rho$.  We stress that, however, it is not true
that $d(\xb_{1} \,\ub^{r}, \yb_{1} \,\ub^{\tau(r)})\le \rho$ for all times
$r\in J_1$.

Assuming that we have defined, for $j>1$,
$J_1=[s_1, \bar s_1] , \dotsc, J_{j-1}=[s_{j-1}, \bar s_{j-1}]$ and
$\xb_{1}, \yb_{1}, \dots, \xb_{j-1} ,\yb_{j-1}$, we let
\[
  s_{j} = \inf\{ r> \bar s_{j-1}\mid r\in A_s\}.
\]
As $s_{j}\in A_s$, setting
\[
  x_{j} = x_1 \ub^{s_{j}}, \quad y_{j}= y_1\ub^{\tau(s_{j})},
\]
we have $d( x_{j} , y_{j} ) \le\rho$.  Thus if $ \xb_{j} := \xb_1 \ub^{s_{j}}$
we may define $\yb_{j}$, in a unique manner, asking that
\begin{equation}\label{eq:blocks:3}
  \qquad
  \pi_2(\yb_{j})= y_1 \ub^{\tau(s_{j})} , \quad d( \xb_{j} , \yb_{j} ) \le \rho.
\end{equation}
Finally, we define
\begin{equation}\label{eq:blocks:4}
  \bar s_{j} := \sup \big\{ r\in A_s\cap [s_j,s] \mid \,
  r-s_j \le l_\eps(\xb_{j}^{-1}\yb_{j}),  \,
  d( \xb_{j} \,\ub^{r-s_j}, \yb_{j} \,\ub^{\tau(r)-\tau(s_j)})
  \le\rho\big\}.
\end{equation}
This means that we have lifted the orbit $x_1\ub^r$ to an orbit
$ \xb_1 \ub^{r}$, and we have chosen, at the times $s_j$, the unique lift
$ \yb_{j}$ of $y_1 \ub^{\tau(s_j)}$ satisfying $d(\xb_j,\yb_j)\le\rho$. The
uniqueness of this lift is justified by the fact that $s_j\in A_s$ and therefore
the point $\xb_j=\xb_1 \ub^{s_j}$ satisfies the injectivity condition
$\IC(\rho,m)$.

With these definitions we have
\begin{equation}\label{eq:blocks:6}
  \forall i \le j:\quad\xb_{j} = \xb_i \ub^{s_{j}- s_{i}} ,\quad  x_{j} = \pi_{2}(\xb_i).
\end{equation}
Moreover, setting
\[
  t_i= \tau(s_i), \quad \bar t_i= \tau(\bar s_i),
\]
we also have
\begin{equation}\label{eq:02_basic_lemma:1}
  x_{j} = x_i \, \ub^{s_{j}- s_{i}} \quad y_{j} = y_i \, \ub^{t_{j}- t_{i}}.
\end{equation}
We stress that it \emph{does not} necessarily hold that
$\yb_{j} = \yb_i \,\ub^{t_{j}- t_{i}}$.  We have therefore a collection
$\cJ= ( J_1, \dots, J_j, \dots)$ of closed intervals $J_j=[s_j,\bar s_j]$ with
disjoint interiors and whose union covers $A_s$.  These intervals $J_i$ are
blocks for the lifts $\xb_i$, $\yb_i$ constructed above of the points
$x \ub^{s_i}, y\ub^{\tau(s_i)} \in M$.

\subsubsection{}
Our previous remark leads to the following definition.  Given two blocks $J_j$
and $J_j$ we say that
\begin{equation}\label{eq:blocks:2}
  J_j\sim J_k\iff \yb_{k} = \yb_j \,\ub^{t_{k}- t_{j}}.
\end{equation}

\subsubsection{}
Suppose that for some $j<k$ we have $J_j\sim J_k $. Then
\begin{equation}\label{eq:02_basic_lemma:2}
  s_k-s_j \ge l_\eps({\yb_{j}}^{-1} \xb_j).
\end{equation}
If not, considering that $d( \xb_{k} , \yb_{k} )\le \rho$ and that
$\yb_{k} = \yb_j \,\ub^{t_{k}- t_{j}}$ we would have $\bar s_j \ge s_k$, by the
definition of $\bar s_j$.

The above inequality implies
\begin{equation}\label{eq:02_basic_lemma:4}
  s_k-s_j > m,
\end{equation}
by the choice of $\eps$.

\subsubsection{}
If $J_j\not \sim J_k$, with $j<k$, we claim that we have
\begin{equation}\label{eq:02_basic_lemma:3}
  \max (s_k - \bar s_j,t_k - \bar t_j )> m.
\end{equation}
Assume by contradiction that $ \max (s_k - \bar s_j,t_k - \bar t_j )\le m$, and
let $\gamma \in \Gamma_2 \setminus \{\eb\}$ such that
$ \yb_k= \gamma \, \yb_j \, \ub^{t_k - t_j}$.  Let us notice that, by
definition, $\bar s_j-s_j \in A_s$ and thus the point
$\xb_{j} \ub^{\bar s_j-s_j}$ satisfies the injectivity condition $\IC(\rho, m)$.
Applying this property to this point and noticing that we have
$ d( \xb_{j}\,\ub^{\bar s_j-s_j} , \yb_j \, \ub^{\bar t_j - t_j})\le \rho$ we
get
\begin{equation*}
  \begin{split}
    \rho \ge d( \xb_{k}, \yb_{k}) &= d( \xb_{j} \,\ub^{s_k-s_j}, \gamma \yb_j
    \ub^{t_k - t_j}) \\ &= d( \xb_{j}\,\ub^{\bar s_j-s_j} \, \ub^{s_k- \bar
      s_j}, \gamma \, \yb_j \, \ub^{\bar t_j - t_j} \, \ub^{t_k- \bar t_j}) >
    \rho,
  \end{split}
\end{equation*}
which is our desired contradiction. As a consequence of the
formula~\eqref{eq:02_basic_lemma:3} and of the fact $\bar s_j \in A_s$ we have
\begin{equation}\label{eq:02_basic_lemma:5}
  \abs{(t_k - \bar t_j) - (s_k - \bar s_j)} \le 4 (s_k - \bar s_j)^{1-\eta}
\end{equation}

\begin{remark}
  By the \eqref{eq:02_basic_lemma:4} and~\eqref{eq:02_basic_lemma:3} the number
  of blocks in the interval $[0,s]$ is finite.
\end{remark}

\subsubsection{Construction of the superblocks.}

The above constructions yields a collection of disjoint blocks
$\cJ=(J_1, \dotsc, J_n)$.  We can construct a new collection of intervals
$\cJ'=(J'_1, \dotsc, J'_k)$ that we will call superblocks and later show that
this collection satisfies the assumptions of Solovay's
Lemma~\ref{lemma:Solovay}.

We will always assume that a collection of intervals is \emph{ordered} in the
obvious way.  Moreover, given two intervals $J'$ and $J''$ we will denote with
$\conv(J', J'')$ their convex hull.

To construct $\mathcal J'$ from $\mathcal J$ we will do the following, see
\cref{fig:superblocks}:

\begin{itemize}
\item If there is no interval $J_j> J_1$ with $J_1\sim J_j$ then we set
  $J'_1= J_1$ and we next consider $J_2$.
\item If $J_{j_1}, J_{j_2} , \dots$ with $1< j_1 < j_2 < \dots$ are intervals
  equivalent to $J_1$ we define $J_1'$ as follows:
  \begin{itemize}
  \item Let $S_1, \dots, S_D$ (ordered in the obvious way) be the connected
    components of $\sublevone$ defined as in \S\ref{Sec:blocks_def_lg}. By Lemma
    \ref{lemma:same_sub_level_sets}, any $J_{j_i}$ is contained in one of these
    $S_k$'s.
  \item Let $k_{\bar i}$ be maximal with the property that $S_{k_{\bar i}}$
    contains one of the $J_{j_i}$ and the intervals
    $S_1=S_{k_1}, \dots, S_{k_{\bar i}}$ are $b$-close.  Let $j_{\bar i}$ be
    maximal with the property that $J_{j_{\bar i}} \subset S_{k_{\bar i}}$ (note
    that there could be several blocks inside a single interval $S_k$)
  \item We define
    \[
      J'_1 = \conv(J_1, J_{j_{\bar i}}) = \conv(J_1, J_2, \dots, J_{j_{\bar
          i}}).
    \]
  \item We restart the procedure from $J_{j_{\bar i} +1}$.
  \end{itemize}
\end{itemize}

\begin{figure}[tb]
  \centering \includegraphics{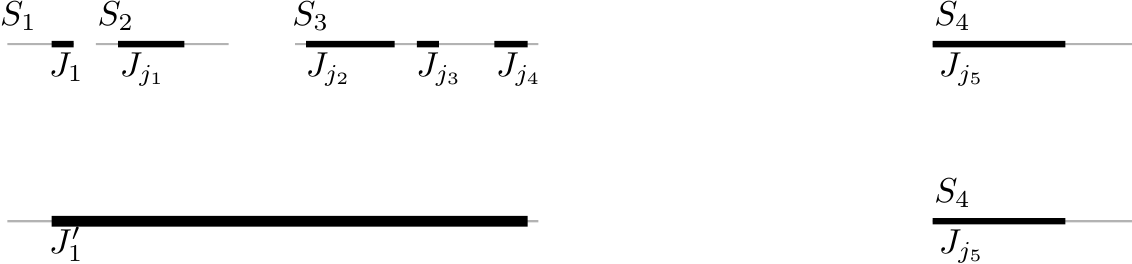}
  \caption{The blocks $J_1, J_{j_1}, \dotsc, J_{j_5}$, in thick black, are
    equivalent. The levels $S_1$, $S_2$, and $S_3$, in grey, are
    $b$-close. Hence $k_{\overline{i}}=3$, and $j_{\overline{i}}=j_4$. We define
    the superblock $J'_1 = \conv(J_1, J_{j_4})$.}
  \label{fig:superblocks}
\end{figure}

Given any ordered list of blocks $\cJ=(J_1, \dotsc, J_n)$ and an equivalence
relation among blocks the above algorithm produces a ordered list of closed
intervals $\cJ'=(J'_1, \dotsc, J'_k)$ such that each $J'\in \mathcal J'$ is the
convex hull of some intervals in~$\mathcal J$.

\sloppy We say that \emph{two intervals $J', J''$ are $b$-separated} if
$ d(J'_j,J'_k) \ge \min\{ \abs{J'},\abs{J''}\}^{1+b}$. Applying the algorithm
above to the list of blocks $\cJ=(J_1, \dotsc, J_n)$ defined in
\S~\ref{sssec:blocks}, we obtain
\begin{lemma}\label{lemma:superblocksalgorithm}
  The collection of super-blocks $\cJ'=(J'_1, \dotsc, J'_k)$ satisfies the
  following:
  \begin{enumerate}
  \item it covers the set $A_s$:
    $\bigcup_{J'\in \mathcal J'} \abs{J'} \supset \bigcup_{J\in \mathcal J}
    \abs{J}\supset A_s$.
  \item Any super-block $J'= \conv(J_i, J_j)$ can be written as the convex hull
    of $h \leq D$ intervals contained in $\sublevi$ which are $b$-close.
  \item Any two equivalent super-blocks are $b$-separated.
  \item $ d(J'_j,J'_k)\ge m/5 >1$ for any $J'_j\neq J'_k\in \cJ'$.
  \end{enumerate}
\end{lemma}
\begin{proof}
  For the first part, it is enough to notice that every $J' \in \cJ'$ is defined
  as a convex hull of elements in $\cJ$.
  
  For the first superblock, we have
  \[
    J_1' = \conv(S_1 \cap J_1, S_2, \dots, S_{k_{\bar i} -1}, S_{k_{\bar i}}
    \cap J_{j_{\bar i}}),
  \]
  with the convention that the term on the right is simply
  $\conv(S_1 \cap J_1, S_1 \cap J_{j_{\bar i}})$ if $k_{\bar i} = 1$.  In any
  case, we can write $J_1'$ as a convex hull of at most $D$ intervals contained
  in $\sublevone$, which are $b$-close by the definition of the algorithm.  This
  proves 2 for the first superblock, the proof for the others is the same.

  Part 3 follows from the definition of the algorithm, in particular from the
  definition of ${k_{\bar i}}$.

  For the last statement, part 4, we reason as follows.  If $J'_j\sim J'_k$,
  then the estimate follows from \eqref{eq:02_basic_lemma:4} and the proven
  estimate $ d(J'_j,J'_k)\ge \abs{J'_j}^{1+b} $. If $J'_j\not \sim J'_k$, then
  the estimate follows from \eqref{eq:02_basic_lemma:3} and
  \eqref{eq:02_basic_lemma:5}.
\end{proof}

Next we prove that also non-equivalent super-blocks are $b$-separated.

\begin{lemma}\label{lemma:blocks_not_related}
  Let $J_p \not \sim J_q$, with $p<q$ be two non-equivalent super-blocks.
  Assume that
  \begin{itemize}
  \item[(i)] $x_p \, \ub^{\bar s_p-s_p}$ satisfies $\IC(\rho, m)$,
  \item[(ii)] $x_p$ and $x_q$ satisfy $\FBR(T_0, c(b), r_0)$.
  \end{itemize}
  Then, $J_p$ and $J_q$ are \emph{$b$-separated}, namely
  \[
    d(J_p, J_q) \geq \min\{ |J_p|, |J_q|\}^{1+b}.
  \]
\end{lemma}
\begin{proof}
  Let us start by recalling that, by definition,
  \[
    c(b) = \frac{1+ \eta}{1+\eta-b}, \quad T_0 \leq \frac{1}{2}\log
    \frac{m}{\Ctau}, \quad r_0 > 20 \kappa \sqrt{\Ctau} m^{- \frac{\eta-2b}{12}}
  \]

  Let us call $L= \min\{ |J_p|, |J_q|\}$; we will show that if
  $ d(J_p, J_q) \le L^{1+b}$, then $J_p \sim J_q$.  Assume that this is not the
  case. Then, there exists $\gamma \in \Gamma_2 \setminus \{ \eb\}$ such that
  $\yb_q = \gamma \, \yb_p \, \ub^{t_q-t_p}$.  Note that by the definition of
  $ \bar s_p$ we have
  $d(\xb_p \, \ub^{\bar s_p-s_p}, \yb_p \, \ub^{\bar t_p-t_p}) \le \rho$. The
  fact that the point $x_p \, \ub^{\bar s_p-s_p}$ satisfies the $\IC(\rho,m)$
  injectivity condition implies that either $\bar s_p-s_p> m$ or
  $\bar t_p-t_p> m$. In either case, we conclude that
  $\bar s_p-s_p= d(J_p, J_q)> m/\Ctau $, from which we get
  $L > (m/\Ctau )^{1/(1+b)} > \sqrt{m/\Ctau }$.  Let
  \begin{equation*}
    \begin{split}
      T_0' &= \left(1 + \frac{1}{3} \eta + \frac{1}{3} b \right) \log L >
      \frac{1}{2} \log \frac{m}{\Ctau} \geq T_0 , \quad \text{ and } \\
      \quad c'&= \frac{1 + \frac{1}{6} \eta + \frac{2}{3} b }{1 + \frac{1}{3}
        \eta + \frac{1}{3} b } < \frac{1+ \eta}{1+\eta-b} = c(b) < 1.
    \end{split}
  \end{equation*}
  Since, by assumption, $x_q$ satisfies the Frequently Bounded Radius Condition
  $\FBR(T_0', c', r_0)$, there exists
  $T \in \left[ \left( 1 + \frac{1}{6} \eta + \frac{2}{3} b \right) \log L,
    \left( 1 + \frac{1}{3} \eta + \frac{1}{3} b \right) \log L\right]$ such that
  the injectivity radius at $\xb_q \ab^T$ is bounded below by $r_0$.

  We will show that
  \[
    d(\gamma^{-1} \xb_q \ab^T, \xb_q \ab^T) \leq 20 \kappa \sqrt{\Ctau} m^{-
      \frac{\eta-2b}{12}} < r_0,
  \]
  which is our contradiction.  This will follow immediately once we show that
  \begin{equation}\label{eq:goal_of_lemma_252}
    \begin{split}
      &d(\xb_q \ab^T, \xb_p \, \ub^{\bar s_p-s_p} \ab^T) +
      d(\xb_p \, \ub^{\bar s_p-s_p} \ab^T, \yb_p \, \ub^{\bar t_p-t_p} \ab^T) \\
      &\qquad + d(\yb_p \, \ub^{\bar t_p-t_p} \ab^T, \yb_p \, \ub^{t_q-t_p}
      \ab^T) + d(\yb_q \ab^T, \xb_q \ab^T) \leq 20 \kappa \sqrt{\Ctau} m^{-
        \frac{\eta-2b}{12}},
    \end{split}
  \end{equation}
  since
  $d(\yb_p \, \ub^{t_q-t_p} \ab^T, \gamma^{-1} \xb_q \ab^T) = d(\yb_q \ab^T,
  \xb_q \ab^T)$.  Let us consider each of the four summands above separately.

\begin{figure}[tb]
  \centering \includegraphics{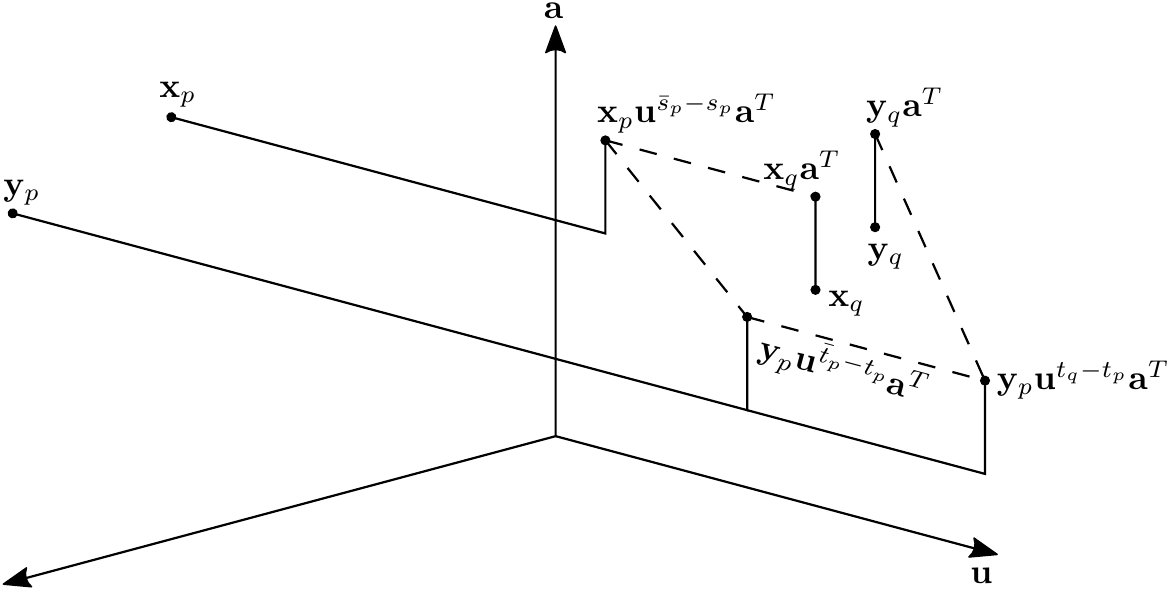}
  \caption{The points involved in the proof of
    equation~\eqref{eq:goal_of_lemma_252}. The dashed lines are the distances we
    are estimating. Distances have exagerated in order to make the picture
    clear.}
  \label{fig:alphasep}
\end{figure}

For the first summand, we have
\begin{equation*}
  \begin{split}
    & d(\xb_q \ab^T, \xb_p \, \ub^{\bar s_p-s_p} \ab^T) = d(\eb,
    \ab^{-T} \ub^{s_q - \bar s_p} \ab^T) = d(\eb,\ub^{e^{-T} d(J_p,J_q)}) \\
    & \qquad = e^{-T} d(J_p,J_q) \le L^{- \left( 1 + \frac{1}{6} \eta +
        \frac{2}{3} b \right) + (1+b)} \le L^{- \frac{\eta-2b}{6}},
  \end{split}
\end{equation*}
which is smaller than $\sqrt{\Ctau} m^{- \frac{\eta-2b}{12}}$.  In a similar
way, we can bound the third summand by
\[
  d(\yb_p \, \ub^{\bar t_p-t_p} \ab^T, \yb_p \, \ub^{t_q-t_p} \ab^T) =
  d(\ub^{e^{-2T} (t_q- \bar t_p)}, \eb),
\]
which, in turn, is less than $5\sqrt{\Ctau} m^{- \frac{\eta-2b}{12}}$, since
either $s_q - \bar s_p >m$ or $t_q - \bar t_p >m$, and hence
$t_q- \bar t_p \le s_q - \bar s_p + 4(s_q - \bar s_p)^{1-\eta} \le 5(s_q - \bar
s_p)$.

It remains to bound the second and fourth term.  Let us start from the latter.
By Lemma \ref{lemma:superblocksalgorithm}, we are precisely in the setting of
Corollary \ref{cor:bounds_on_coeffs}.  Hence, we can apply
Lemma~\ref{lemma:dist_a_t}: since
$T \leq \left( 1 + \frac{1}{3} \eta + \frac{1}{3} b \right) \log L$, we get
\begin{equation*}
  \begin{split}
    d(\yb_q \ab^T, \xb_q \ab^T) \leq 3 \kappa \left(L^{-\frac{1}{2} +
        \frac{5}{6}\eta} + L^{-\frac{\eta}{2}+\frac{1}{3}\eta +
        \frac{1}{3}b}\right) = 6 \kappa L^{-\frac{\eta-2b}{6}} \leq 6 \kappa
    \sqrt{\Ctau}L^{-\frac{\eta-2b}{12}}.
  \end{split}
\end{equation*}
The proof of the desired inequality
\[
  d(\xb_p \, \ub^{\bar s_p-s_p} \ab^T, \yb_p \, \ub^{\bar t_p-t_p} \ab^T) \le 6
  \kappa \sqrt{\Ctau} L^{-\frac{\eta-2b}{12}}
\]
for the remaining summand is completely analogous: it follows again from
Lemma~\ref{lemma:dist_a_t}, by considering the negative parametrizations $s'=-s$
and $t'(s)= -t(-s)$, $s \in [-L,0]$. This completes the proof of
\eqref{eq:goal_of_lemma_252} and hence yields the result.
\end{proof}

\subsubsection{}\label{sec_proof_basic}
We can finally complete the proof of Ratner's Basic Lemma.

\begin{proof}[Proof of \cref{lemma:basic}]
  We start by recalling that all the choice of the relevant constants were done
  in \S\ref{sec:alpha_and_K_2}.  In particular, we have chosen $b$.  Let
  $\theta=\theta(b)$ given by Solovay's Lemma~\ref{lemma:Solovay}.

  Let $x, y \in M_2$, and let $\lambda$ and $A$, with
  \[
    \frac{\abs{A\cap [0,\lambda]}}{\lambda}>1-\frac{\theta}{8},
  \]
  be as in the statement.  We consider $\xb_1$ and $\yb_1$ lifts of $x$ and $y$
  such that $d(\xb_1,\yb_1)<\rho$.  With these data we build blocks
  $\{ J_1, \dotsc, J_n \}$ as in \S\ref{sec:blocks}.  By applying the
  superblocks algorithm to this collection of blocks we obtain a collection of
  superblocks as in \cref{lemma:superblocksalgorithm}.

  The collection of superblocks we have obtained thus satisfies the assumptions
  of \cref{lemma:Solovay}.  This implies that there exists a superblock
  $J = [s_J, \bar{s_J}]$ of length $\abs{J}\ge 3/4\lambda$.
  \Cref{lemma:superblocksalgorithm} guarantees that we can apply
  \cref{cor:bounds_on_coeffs} to the points $\xb_J$ and $\yb_J$ which are lifts
  to $G$ of the points $x\ub^{s_J}$ and $y\ub^{\tau(s_J)}$ respectively.  We
  obtain that
  \[
    \gb = \xb_J^{-1} \yb_J = \exp \Big( \sum_{\iota, j} c_{j, \iota}
    \bel_{j,\iota} \Big) \, \exp(aA) \, \exp(\bar u \bar U) \, \ub^u
  \]
  satisfies
  \[
    |u|\leq \varepsilon, \quad |a|\leq \kappa \lambda^{-\eta/2}, \quad |\bar
    u|\leq \kappa \lambda^{-1-\eta/2}, \quad |c_{\iota,j}|\leq \kappa
    \varepsilon \lambda^{-j+\eta/2}.
  \]
  which proves the first part

  Let us now show the second claim.  Since, by assumption, there exists
  $\lambda_0 \geq m$ such that we have that
  $\abs{A \cap [0,\lambda]} \geq (1-\theta/8)\lambda$ for all
  $\lambda \geq \lambda_0$, we can find a sequence $\lambda_n \to \infty$ such
  that $\lambda_n < \lambda_{n+1} < 8\lambda_n/7$.  For all such $\lambda_n$ we
  can repeat the above construction and obtain an interval
  $J_{n}=[s_n,\bar{s_n}]$ having length $\abs{J_{n}}\geq 3/4 \lambda_{n}$.  By
  our choice of $\lambda_{n}$ we have
  \[
    J_{n}\cap J_{n+1}\neq\emptyset,
  \]
  and hence $J_{n}\subset J_{n+1}$ for all $n$; more precisely,
  $J_n=[s_0,\bar{s_n}]$ where $s_0 \leq \lambda_0$.  Applying the same bounds
  given by \cref{cor:bounds_on_coeffs} we have, for
  $\gb=(\xb\ub^{s_0})^{-1}\yb\ub^{\tau(s_0)}$
  \[
    \abs{u}\le \eps, \quad \abs{a} \le \kappa \lambda_{n}^{-\eta/2} \quad
    \abs{\bar u}\le \kappa \lambda_{n}^{-1-\eta/2}, \quad \text{ and } \quad
    \abs{c_{j,\iota}} \le \kappa \epsilon \lambda_{n}^{-j +\eta/2}.
  \]
  for all $n$. This implies that $\abs{a}$, $\abs{\bar{u}}$ and
  $\abs{c_{j,\iota}}$ are $0$, if $j\neq 0$. Thus, there exists $\gb$, commuting
  with $\ub$ such that $\yb=\xb\gb$ as we wanted.
\end{proof}


%% file: 05_normaliser.tex

\section{Normaliser}\label{sec:centraliser}

\subsubsection{}\label{sec:centraliser_1}
In this section we will apply the Basic Lemma~\ref{lemma:basic}. Before we state
formally the main result of this section, let us recall all the constants we
have defined so far.

The two functions $\alpha_1$ and $\alpha_2$ defining the time-changes
$\widetilde{\phi}_1^t$ and $\widetilde{\phi}_2^t$, define the constant $\Ctau>0$
in~\cref{ass:1}. The strong spectral gap yields $\eta\in(0,1/4)$
in~\cref{thm:mixing_condition}. From equation~\eqref{eq:defalpha} we obtain the
constants $b$ and $c(b)$. Finally $\theta=\theta(b)\in(0,1)$ is the one given by
Solovay's Lemma~\ref{lemma:Solovay} and used in the Basic
Lemma~\ref{lemma:basic}.  Having recalled all the important constants, we pick
an $\omega>0$ satisfying
\[
  \omega< \frac{\theta}{50\Ctau^4}.
\]
Using \cref{prop:good_set} we obtain a set $\Good$ and constants $m_0, \rho$.
Let $\delta>0$ be the uniform continuity constant of $\psi$ on $\Good$ so that,
if $x,y\in\Good$, then
\[
  d(x,y)<\delta \implies d(\psi(x),\psi(y))<\rho.
\]
Finally, we choose $\delta' < \rho$ and we recall that $m = 2\Ctau^4 m_0>m_0$.

For any $n \in \N$, there exist $s_n >0$ and a set $K_{\ub_1}(n) \subset M_1$ of
measure greater than $1-2^{-n}$ such that for all $x \in K_{\ub_1}(n)$ and all
$s\geq s_n$
\[
  \abs{ \{ t \in [0,s] : x \ub_1^t \in \Good \}} \geq \left(
    1-\frac{\theta}{40\Ctau^4} \right) s.
\]
We can take $s_n > m \omega_0^{-1}$.  This section, \S\ref{sec:centraliser}, is
devoted to the proof of the following result.

\begin{proposition}\label{thm:centraliser}\def\gbtilde{\gb_2}
  Let $\gb_1\in G_1$, $\gbtilde \in G_2$, with $d(\eb_1, \gb_1) < \delta$ and
  $d(\eb_2, \gbtilde)< \delta'$, satisfy the commuting relations
  $\gb_1 \ub_1^t \gb_1^{-1} = \ub_1^{c^{-1}t}$ and
  $\gbtilde \ub_2^t \gbtilde^{-1} = \ub_2^{c^{-1}t}$ for all $t \in \R$.  Fix
  $n \in \N$, and let $x\in K_{\ub_1}(n)$, $y=x \gb \in K_{\ub_1}(n)$.  Then,
  for every $s\geq s_n$ there exist
  \begin{itemize}
  \item points $x_0 = x \ub_1^{t_0}$, $y_0 = y \ub_1^{ct_0} = x_0 \gb$, with
    $0 \leq t_0\leq 4\omega s$,
  \item an interval $[0,\lambda_s] \subset [0,s]$, with
    $\lambda_s \geq \Ctau^{-2}\left( 1-\frac{\theta}{16\Ctau^4}\right)s$,
  \item a compact subset $A_s \subset [0,\lambda_s]$,
  \item a strictly increasing Lipschitz function
    $\tau=\tau_{x_0,y_0} \colon [0,\lambda_s] \to \R_{>0}$ with $\tau(0)=0$,
  \end{itemize}
  such that, if we set $x_0' = \psi(x_0)$, $y_0' = \psi(y_0)$, and
  $\overline{x}_0' = \psi(y_0)\gbtilde^{-1}$, then
  \begin{itemize}
  \item $0,\lambda_s \in A_s$,
  \item $|A_s| > \bigl(1-\frac{\theta}{8}\bigr) \lambda_s$,
  \item $x_0' \ub_2^r$ and $\overline{x}_0' \ub_2^{\tau(r)}$ belong to
    $\psi(\Good)$ for all $r \in A_s$ and satisfy
    \[
      d\left( x_0' \ub_2^r, \overline{x}_0' \ub_2^{\tau(r)} \right) < \rho,
    \]
  \item if $r',r \in A_s$ and $r'-r>m$ or $\tau(r')-\tau(r) >m$, then
    \[
      \left\lvert \tau(r')-\tau(r) -( r'-r) \right\rvert \leq 4(r' -
      r)^{1-\eta}.
    \]
  \end{itemize}
\end{proposition}

\subsubsection{}
Let $\gb_1$, $x$, and $y$ be as in the assumptions of
Proposition~\ref{thm:centraliser}, in particular $d(x,y)<\delta$.  We can assume
that $\frac{1}{2} \leq c \leq 2$.

Clearly,
\[
  y \ub_1^{cs} = x \gb_1 \ub_1^{cs} = x \ub_1^s \gb_1,
\]
and
\[
  \psi(x \ub_1^s ) = \psi(x )\ub_2^{z(x,s)},\qquad \psi(y \ub_1^{cs}) = \psi(y
  )\ub_2^{z(y,cs)}.
\]
Define
\[
  \forall s\ge 0 : \quad B_s:=\{ r\in[0,s]: x \ub_1^{r}\in \Good, y \ub_1^{cr}
  \in \Good \},
\]
which is a compact subset of $[0,s]$.  By definition of $K_{\ub_1}(n)$, the
estimate on the measure of $\Good$ in Proposition~\ref{prop:good_set}, and our
choice of $\omega$, since $c \leq 2$, we have
\begin{equation}\label{eq:scheme_for_centraliser:1}
  \forall s>2 s_n: \quad \abs{B_s} \ge \left( 1-\frac{\theta}{16\Ctau^4}\right)s.
\end{equation}
Let us assume now that $d(\eb, \gb) = d(x,y) < \delta$, and consequently
$d(x\ub_1^s, y \ub_1^{cs}) < \delta$ for all $s \in \R$. Then, by
Proposition~\ref{prop:good_set}, we have
\begin{equation}\label{eq:scheme_for_normaliser:1}
  d \left( \psi(x) \ub_2^{z(x,s)}, \psi(y) \ub_2^{z(y,cs)}\right) < \rho,
  \text{\ \ \ for all\ }s \in B_s.
\end{equation}

\begin{lemma}\label{lemma:z_for_centraliser}
  Let $t \in B_s$. Assume that there exists $t<t' \leq s$ such that
  \[
    z(x,t')-z(x,t) > cm \text{\ \ \ or\ \ \ } z(y,t')-z(y,t) > cm.
  \]
  Then,
  \[
    \left\lvert c^{-1} \big( z(y,ct')-z(y,ct) \big) - ( z(x,t')-z(x,t))
    \right\rvert \leq 4 \left( z(x,t')-z(x,t) \right)^{1-\eta}.
  \]
\end{lemma}
\begin{proof}
  Using Lemma~\ref{lemma:z_cocycle} and the definition of $m= 2\Ctau^4 m_0$,
  under our assumption it is easy to see that
  \[
    z(x,t')-z(x,t) > m_0 \text{\ \ \ and\ \ \ } z(y,t')-z(y,t) > m_0 \text{\ \ \
      and\ \ \ } t'-t > m_0.
  \]
  Recall the definition $z(x,t) := w_2(\psi(x),\xi_1(x, t))$, of $w_2$ and of
  $\xi_1$.  By Lemma~\ref{lemma:z_cocycle},
  \begin{equation*}
    \begin{split}
      &z(x,t') - z(x,t) -(\xi_1(x, t') - \xi_1(x, t)) \\
      & \quad = z(x,t') - z(x,t) - \int_{w_2(\psi(x),\xi_1(x, t))}^{w_2(\psi(x),\xi_1(x, t'))} \alpha_2(\psi(x) \ub_2^s) \D s \\
      & \quad = \int_{z(x,t)}^{z(x,t')} (1-\alpha_2)(\psi(x) \ub_2^s) \D s =
      \int_0^{z(x \ub_1^t,t'- t)} (1-\alpha_2)(\psi(x \ub_1^t) \ub_2^s) \D s.
    \end{split}
  \end{equation*}
  Since $x \ub_1^t \in \Good$ and $z(x,t')-z(x,t) > m_0$, by
  \cref{thm:mixing_condition}
  \[
    \abs{z(x,t') - z(x,t) -(\xi_1(x, t') - \xi_1(x, t))} \leq \frac{1}{\Ctau^4}
    (z(x \ub_1^t,t'- t))^{1-\eta} \leq \frac{1}{\Ctau^2}(t'-t)^{1-\eta}.
  \]
  Similarly,
  \begin{equation*}
    \begin{split}
      &\left\lvert \xi_1(x,t') - \xi_1(x,t) -(t' - t) \right\rvert \\
      & \quad =\left\lvert \xi_1(x,t') - \xi_1(x,t) -(w_1(x, \xi_1(x,t')) - w_1(x, \xi_1(x,t))) \right\rvert \\
      & \quad =\left\lvert \int_{w_1(x, \xi_1(x,t))}^{w_1(x, \xi_1(x,t'))}
        (1-\alpha_1)(x \ub_1^s) \D s \right\rvert =\left\lvert \int_{t}^{t'}
        (1-\alpha_1)(x \ub_1^s) \D s \right\rvert \leq \frac{1}{\Ctau^4} ( t' -
      t)^{1-\eta}.
    \end{split}
  \end{equation*}
  Therefore,
  \[
    \left\lvert z(x,t') - z(x,t) -(t' - t) \right\rvert \leq \frac{2}{\Ctau^2}(
    t' - t)^{1-\eta}.
  \]
  We can apply the same reasoning for $y$ instead of $x$, yielding
  \[
    \left\lvert c^{-1} (z(y,ct') - z(y,ct) ) -(t' - t) \right\rvert \leq
    \frac{2}{\Ctau^2} ( t' -t)^{1-\eta}.
  \]
  We deduce that
  \[
    \left\lvert c^{-1} (z(y,ct') - z(y,ct) ) -(z(x,t') - z(x,t)) \right\rvert
    \leq \frac{4}{\Ctau^2}( t' - t)^{1-\eta}.
  \]
  The rough bounds on $z$ and the cocycle relation in Lemma
  \ref{lemma:z_cocycle} conclude the proof.
\end{proof}

\subsubsection{Proof of \cref{thm:centraliser}}

By the compactness of $B_s$, the points
\[
  t_0= \inf B_s\quad \text{ and }\quad t_1= \sup B_s
\]
are in $B_s$.  Let
\[
  t'_0= z(x,t_0),\quad t'_1= z(x,t_1)\quad\text{ and }\quad \lambda_s=
  t'_1-t'_0.
\]
Set
\[
  B'_s = z(x, B_s) \subset [t'_0,t'_1]
\]
From \eqref{eq:scheme_for_centraliser:1} we have
\[
  t_0 < \frac{\theta}{16\Ctau^4} s ,\qquad
  \left(1-\frac{\theta}{16\Ctau^4}\right) s \le t_1-t_0 \le s.
\]
By the uniform Lipschitz estimate on $z(x, \cdot)$, this yields
\[
  \Ctau^{-2}\left(1-\frac{\theta}{16\Ctau^4}\right)s \le \lambda_s =t'_1-t'_0
  \le \Ctau^2 s,
\]
and
\[
  \frac { \abs{B'_s} }{\lambda_s} \ge 1-\frac{\theta}{16}.
\]
We let
\[
  A'_s = \{t'_0\} \cup (B_s'\cap [t'_0+m, t'_1]), \quad A_s = A'_s-t'_0.
\]
Then, since $s \geq s_n$, if $n$ is sufficiently large, we have
\[
  \frac {\abs{A'_s}}{\lambda_s} \ge 1-\frac{\theta}{8}.
\]

Let
\begin{equation*}
  \begin{split}
    x_0 = x \ub_1^{t_0}, & \quad y_0= y \ub_1^{ct_0}= x_0 \gb \\
    x'_0 = \psi(x_0)=\psi(x) \ub_2^{z(x,t_0)} =\psi(x) \ub_2^{s_0}, & \quad
    y'_0= \psi(y_0) =\psi(y) \ub_2^{z(y,ct_0)}.
  \end{split}
\end{equation*}
By Lemma \ref{lemma:z_cocycle}, we have
\[
  \psi (x_0 \ub_1^t)= \psi (x\ub_1^{t+t_0})= \psi (x)\ub_2^{z(x, t+t_0)}= \psi
  (x_0)\ub_2^{z(x, t+t_0)-z(x, t_0) } = \psi (x_0)\ub_2^{z(x_0, t) },
\]
and similarly
\[
  \psi (y_0\ub_1^{ct})= \psi (y_0)\ub_2^{z(y_0, ct) }.
\]
We know from \eqref{eq:scheme_for_normaliser:1} that the $U_2$-orbit of $y_0'$
shadows the $U_2$-orbit of $x_0'$ at times rescaled by $c$. We now define a
point $\overline{x}_0'$ whose $U_2$-orbit shadows the $U_2$-orbit of $x_0'$ at
synchronous times.

Let $\gbtilde$, with $d(\eb, \gbtilde) < \delta'$, be as in the assumptions: an
element in the normaliser of $U_2$ such that
$\gbtilde \ub_2^s \gbtilde^{-1} = \ub_2^{c^{-1}s}$. We set
\[
  \bar x_0' = \psi(y_0) \gbtilde^{-1}.
\]
Then
\[
  \bar x_0' \ub_2^t = \Big ( \psi(y_0) \ub_2^{ct} \Big ) \gbtilde^{-1}, \quad
  \forall t \in \R,
\]
and therefore,
\begin{equation}\label{eq:scheme_for_normaliser:2}
  d (\bar x_0' \ub_2^t,\psi(y_0) \ub_2^{ct} ) < \delta', \quad \forall t\in \R.
\end{equation}

Define
\[
  \tau(r) = \tau_{x_0,y_0}(r) = c^{-1}z(y_0, c \eta(x_0,r)),
\]
where $\eta(x_0, \cdot)$ is the inverse of $z(x_0, \cdot)$.  It follows from the
definition of $A_s$ and Proposition \ref{prop:good_set} that for $r \in A_s$ we
have
\begin{equation*}
  \begin{split}
    &x_0' \ub_2^r\in \psi(\Good), \quad y_0'\ub_2^{\tau_{x_0,y_0}(r)}\in \psi(\Good), \\
    &\text{and } \quad d(x_0' \ub_2^r, \bar x_0'\ub_2^{\tau_{x_0,y_0}(r)}) <
    \rho + \delta' < 2 \rho,
  \end{split}
\end{equation*}
where we used \eqref{eq:scheme_for_normaliser:1} and
\eqref{eq:scheme_for_normaliser:2}.  In particular
\[
  d(x_0', \bar x_0') < \rho.
\]
since $0\in A_s$.  Suppose
\[
  r'-r > m, \qquad\text{or} \qquad \tau_{x_0,y_0}(r') - \tau_{x_0,y_0}(r)>m.
\]
These conditions are equivalent to $z(x,t')-z(x,t) > m$ or
$z(y,ct')-z(y,ct) > m$. If furthermore $r\in A_s$, then Lemma
\ref{lemma:z_for_centraliser} applies and we have
\begin{equation}
  \label{cond_separ2}
  \abs[\Big]{\big(\tau_{x_0,y_0}(r') -  \tau_{x_0,y_0}(r) \big) -\big(r'-r \big) } \\
  \le  4
  \big(r'-r\big)^{1-\eta}.
\end{equation}
The proof is then complete.\qed

\subsection{Action of the conjugacy on the normaliser}

Thanks to the Basic Lemma, we can now show that the isomorphism $\psi$ maps the
foliation with leaves tangent to the normaliser of $U_1$ to the one with leaves
tangent to the normaliser of $U_2$.  More precisely, we prove the following
proposition.

\begin{proposition}\label{thm:action_on_normaliser}
  Let $\gb \in G_1$ with $d(\eb, \gb) < \delta$ be such that
  $\gb \ub_1^{ct} = \ub_1^{t} \gb$ for some $c \in (1/2,2)$.  For almost every
  $x \in M_1$ there exists $\Phi(x, \gb) \in G_2$ satisfying
  $ \Phi(x, \gb) \ub_2^{ct}= \ub_2^t \Phi(x, \gb)$ and such that
  \[
    \psi(x \gb) = \psi(x) \Phi(x, \gb).
  \]
  Moreover, for every $y \in \R$, we have
  \[
    \Phi(x\ub_1^t, \gb) \, \Phi(x, \gb)^{-1} = \ub_2^{z(x\gb, ct) - cz(x,t)}.
  \]
  Finally, $\Phi(x, \gb) \exp(\R \ub_2) \in N_{G_2}(\ub_2) / \exp(\R\ub_2)$ is
  constant almost everywhere and for almost all $x\in M_1$, the map
  \[
    \Phi(x, \cdot) \colon N_{G_1}(\ub_1) / \exp(\R\ub_1) \to N_{G_2}(\ub_2) /
    \exp(\R\ub_2)
  \]
  is a group homomorphism.
\end{proposition}

\begin{proof}
  From \S\ref{sec:centraliser_1}, we have that for every $n \in \N$, the measure
  of the set $K_{\ub_1}(n) \cap K_{\ub_1}(n) \gb^{-1}$ is at least $1-2^{-n+1}$.
  Therefore, for almost every $x \in M_1$, there exists $n=n(x) \in \N$ such
  that $x\in K_{\ub_1}(n)$ and $y = x \gb \in K_{\ub_1}(n)$.  Choose
  $\gbtilde \in G$ such that $d(\eb, \gbtilde) < \delta'$ and
  $\gbtilde \ub_2^{ct} = \ub_2^{t} \gbtilde$.  Combining Proposition
  \ref{thm:centraliser} with Lemma \ref{lemma:basic}, we deduce that there
  exists $\gb_0 = \gb_0(x) \in G_2$, depending on $\gb$ and possibly on $x$,
  which commutes with $\ub_2^t$ and with $d(\eb,\gb_0) \leq \varepsilon$ such
  that
  \begin{equation}\label{eq:prop_psi_norm_1}
    \psi(y \ub_1^{c t_0})\gbtilde^{-1} = \psi(x\ub_1^{t_0})\gb_0,
  \end{equation}
  where $0 \leq t_0 \leq s_{n(x)}$.

  We can rewrite \eqref{eq:prop_psi_norm_1} as
  \begin{equation}\label{eq:prop_psi_norm_2}
    \begin{split}
      \psi(x \gb) &= \psi(x) \ub_2^{z(x,t_0)} \gb_0 \gbtilde \ub_2^{-z(y,ct_0)}\\
      &= \psi(x) \gb_0 \gbtilde \ub_2^{cz(x,t_0)-z(y,ct_0)},
    \end{split}
  \end{equation}
  where the element $\gb_0 \gbtilde$ satisfies
  $d(\eb, \gb_0 \gbtilde) \leq \delta'+\rho< 2\rho$ and belongs to the
  normaliser $N_{G_2}(\ub_2)$ of $\ub_2^t$, namely
  $(\gb_0 \gbtilde) \ub_2^{ct} = \ub_2^t (\gb_0 \gbtilde)$.  Equality
  \eqref{eq:prop_psi_norm_2} proves the existence of
  $\Phi(x,\gb) \in N_{G_2}(\ub_2)$ satisfying
  $\psi(x \gb) = \psi(x) \Phi(x, \gb)$.

  We have a well-defined measurable map
  \begin{equation*}
    \begin{split}
      \Phi_{\gb} \colon M_1 &\to N_G( \ub_2 ) / \exp(\R U_2) \\
      x &\mapsto \gb_0(x) \gbtilde \exp(\R U_2).
    \end{split}
  \end{equation*}
  By construction, it is easy to see that
  $\Phi_{\gb}(x \ub_1^t) = \Phi_{\gb}(x)$ for all $t \in \R$.  Since
  $\Phi_{\gb}$ is bounded, by ergodicity of the unipotent flow $\ub_1^t$, the
  map $\Phi_{\gb}$ must be constant.  This proves that
  $\Phi(x, \gb) \exp(\R \ub_2)$ is constant almost everywhere.  Note that the
  map $\Phi$ satisfies
  \[
    \Phi(x, \gb_1 \gb_2) = \Phi(x, \gb_1) \Phi(x\gb_1, \gb_2)
  \]
  for all $\gb_1,\gb_2$ in the normaliser of $\ub_1$.  In particular,
  $\Phi(x,\cdot)$ descends to a group homomorphism between
  $N_{G_1}(\ub_1) / \exp(\R\ub_1)$ and $N_{G_2}(\ub_2) / \exp(\R\ub_2)$.

  Finally, let $t \in \R$ be fixed. Then,
  \begin{equation*}
    \begin{split}
      \psi(x \ub_1^t \gb) &= \psi(x \gb) \ub_2^{z(x\gb, ct)}
      = \psi(x) \Phi(x,\gb) \ub_2^{ z(x\gb, ct)} \\
      &= \psi(x \ub_1^t) \ub_2^{-z(x,t)}\Phi(x,\gb) \ub_2^{z(x\gb, ct)} \\ &=
      \psi(x \ub_1^t) \Phi(x,\gb) \ub_2^{-cz(x,t)+ z(x\gb, ct)}.
    \end{split}
  \end{equation*}
  Since the expression above also equals $\psi(x \ub_1^t) \Phi(x \ub_1^t,\gb)$,
  the proof of the cocycle relation for $\Phi$ is complete.
\end{proof}

Let us fix $x_0 \in M_1$, and denote $\Phi(\gb) = \Phi(x_0,\gb)$.  Since
$\Phi(x,\gb) \exp(\R\ub_2) = \Phi(\gb) \exp(\R\ub_2)$ almost everywhere, for
almost every $x \in M_1$ there exists $\beta(x,\gb) \in \R$ such that
\[
  \Phi(x,\gb) = \Phi(\gb) \ub_2^{\beta(x,\gb)}.
\]
The proof of \cref{thm:action_on_normaliser} implies that
\[
  d(\eb,\Phi(\gb))\leq \varepsilon, \quad \text{ and } \quad \text{ if } x \in
  K_{\ub_1}(n), \text{ then } | \beta(x,\gb) | \leq s_n.
\]
With this notation, from \cref{thm:action_on_normaliser}, we now obtain the
following result about the cohomology of the time-change functions $\alpha_1$
and $\alpha_2$, which is a more general version of
\cref{thm:cohomologycentraliser}.

\begin{corollary}\label{cor:cohomologyofalphas}
  Let $\gb \in G$ with $d(\eb, \gb) < \delta$ be such that
  $\gb \ub_1^{ct} = \ub_1^{t} \gb$ for some $c \in (1/2,2)$.  Then we have
  \[
    \int_0^t \alpha_1(x\ub_1^s\gb)\, \D s - \int_0^{z(x,t)}
    \alpha_2(\psi(x)\ub_2^s\Phi(\gb))\, \D s = f(x\ub_1^t)-f(x),
  \]
  where
  \[
    f(x)=\frac{1}{c}\int_0^{\beta(x, \gb)} \alpha_2(\psi(x)\Phi(\gb)\ub_2^s)\,
    \D s.
  \]
  In particular, if $\alpha_2 =1$, then $\alpha_1(x)$ is cohomologous with
  $\alpha_1(x\mathbf{c})$, where $\mathbf{c}$ belongs to the centraliser of
  $\ub_1$, and the same holds for $\alpha_2$ if $\alpha_1$ is trivial.
\end{corollary}

\begin{proof}
  By a direct computation, recalling the definition of $z(x,t)$, we have
  \[
    \begin{split}
      \int_0^t \alpha_1(x\ub_1^s\gb)\, \D s &=
      \frac{1}{c} \int_0^{ct} \alpha_1(x\gb\ub_1^s)\, \D s\\
      &=\frac{1}{c} \int_0^{z(x\gb,ct)} \alpha_2(\psi(x\gb)\ub_2^s)\, \D s\\
      &= \frac{1}{c} \int_0^{z(x\gb,ct)}
      \alpha_2(\psi(x)\Phi(\gb)\ub_2^{\beta(x,\gb)+s})\, \D s,
    \end{split}
  \]
  which we can write as
  \[
    \frac{1}{c} \int_0^{z(x\gb,ct)+\beta(x,\gb)}
    \alpha_2(\psi(x)\Phi(\gb)\ub_2^s)\, \D s - \frac{1}{c} \int_0^{\beta(x,\gb)}
    \alpha_2(\psi(x)\Phi(\gb)\ub_2^s)\, \D s,
  \]
  using $\beta(x \ub_1^t, \gb) - \beta(x, \gb) = z(x\gb, ct) - c z(x, t)$, we
  have
  \begin{multline*}
    \int_0^t \alpha_1(x\ub_1^s\gb)\, \D s = \\
    \frac{1}{c} \int_0^{cz(x,t)+\beta(x\ub_1^t,\gb)}
    \alpha_2(\psi(x)\Phi(\gb)\ub_2^s)\, \D s - \frac{1}{c} \int_0^{\beta(x,\gb)}
    \alpha_2(\psi(x)\Phi(\gb)\ub_2^s)\, \D s.
  \end{multline*}
  Since
  \[
    \int_0^{z(x,t)} \alpha_2(\psi(x)\ub_2^s\Phi(\gb))\, \D s =
    \frac{1}{c}\int_0^{cz(x,t)} \alpha_2(\psi(x)\Phi(\gb)\ub_2^s)\, \D s,
  \]
  joining these two pieces we obtain the Corollary.
\end{proof}



%% file: 06_conclusion.tex

\section{Proof of the Main Result}\label{sec:conclusion}
In this section we prove our main result, namely~\cref{thm:main}.

\subsection{Joinings and a map $\zeta$}
We recall that a \emph{joining} of the flows $\widetilde{\phi}_1$ and
$\widetilde{\phi}_2$ is a measure $\mu$ on $M_1\times M_2$ which is invariant
under the action of $\widetilde{\phi}_1\times\widetilde{\phi}_2$ and such that,
when projected to any of the two factors $M_i$, coincides with
$\alpha_i\mu_i$. The measure obtaine by projections to the factors are called
\emph{marginals}. For more details, see, e.g.,~\cite{Glasner:Joinings}.

Since $\psi$ is a measurable conjugacy between the two flows, we can construct a
joining, supported on the graph $\graph_\psi = \{(x,\psi(x)), x\in M_1\}$, such
that
\[
  \mu_\psi (B) = \widetilde{\mu_1}(\{x\in M_1: (x,\psi(x))\in B\}) =
  \widetilde{\mu_2}(\{y\in M_2: (\psi^{-1}(y),y)\in B\}).
\]

Let $\phi^t_A = \phi^t_{A_1\times A_2}\colon M_1\times M_2$ the diagonal action
given by $(x,y)\mapsto (x\ab_1^t, y\ab_2^t)$. If we push forward the joining
$\mu_\psi$ by the action of $A^t$ we obtain a new measure whose support is given
by
\[
  \phi^t_A\graph_\psi = \{(x\ab_1^t, \psi(x)\ab_2^t), x\in M_1\} = \{(x,
  \psi_t(x)), x\in M_1\} = \graph_{\psi_t},
\]
where $\psi_t(x) = \psi(x \ab_1^{-t})\ab_2^t$.  The marginals of this measure
are given by
\[
  (p_1 \circ \phi_A^t)_* \mu_\psi = (\tau_1 \circ \phi_{A_1}^{-t}) \mu_1 \qquad
  \text{and} \qquad (p_2 \circ \phi_A^t)_* \mu_\psi = (\tau_2 \circ
  \phi_{A_2}^{-t}) \mu_2,
\]
where $p_i\colon M_1\times M_2 \to M_i$ is the projection on the $i^\text{th}$
coordinate. Hence $\psi_t$ is a measurable conjugacy between the (ergodic) flows
generated by $(\tau_1 \circ \phi_{A_1}^{-t})^{-1}U_1$ and
$(\tau_2 \circ \phi_{A_2}^{-t})^{-1}U_2$.

In this section, to ease the notation, let us call $G = G_1 \times G_2$ and
$M = M_1 \times M_2$.

\begin{lemma}\label{lem:tight}
  The set $\{(\phi_A^{t})_*\mu_\psi : t >0 \}$ is relatively sequentially
  compact in the space of probability measures on $M$ with the topology of weak
  convergence.
\end{lemma}
\begin{proof}
  It is sufficient to prove that $\{(\phi_A^{t})_*\mu_\psi : t >0 \}$ is tight.
  For any given $\varepsilon >0$, there exist compact sets $K_i \subset M_i$ of
  measure $\mu_i(K_i) \geq 1- \varepsilon/(2\Ctau)$.  Let
  $K = K_1 \times K_2 \subset M$. Then,
  \[
    \begin{split}
      (\phi_A^{t})_*\mu_\psi(K) &= \widetilde{\mu_1}\{x \in M_1 : x \in
      K_1\ab_1^{-t},\
      \psi(x) \in K_2\ab_2^{-t}\} \\
      &\geq 1- \widetilde{\mu_1}(M_1 \setminus K_1\ab_1^{-t})
      - \widetilde{\mu_2}(M_2 \setminus K_2\ab_2^{-t}) \\
      & \geq 1- \Ctau \mu_1(M_1 \setminus K_1) - \Ctau \mu_2(M_2 \setminus K_2)
      \geq 1- \varepsilon.
    \end{split}
  \]
  This proves tightness and hence the compactness claim.
\end{proof}

By \cref{lem:tight}, there exist weak-* limit points of the family
$\{(\phi_A^{t})_* \mu_\psi : t>0 \}$.  The following result shows that any such
limit point $\nu$ must be invariant under diagonal action of the unipotent flows
$(x,y)\mapsto (x\ub_1^t, y\ub_2^t)$.

\begin{lemma}\label{lem:invariance_limit_nu}
  Let $\nu$ a weak-* limit point of $\{(\phi_A^{t})_* \mu_\psi : t>0 \}$.  Then,
  $(\phi_U^r)_*\nu = \nu$, where $\phi_U^r(x,y) = (x\ub_1^r, y\ub_2^r)$.
\end{lemma}

\begin{proof}
  Let $t_n \to \infty$ be an increasing sequence such that
  $\mu_n := (\phi_A^{t_n})_* \mu_\psi$ converges weakly to $\nu$. It is enough
  to show that
  \begin{equation}\label{eq:claim_limit_pt}
    \lim_{n \to \infty}\abs{((\phi_U^r)_*\mu_n)(f_1 \otimes f_2)- 
      \mu_n(f_1 \otimes f_2)} = 0,
  \end{equation}
  for every pair of bounded Lipschitz functions $f_i \colon M_i \to \C$.

  Let us call $r_n = e^{t_n}r$. We compute
  \[
    \begin{split}
      ((\phi_U^r)_*\mu_n)(f_1 \otimes f_2) &= \mu_n(f_1 \circ \phi_1^r \otimes
      f_2 \circ \phi_2^r) = \mu_\psi(f_1 \circ \phi_1^r \circ
      \phi_{A_1}^{t_n}\otimes
      f_2 \circ  \phi_2^r\circ \phi_{A_2}^{t_n}) \\
      &=\mu_\psi(f_1\circ \phi_{A_1}^{t_n} \circ \phi_1^{r_n } \otimes
      f_2\circ \phi_{A_2}^{t_n}  \circ \phi_2^{r_n })\\
      &=\mu_\psi(f_1\circ \phi_{A_1}^{t_n} \circ
      \widetilde{\phi_1}^{\xi_1(\cdot, r_n )} \otimes
      f_2\circ \phi_{A_2}^{t_n}  \circ \widetilde{\phi_2}^{\xi_2(\cdot, r_n )}) \\
      &=\mu_\psi(f_1\circ \phi_{A_1}^{t_n} \circ
      \widetilde{\phi_1}^{\xi_1(\cdot, r_n )-r_n} \otimes
      f_2\circ \phi_{A_2}^{t_n}  \circ \widetilde{\phi_2}^{\xi_2(\cdot, r_n )-r_n}), \\
    \end{split}
  \]
  where, in the last line, we used the invariance of $\mu_\psi$ under the
  product flow $\widetilde{\phi_1^t} \times \widetilde{\phi_2^t}$.  We will now
  call $e_i(x,n) := \xi_i(x,r_n) - r_n$.  Denoting by $C_i$ the Lipschitz
  constant of $f_i$, we obtain
  \begin{multline*}
    \abs{f_i\circ \phi_{A_i}^{t_n} \circ \widetilde{\phi_i}^{e_i(x,n) }(x)
      -f_i\circ \phi_{A_i}^{t_n} (x)} =
    \abs{f_i(x \ub_i^{w_i(x,e_i(x,n))} \ab_i^{t_n} ) -f_i(x \ab_i^{t_n} )   }\\
    \qquad \qquad = \abs{f_i(x \ab_i^{t_n} \ub_i^{e^{-t_n} w_i(x,e_i(x,n))})
      -f_i(x \ab_i^{t_n} ) } \leq C_i e^{-t_n} \abs{w_i(x,e_i(x,n))}
  \end{multline*}
  for all $x \in M_i$.

  Fix $\varepsilon >0$. By \cref{thm:mixing_condition}, there exist sets
  $Y_i \subset M_i$ of measure $\widetilde{\mu_i}(Y_i) \geq 1-\varepsilon$ and
  $N>1$ such that for all $x \in Y_i$ and $n \geq N$ we have
  \begin{multline*}
    \abs{f_i\circ \phi_{A_i}^{t_n} \circ \widetilde{\phi_i}^{e_i(x,n) } (x)
      -f_i\circ \phi_{A_i}^{t_n} (x)} \leq C_i e^{-t_n} \abs{w_i(x,e_i(x,n))} \\
    \qquad \qquad \leq C_i \, \Ctau e^{-t_n}\abs{e_i(x,n) } \leq C_i \Ctau^{-3}
    e^{-\eta t_n} r^{1-\eta} < \varepsilon.
  \end{multline*}
  The claim \eqref{eq:claim_limit_pt} follows combining the bounds
  \begin{multline*}
    \vert((\phi_U^r)_*\mu_n)(f_1 \otimes f_2)  - \mu_\psi( \one_{Y_1} \cdot (f_1\circ \phi_{A_1}^{t_n}  \circ \widetilde{\phi_1}^{e_1(\cdot ,n) } ) \otimes \one_{Y_2} \cdot ( f_2\circ \phi_{A_2}^{t_n}  \circ \widetilde{\phi_2}^{e_2(\cdot ,n)})) \vert \\
    \quad \leq \|f_1\|_\infty \cdot \|f_2\|_\infty [\widetilde{\mu_1}(M_1
    \setminus Y_1) +\widetilde{\mu_1}(\psi^{-1}(M_2 \setminus Y_2)) ] \leq
    2\|f_1\|_\infty \cdot \|f_2\|_\infty \varepsilon
  \end{multline*}
  and
  \begin{multline*}
    \vert \mu_n(f_1 \otimes f_2)  -\mu_\psi( \one_{Y_1} \cdot (f_1\circ \phi_{A_1}^{t_n}  \circ \widetilde{\phi_1}^{e_1(\cdot ,n)} ) \otimes \one_{Y_2} \cdot ( f_2\circ \phi_{A_2}^{t_n}  \circ \widetilde{\phi_2}^{e_2(\cdot ,n)})) \vert \\
    \leq \Big\vert \int_{Y_1 \cap \psi^{-1}(Y_2)} f_1\circ \phi_{A_1}^{t_n}  \circ \widetilde{\phi_1}^{e_1(x ,n)}(x) f_2\circ \phi_{A_2}^{t_n}  \circ \widetilde{\phi_2}^{e_2(\psi(x) ,n)}(\psi(x)) \\
    \quad - f_1\circ \phi_{A_1}^{t_n} (x) f_2\circ \phi_{A_2}^{t_n}  (\psi(x))\diff \widetilde{\mu_1}(x) \Big\vert +  \|f_1\|_\infty \cdot \|f_2\|_\infty  \cdot \widetilde{\mu_1}[M_1 \setminus (Y_1 \cap \psi^{-1}(Y_2))]\\
    \leq (\|f_1\|_\infty + \|f_2\|_\infty + 2\|f_1\|_\infty \cdot
    \|f_2\|_\infty) \varepsilon,
  \end{multline*}
  thus the proof is complete.
\end{proof}

By \cref{lem:invariance_limit_nu}, the projections $(p_i)_*\nu$ of any limit
point $\nu$ as above are invariant under the action of $\ub_i^t$.  Thanks to
Ratner's Measure Classification
Theorem~\cite{Ratner:Raghunathan1,Ratner:Raghunathan2,Ratner:Raghunathan3}, any
of its ergodic component is an algebraic measure supported on a closed subgroup
$\exp(\R\ub_i) < L_i < G_i$.  Since the measures
$(p_i)_* [(\phi_A^{t})_* \mu_\psi]$ are all absolutely continuous with respect
to the Haar measure on $G_i$ with density uniformly bounded away from 0, it
follows that $L_i =G_i$. Therefore, the projections $(p_i)_*\nu$ are the Haar
measures $\mu_i$ on $G_i$.  In summary, we have shown that any weak-* limit
$\nu$ is in turn an ergodic joining of $\ub_1^t$ and $\ub_2^t$.  To this we can
apply the Joinings Theorem of~\cite{Ratner:Raghunathan2}, as we will now
explain.

We want to use the limiting joining $\nu$ to construct a measurable map
$\zeta\colon M_1\to M_2$. By construction this map will be the limit of the
sequence $\psi_{t_n}$.

By algebraicity of $\nu$ there exists a point $x(\nu)\in M$ such that the
measure $\nu$ is supported on the orbit $x(\nu)\Lambda(\nu)$, where
$\Lambda(\nu)=\Stab_G(\nu)$ are the elements of $G$ which leave invariant the
measure $\nu$. Let us define the groups
\[
  \Lambda_1(\nu) = \{\gb\in G_1 : (\gb, \eb)\in\Lambda(\nu) \}, \qquad
  \Lambda_2(\nu) = \{\gb\in G_2 : (\eb, \gb)\in\Lambda(\nu) \}.
\]
In~\cite{Ratner:Raghunathan2}, it is shown that these groups are closed normal
subgroups of $G_1$ and $G_2$ respectively and, moreover,
$\Gamma_i \cap \Lambda_i$ is a lattice in $\Lambda_i$.  In particular, this
means that $\Gamma_i \Lambda_i$ is a closed orbit in
$M_i = \Gamma_i \backslash G_i$. Therefore, $\Gamma_i \Lambda_i$ is a discrete
subgroup of the factor $G_i / \Lambda_i$.  Since the lattices $\Gamma_i$ are
irreducible, this implies that either $\Lambda_i = G_i$, or $\Lambda_i$ is
\emph{finite}.  We remark that the first possibility occurs if and only if $\nu$
is the trivial joining $\mu_1 \otimes \mu_2$, thus we now assume that $\nu$ is
not $\mu_1 \otimes \mu_2$.

For $x\in M_1$ let
\[
  \zeta(x) =\zeta_\nu(x) = \{y\in M_2: (x,y)\in x(\nu)\Lambda(\nu)\}
\]
the fiber of $\nu$ at $x$. The finiteness of the $\Lambda_i$'s implies that this
fiber is finite. In fact, it follows from~\cite[Theorem~2]{Ratner:Raghunathan2}
that there exist an element $\overline{\gb}\in G_2$ and a continuous surjective
homomorphism $\varpi\colon G_1 \to G_2/\Lambda_2$, with kernel $\Lambda_1$ such
that $\varpi(\ub_1)=\ub_2 \Lambda_2$ and
\[
  \zeta(\Gamma_1 \xb) = \{\Gamma_2 \overline{\gb}\beta_i \varpi(\xb),
  i=1,\dotsc,n\},
\]
where $\varpi(\Gamma_1)=\{\overline{\Gamma}\beta_i: i=1,\dotsc,n\}$, with
$\overline{\Gamma} = \varpi(\Gamma_1) \cap
\overline{\gb}^{-1}\Gamma_2\overline{\gb}\Lambda_2$ of finite index, equal to
$n$, inside both $\varpi(\Gamma_1)$ and
$\overline{\gb}^{-1}\Gamma_2\overline{\gb}\Lambda_2$. Up to passing to a finite
quotient, we can assume that both $\Lambda_i$ are trivial and that
$\varpi\colon G_1\to G_2$ is an isomorphism, with
$\varpi(\Gamma_1)\subset \overline{\gb}^{-1}\Gamma_2\overline{\gb}$. Hence, we
obtain a map $\zeta\colon M_1\to M_2$, given by
$\zeta(\Gamma_1 \xb) = \Gamma_2 \overline{\gb} \varpi(\xb)$, for $\mu_1$-a.e.\
$\Gamma_1\xb$, which, by construction satisfies
$\zeta(x\ub_1^s)=\zeta(x)\ub_2^s$.

\subsubsection{}
We are now ready to prove our main result.

\begin{proof}[Proof of \cref{thm:main}]
  We want to show that the conjugacy $\psi$ is given by
  $\psi(x) = \zeta(x)\mathbf{c}(x)\ub_2^{t(x)}$, where $\mathbf{c}(x)$ belongs
  to the centraliser of $\ub_2$, and $t(x)\in\R$.
    
  By compactness of the set $\Good$, we can find an $n_0$ such that if
  $n\ge n_0$ and $x\in\Good$ then $d(\psi_{t_n}(x), \zeta(x))<\rho$. Let
  $\Good_n = \phi_{A_1}^{t_n}(\Good) \cap\Good$, whose measure is
  $\mu_1(\Good_n)>1-2\omega$.  Denote $Q$ the generic set of $\Good_n$ for
  $\ub_1^t$ and $Q_n = \Good_n \cap Q$. Then for $x\in Q_n$ we introduce the set
  \[
    A = A(x) = \{s\in\R_+: x\ub_1^s\in \Good_n\}.
  \]
  By genericity we have that the relative length of $A$ is large. In fact, as
  $T\to\infty$, we have
  \[
    \frac{\abs{A\cap [0,T]}}{T}\to 1-2\omega > 1 - \frac{\theta}{8},
  \]
  where the last inequality follows by our choice of $\omega$ in
  \S\ref{sec:centraliser_1}.

  Let $x'=\psi_{t_n}(x)\in M_2$. By \cref{thm:action_on_normaliser}, we can
  define a function $\sigma$ such that
  \[
    x'\ub_2^{\sigma(s)} = \psi_{t_n} (x\ub_1^s).
  \]
  We want to apply Ratner's Basic Lemma~\ref{lemma:basic} to the pair of points
  $x'$ and $\zeta(x)$ and the set $A$. By construction, if $s\in A$, then
  \[
    d(x'\ub_2^{\sigma(s)},\zeta(x)\ub_2^s) = d(\psi_{t_n}
    (x\ub_1^s),\zeta(x\ub_1^s))< \rho.
  \]
  Since $0\in A$, it only remains to check the last condition in the Lemma.  Let
  us assume that $s'-s>m$ and show that
  $\abs{(\sigma(s')-\sigma(s))-(s'-s)}\le 4(s'-s)^{1-\eta}$. The case when
  $\sigma(s')-\sigma(s)>m$ is similar. We have that
  $x\ub_1^s\ab_1^{-t_n}\in\Good$ and
  $\psi(x\ub_1^s\ab_1^{-t_n})\in\psi(\Good)$. Moreover
  \[
    x\ub_1^{s'}\ab_1^{-t_n} = x\ub_1^s\ab_1^{-t_n}\ub_1^{r_n(s'-s)},
  \]
  and
  \[
    \psi(x\ub_1^{s'}\ab_1^{-t_n}) =
    \psi(x\ub_1^s\ab_1^{-t_n})\ub_2^{z(x\ub_1^s\ab_1^{-t_n},r_n(s'-s))}.
  \]
  Hence, by definition of $\sigma(s)$ we have
  \[
    \begin{split}
      x'\ub_2^{\sigma(s')} &= \psi_{t_n} (x\ub_1^{s'}) \\
      &=  \psi_{t_n} (x\ub_1^s) \ub_2^{z(x\ub_1^s\ab_1^{-t_n},r_n(s'-s))/r_n} \\
      &= x' \ub_2^{z(x\ub_1^s\ab_1^{-t_n},r_n(s'-s))/r_n}\ub_2^{\sigma(s)},
    \end{split}
  \]
  which implies
  \[
    \sigma(s')-\sigma(s) = \frac{z(x\ub_1^s\ab_1^{-t_n},r_n(s'-s))}{r_n}.
  \]
  As in \cref{lemma:z_for_centraliser}, we have
  \[
    \abs{z(x\ub_1^s\ab_1^{-t_n},r_n(s'-s))-r_n(s'-s)} \le
    \frac{2}{\Ctau}r_n(s'-s)^{1-\eta},
  \]
  and, finally,
  \[
    \abs{(\sigma(s')-\sigma(s))-(s'-s)}\le 4(s'-s)^{1-\eta}.
  \]
  By the Basic Lemma~\ref{lemma:basic}, we have that
  $x'=\zeta(x)\mathbf{c}\ub_2^t$, for some time $t=t(x)\in\R$ and where
  $\mathbf{c}=\mathbf{c}(x)$ commutes with $\ub_2$.

  By construction, we have that
  $\zeta\circ \phi_{A_1}^t = \phi_{A_2}^t \circ\zeta$. Then, for every point
  $x\in \phi_{A_1}^{-t_n} Q_n$ we have $\psi(x) = \zeta(x)\mathbf{c}\ub_2^t$, as
  above. The set
  \[
    \{x\in M_1: \psi(x) = \zeta(x)\mathbf{c}\ub_2^t, \text{ for some $t\in\R$
      and $\mathbf{c}$ commuting with $\ub_2$} \}
  \]
  is $\ub_1^t$-invariant. Since it contains the set $\phi_{A_1}^{-t_n} Q_n$,
  which has positive measure, we conclude by ergodicity of $\ub_1^t$, that
  $\psi(x)= \zeta(x)\mathbf{c}\ub_2^t$, for $\mu_1$-a.e.\ $x$, as we wanted.
\end{proof}

%% file: 07_sl2xG.tex

\section{Proof of \cref{thm:sl2xG}}\label{sec:generalisingRatner}

In this section we prove \cref{thm:sl2xG}, which is a direct generalisation of
the main result of~\cite{Ratner:Acta}. We will specialise to the case
$G_i=\SL_2(\R)\times G_i'$, where, for $i=1,2$, $G_i'$ is a connected semisimple
linear group with algebra $\mathfrak{g}_i'$, and $\Gamma_i \leq G_i$ is an
irreducible lattice.  We consider the unipotent 1-parameter subgroup
\[
  \ub_i^t = \begin{pmatrix}
    1 & t \\
    0 & 1
  \end{pmatrix}\times\eb,
\]
where $\eb=\eb_i\in G'_i$ is the identity.  We remark that $\ub_i^t$ commutes
with multiplication by elements in $G_i'$. We let
\[
  \ab_i^t = \begin{pmatrix}
    e^{\frac{t}{2}} & 0 \\
    0 & e^{-\frac{t}{2}}
  \end{pmatrix}\times\eb,
  \qquad \text{and} \qquad \overline{\ub}_i^t = \begin{pmatrix}
    1 & 0 \\
    t & 1
  \end{pmatrix}\times\eb,
\]
and denote by $U_i$, $A_i$ and $\overline{U}_i$ the corresponding elements in
the Lie algebras $\mathfrak{g}_i$.  Note that any element
$\gb_i \in N_{G_i}(\ub_i)$ of the normaliser $N_{G_i}(\ub_i)$ of $\ub_i$ can be
written as $\gb_i = \ab_i^a \overline{\ub}_i^{\overline{u}} \gb_i'$, where
$\gb_i' \in G_i'$ and $a,\overline{u} \in \R$.

\subsection{Convergence of $\psi_{t_n}$} 
In the previous section, in order to ensure the convergence, up to a
subsequence, of $\psi_{t_n}$ we had to pass to joinings, and then exploit the
fact that sequences of (tight) measures have limit points.
In the more restricted setting of this section, we can prove directly, following
Ratner's original strategy in~\cite{Ratner:Acta}, that $\psi_{t_n}$
converges. The key difference is that now we know that the \emph{only} direction
which is expanded, when we push by the Cartan elements $\ab_i$, is the one of
the unipotent flow.

\subsubsection{}
We recall that, in \S\ref{sec:centraliser_1} we have constructed, for any
$n\in\N$ a set $K_{\ub_1}(n)\subset{M_1}$ of measure greater than $1-2^{-n}$
such that, if $x\in K_{\ub_1}(n)$ and $s\geq s_n$, then the $\ub_1$-orbit of $x$
enters $\Good$ before time $s_n$, and actually belongs to $\Good$ for a large
frequency of times.

Fix a constant $0<\gamma<\frac{\eta}{2}$ and let $t_n = \log s_n^{1+\gamma}$.
As in \S\ref{sec:conclusion} we will consider the sequence of isomorphisms given
by $\psi_{t_n}(x) = \psi(x\ab_1^{-t_n})\ab_2^{t_n}$. Let
\begin{equation}\label{eq:defV}
  V = \bigcap_{n\in\N} \phi_{A_1}^{t_n}(K_{\ub_1}(n)).
\end{equation}
Since $\mu_1(K_{\ub_1}(n))>1-2^{-n}$, then $\mu_1(V)>\frac{1}{2}$.

We now show that the restriction to $V$ of $\psi_{t_n}$ \lq\lq commutes in the
limit\rq\rq\ with the multiplication by elements in the normaliser of $\ub_1$.
\begin{lemma}\label{lemma:conmutationU}
  Let $\gb_1 \in N_{G_1}(\ub_1)$ with $d(\ab, g_1)\leq \delta$. There exists
  $\Phi(\gb_1) \in N_{G_2}(\ub_2)$ with $d(\eb, \Phi(\gb_1))\leq \varepsilon$
  such that for any $x \in V$, and any $y = x\gb_1 \in V$ we have
  \[
    d(\psi_{t_n}(x)\Phi(\gb), \psi_{t_n}(y)) \to 0,
  \]
  as $n\to\infty$.
\end{lemma}
\begin{proof}
  We can write $\gb_1 = \gb_1' \ab_1^a \ub_1^u$ for some $|a|, |u| \leq \delta$
  and $\gb_1' \in G_1'$ with $d(\eb, \gb_1') \leq \delta$. Define
  $z=x\gb_1' \ab_1^a$.

  Let $x_n = x\ab_1^{-t_n}$, $z_n = z\ab_1^{-t_n} = x_n \gb_1' \ab_1^a$, and
  $y_n = y \ab_1^{-t_n} =z_n \ub_1^{s_n^{1+\gamma} u}$.  Since
  $x_n \in K_{\ub_1}(n)$, \cref{thm:action_on_normaliser} and the discussion
  after it imply that we can write
  \[
    \psi(z_n) = \psi(x_n) \Phi(\gb_1' \ab_1^a) \ub_2^{\beta_n},
  \]
  where $|\beta_n|\leq s_n$ and
  $ \Phi(\gb_1' \ab_1^a) \in \exp (\R A_2 \oplus \mathfrak{g}_2' )$.  Therefore,
  since $\psi(y_n) = \psi(z_n) \ub_2^{z(z_n,s_n^{1+\gamma} u)}$, we deduce that
  \[
    \psi(y_n) \, \ab_2^{t_n} = \psi(x_n) \ab_2^{t_n} \Phi(\gb_1' \ab_1^a)
    \ub_2^{s_n^{-(1+\gamma)}(\beta_n + z(z_n,s_n^{1+\gamma} u))}.
  \]
  Note that $|s_n^{-(1+\gamma)} \beta_n | \to 0$ as $n \to \infty$, hence we can
  focus on the term $s_n^{-(1+\gamma)} z(z_n,s_n^{1+\gamma} u)$.

  By definition, since $z_n \in K_{\ub_1}(n)$, there exists $0\leq r_n \leq s_n$
  such that $z_n \ub_1^{r_n} \in \Good$. Using the cocycle relation of
  \cref{lemma:z_cocycle} and the properties of the set $\Good$ in
  \cref{prop:good_set}, we deduce that
  \begin{equation*}
    \begin{split}
      |z(z_n,s_n^{1+\gamma} u) - s_n^{1+\gamma} u| &\leq |z(z_n,r_n)| +
      |z(z_n \ub_1^{r_n} ,s_n^{1+\gamma} u - r_n) - s_n^{1+\gamma} u| \\
      &\leq C_\alpha^2 s_n + r_n + C_{\alpha}^{-4} (s_n^{1+\gamma} u -
      r_n)^{1-\eta}.
    \end{split}
  \end{equation*}
  Thus, it follows that $s_n^{-(1+\gamma)} z(z_n,s_n^{1+\gamma} u) \to u$.  We
  conclude that
  \[
    d(\psi_{t_n}(y), \psi_{t_n}(x)\Phi(\gb_1' \ab_1^a) \ub_2^u) \leq
    |s_n^{-(1+\gamma)} z(z_n,s_n^{1+\gamma} u) - u| \to 0,
  \]
  which proves the lemma.
\end{proof}

\begin{lemma}\label{lemma:conmutationUbarA}
  For any $x \in V$ and $y = x\overline{\ub}_1^{\delta_0} \in V$, with
  $\abs{\delta_0}<\delta$, we have
  \[
    d(\psi_{t_n}(x)\overline{\ub}_2^{\delta_0}, \psi_{t_n}(y)) \to 0,
  \]
  as $n\to\infty$.
\end{lemma}

The proof of this result follows the same strategy and actually follows quite
closely the proof of \cref{thm:centraliser}. We will use the parametrisation
$q(t)$ we defined in \S\ref{sec:geom_U_flow_1}.

\begin{proof}
  Let $x_n = x\ab_1^{-t_n}\in K_{\ub_1}(n)$ and
  $y_n = y\ab_1^{-t_n} = x_n \overline{\ub}_1^{\delta_n}\in K_{\ub_1}(n)$, where
  $\delta_n = \delta_0 s_n^{-1-\gamma}$. Let $t^*\in[0,s_n]$ the first time that
  both the $\ub_1$-orbit of $x_n$ and $y_n$ belong to $\Good$. Call
  $v_n = x_n \ub_1^{t^*}$ and $w_n = y_n \ub_1^{t^*}$. Then, we have that
  $w_n = v_n \ab_1^a \overline{\ub}_1^{\bar u}$ for some very small $a$ and
  $\bar u$.  More precisely, using~\eqref{eq:sl2R:1}, we have that
  \begin{equation}\label{eq:estimatesaubar}
    a = 2 \log (1-\delta_nt^*), \qquad \bar u = \delta_n(1-\delta_n t^*),
  \end{equation}
  in particular, $\bar u \le \delta_n$. Using the notation of
  \S\ref{sec:geom_U_flow_1}, the parametrization $q(t)$ is defined by
  \[
    v_n \ub_1^t \exp(a(t)A_1) \exp(\bar u(t) \overline{U}_1) = w_n \ub_1^{q(t)}.
  \]
  where, by~\eqref{eq:sl2R:2},
  \[
    q(t) = t-\frac{\bar u t^2 +t(1-e^a)}{(e^a-\bar u t)}.
  \]
  Call
  \[
    B_n = \{t\in [0,s_n] : v_n\ub_1^{t}\in\Good \text{ and }
    w_n\ub_1^{q(t)}\in\Good\},
  \]
  which is a compact subset of $[0,s_n]$. We remark that $0\in B_n$.
    
  The fact that $x_n\in K_{\ub_1}(n)$ implies that its $\ub_1$-orbit is outside
  of $\Good$ for at most $\frac{1}{40\Ctau^4}s_n$ time, and the same holds for
  $y_n$. Since we are assuming that $t^*$ is the first time both their orbits
  are in $\Good$, then the orbits of $v_n$ and $w_n$, under $\ub_1$, spend at
  least $(1-\frac{1}{20\Ctau^4})ss_n$ of their time inside $\Good$. Finally, as
  $\abs{q'(t) -1}\leq \delta_0 s_n^{-\gamma}$, which is arbitrarily small if $n$
  is sufficiently large, we deduce that
  \begin{equation}\label{eq:sizeofBn}
    \frac{\abs{B_n}}{s_n} \ge \left(1 - \frac{\theta}{18\Ctau^4}\right).
  \end{equation}
  By definition, using again~\eqref{eq:sl2R:1}, we have that
  \[
    d(v_n \ub_1^t, w_n \ub_1^{q(t)}) \leq \abs{a(t)} + \abs{\bar u(t)} \leq 4
    \delta_n s_n \leq \delta
  \]
  for all $t\in B_n$. Therefore the points $v'_n =\psi(v_n)$ and
  $w'_n =\psi(w_n)$ satisfy
  \[
    v_n' \ub_2^{z(v_n,t)} \in \psi(\Good), \qquad w_n' \ub_2^{z(w_n,q(t))} \in
    \psi(\Good),
  \]
  and
  \[
    d( v_n' \ub_2^{z(v_n,t)}, w_n' \ub_2^{z(w_n,q(t))} ) \leq \rho, \qquad
    \text{ for all } t \in B_n.
  \]
    
  Let us consider the point
  $\overline{v}_n' = v_n' \ab_2^{a} \overline{\ub}_2^{\bar u}$. As before, we
  have that
  \[
    d( \overline{v}_n' \ub_2^{q(t)}, v_n' \ub_2^{t} ) \leq \rho,
  \]
  for all $t\in B_n$, where the function $q(t)$ is the same as above.

  Let $t=t(r)$ be the Lipschitz function with $t(0)=0$ uniquely defined by
  $r=z(w_n,q(t))$, and let
  \[
    \chi(r) = q(z(v_n, t(r))).
  \]
  It is easy to see that also $\chi$ is Lipschitz and satisfies $\chi(0)=0$.
  Then, by the triangle inequality,
  \[
    d( \overline{v}_n' \ub_2^{\chi(r)}, w_n' \ub_2^{r} ) \leq \rho, \qquad
    \text{ for all } r \in B_n' := z(w_n, q(B_n)).
  \]
  We remark that $B_n'$ is compact and that $0\in B_n'$.

  Let $t_1 = \sup B_n$ and $\lambda_n = z(w_n, q(t_1))$.
  From~\eqref{eq:sizeofBn} we have
  \[
    \left(1 - \frac{\theta}{18\Ctau^4}\right)s_n \le t_1 \le s_n.
  \]
  By the uniform Lipschitz estimate on $z$, this yields
  \[
    \Ctau^{-2}\left(1 - \frac{\theta}{18\Ctau^4}\right)s_n \le \lambda_n \le
    \Ctau^2 s_n,
  \]
  and
  \[
    \frac{\abs{B_n'}}{\lambda_n} \ge 1 - \frac{\theta}{16}.
  \]
  We let
  \[
    A_n = \{0\} \cup (B_n' \cap [m, s_n]).
  \]
  Then, if $n$ is sufficiently large, we have
  \[
    \frac{\abs{A_n}}{\lambda_n} \ge 1 - \frac{\theta}{8}.
  \]
 
  Since $r\in B_n'$ if and only if $t(r) \in B_n$, then
  $w_n' \ub_2^r = w_n' \ub_2^{z(w_n, q(t))}\in \psi(\Good)$ and hence statisfies
  the Injectivity Condition $\IC(\rho,m)$ and the Frequently Bounded Radius
  Condition $\FBR(T_0, c(b), r_0)$.

  We want to apply the Basic Lemma, to the parametrization $\chi$ and the points
  $w_n'$ and $\overline{v}_n'$, with the exponent $\gamma$ instead of $\eta$ in
  \cref{lemma:basic}. It only remains to verify Condition~4. To do so, let
  $r', r \in B_n'$ and assume that either $r' - r > m$ or that
  $\chi(r') - \chi(r)> m$. Let $t'=t(r')$ and $t=t(r)$ be the corresponding
  times in $B_n$. We claim that we have that
  \begin{equation}\label{eq:basic_lemma_chi}
    q(t') - q(t) > m_0, \qquad \text{and} \qquad t' - t > m_0.
  \end{equation}
  Assuming the claim for a moment, let us estimate
  $\abs{(\chi(r')-\chi(r))-(r'-r)}$.

  By the cocycle relation, we have that
  \begin{equation}\label{eq:rrz}
    r' - r = z(w_n, q(t')) - z(w_n, q(t)) = z(w_n\ub_1^{q(t)}, q(t')-q(t)),
  \end{equation}
  with $w_n\ub_1^{q(t)}\in\Good$. Then, by \cref{thm:mixing_condition} applied
  to $z$ (see also the proof of \cref{lemma:z_for_centraliser}), and using our
  claim, we obtain
  \[
    \begin{split}
      z(w_n\ub_1^{q(t)}, q(t')-q(t))
      &\le (q(t')-q(t)) + \frac{1}{\Ctau^4}(q(t')-q(t))^{1-\eta} \\
      &\le q'(\xi)\left( (t'-t) + \frac{1}{\Ctau^4}(t'-t)^{1-\eta}\right),
    \end{split}
  \]
  for some $\xi\in[0,s_n]$.

  Similarly we have, for some $\tilde{\xi}\in[0,s_n]$, that
  \begin{equation}\label{eq:chichiz}
    \begin{split}
      \chi(r') - \chi(r)  &= q(z(v_n, t')) - q(z(v_n, t)) \\
      &= q'(\tilde{\xi}) (z(v_n, t') - z((v_n, t)) \\
      &= q'(\tilde{\xi}) z(v_n\ub_1^t, t'- t) \\
      &\le q'(\tilde{\xi}) \left( (t'-t) + \frac{1}{\Ctau^4} (t'-t)^{1-\eta}
      \right),
    \end{split}
  \end{equation}
  since $v_n\ub_1^t\in\Good$.

  Putting the previous inequalities together, we obtain
  \begin{equation}\label{eq:estimatechir}
    \abs{(\chi(r')-\chi(r))-(r'-r)} \le (t'-t)\cdot\abs{q'(\xi)-q'(\tilde \xi)} + 
    \frac{4}{\Ctau^4}(t'-t)^{1-\eta}.
  \end{equation}
  Using again the mean value theorem, we obtain a point $\xi^*$ such that
  $\abs{q'(\xi)- q'(\tilde \xi)} = \abs{\tilde \xi - \xi} \cdot
  \abs{q''(\xi^*)}$. By a direct computation,
  \begin{equation}\label{eq:qsecond}
    \abs{q''(\xi^*)} = \frac{2\bar u e^a}{(e^a - \bar u\xi^*)^3} \le 2\delta_n,
  \end{equation}
  where we used the estimates~\eqref{eq:estimatesaubar} and the fact that
  $\xi^*\in [0,s_n]$ to deduce that $\bar u \xi^* \le \delta_0 s_n^{-\gamma}$
  and hence the denominator is between $0$ and $1$.
    
  Moreover, since $\tilde\xi$ and $\xi$ belong to the interval $[t, t']$ and
  both are lesser than $s_n$, we have that
  \begin{equation}\label{eq:deltan}
    \delta_n = \delta_0 s_n^{-1-\gamma} \le \delta_0 
    \cdot \abs{\tilde \xi - \xi}^{-1} \cdot (t' - t)^{-\gamma}.
  \end{equation}
  Plugging~\eqref{eq:qsecond} and~\eqref{eq:deltan}
  into~\eqref{eq:estimatechir}, we get
  \[
    \abs{(\chi(r')-\chi(r))-(r'-r)} \le 2 \delta_0 (t'-t)^{1-\gamma} +
    \frac{4}{\Ctau^4}(t'-t)^{1-\eta}.
  \]
  We now go back from $(t'-t)$ to $(r'-r)$. Using~\eqref{eq:rrz} and the
  Lipschitz estimate on $z$ we have
  \begin{equation}\label{eq:tandr}
    \frac{1}{2\Ctau^2}(t'-t) \le \frac{1}{\Ctau^2}(q(t')-q(t)) \le r'-r
    \le \Ctau^2(q(t')-q(t)) \le 2\Ctau^2 (t'-t).
  \end{equation}
  Thus, we have shown that
  \[
    \abs{(\chi(r')-\chi(r))-(r'-r)} \le 4 (r' - r)^{1-\gamma},
  \]
  which proves Condition~4 of the Basic Lemma~\ref{lemma:basic}, as we wanted.

  It remains to prove that, if $r' - r > m$ or $\chi(r') - \chi (r) > m$
  then~\eqref{eq:basic_lemma_chi} holds. In the former case,
  from~\eqref{eq:tandr} we have
  \[
    t' - t \ge \frac{r'-r}{2\Ctau^2} > \frac{m}{2\Ctau^2} > m_0.
  \]
  Similarly, since $\frac{1}{2}\le q'(\tilde \xi)\le 2$,
  \[
    q(t') - q(t) = q'(\xi) (t' - t) \ge \frac{t'-t}{2} \ge \frac{r'-r}{4\Ctau^2}
    > \frac{m}{4\Ctau^2} > m_0.
  \]
    
  The latter case, when $\chi(r') - \chi (r) > m$, can be dealt similarly
  using~\eqref{eq:chichiz} to obtain
  \[
    \frac{1}{2\Ctau^2}(t'-t) \le \chi(r') - \chi(r) \le 2\Ctau^2 (t'-t).
  \]

  Having verified all the assumptions of the Basic Lemma~\ref{lemma:basic}, we
  conclude that there exists a $T\in [0, s_n]$ such that
  \[
    \overline{v}'_n \ub_2^T = w'_n \ub_2^{\chi(T)} \gb_c \exp(a_T A_2)
    \exp(\bar{u}_T\overline{U}_2) \ub_2^{\epsilon_T},
  \]
  with $\gb_c\in G'_2$ such that $d(\eb,\gb_c)<\epsilon$, and
  \begin{equation}\label{eq:bounds_on_ophor}
    \abs{a_T} \le C s_n^{-\frac{\eta}{2}}, 
    \qquad \abs{\bar{u}_T}\le C s_n^{-1-\frac{\eta}{2}}, 
    \qquad \abs{\epsilon_T} \le \epsilon.
  \end{equation}
  By the definition of $\overline{v}'_n$, we have
  \[
    v'_n \ab_2^{a} \overline{\ub}_2^{\bar u} = w'_n \ub_2^{\chi(T)} \gb_c
    \ab_2^{a_T} \overline{\ub}_2^{\bar u_T} \ub_2^{\epsilon_T-T},
  \]
  applying $\ab_2^{t_n}$ to both side, this yields
  \[
    v_n' \ab_2^{a+t_n} \overline{\ub}_2^{\bar u s_n^{1+\gamma}} = w_n'
    \ab_2^{t_n} \ub_2^{\chi(T)s_n^{-(1+\gamma)}} \gb_c \ab_2^{a_T}
    \overline{\ub}_2^{\bar{u}_Ts_n^{1+\gamma}}
    \ub_2^{(\epsilon_T-T)s_n^{-(1+\gamma)}}.
  \]
  Using the definition of $v_n'$, the left hand side is equal to
  \[
    \psi_{t_n}\Bigl(x\ub_1^{t^*e^{-t_n}}\Bigr)\ab_2^{a} \overline{\ub}_2^{\bar u
      s_n^{1+\gamma}}.
  \]
  Since $t^*\le s_n$, using the definition of $t_n$
  and~\eqref{eq:estimatesaubar}, we have
  \[
    d\left(\psi_{t_n}\left(x\ub_1^{t^*e^{-t_n}}\right)\ab_2^{a}
      \overline{\ub}_2^{\bar u s_n^{1+\gamma}},
      \psi_{t_n}(x)\overline{\ub}_2^{\delta_0}\right)\to 0
  \]
  as $n\to\infty$. In a similar fashion,
  \[
    w_n' \ab_2^{t_n} = \psi_{t_n}\left(y\ub_1^{t^*e^{-t_n}}\right),
  \]
  which approaches $\psi_{t_n}(y)$ for large $n$. Finally, by the
  estimates~\eqref{eq:bounds_on_ophor}, since $\gamma<\frac{\eta}{2}$,
  \[
    d\left(\psi_{t_n}(y),\ub_2^{\chi(T)s_n^{-(1+\gamma)}} \gb_c \ab_2^{a_T}
      \overline{\ub}_2^{\bar{u}_Ts_n^{1+\gamma}}
      \ub_2^{(\epsilon_T-T)s_n^{-(1+\gamma)}}\right) \to 0,
  \]
  as $n\to\infty$. In summary, we have obtained
  \[
    d(\psi_{t_n}(x)\overline{\ub}_2^{\delta_0}, \psi_{t_n}(y))\to 0,
  \]
  as we wanted.
\end{proof}

We want to show that the sequence of isomorphisms $\psi_{t_n}$ converges to a
function $\zeta$, which we will later show to be a measurable conjugacy of the
unipotent flows $\phi_{U_1}^t$ and $\phi_{U_2}^t$. Using elementary arguments,
given a point $x\in M_1$, we could find a subsequence, depending on the point,
such that $\psi_{t_n}(x)$ converges along that subsequence.  However, thanks to
the previous Lemmas, we can obtain a subsequence which is \emph{independent} of
the point, and ensure convergence of $\psi_{t_n}$ on a full measure set, as we
now show.

\begin{corollary}\label{cor:convergencepsi}
  There exist a $\phi_{U_1}^t$-invariant set $\Omega\subset M_1$ with
  $\mu_1(\Omega)=1$ and a subsequence $(n_k)_{k\in\N}$ such that, if
  $x\in\Omega$ then $\lim_{k\to\infty} \psi_{t_{n_k}}(x)=:\zeta(x)\in M_2$
  exists and $\zeta(x\ub_1^t) = \zeta(x)\ub_2^t$, for all $x\in\Omega$ and
  $t\in\R$.
\end{corollary}

\begin{proof}
  We begin by showing that, on a set of full measure, the sequence of
  isomorphisms $\psi_n$ has a converging subsequence.

  Let $K_n$ be an exhaustion of compact sets of $M_2$ such that
  $\mu_2(M_2\setminus K_n)<2^{-n}$, for all $n\in\N$. Denote $K_n^\complement$
  the complement of $K_n$ inside $M_2$ and
  \[
    F_n = \phi_{A_1}^{t_n} \circ \psi^{-1} \circ \phi_{A_2}^{-t_n}
    (K_n^\complement).
  \]
  Since $\sum_{n=1}^\infty \mu_1(F_n) <\infty$, by the Borel-Cantelli lemma the
  set
  \[
    F = \{ x\in M_1 : x \text{ belongs to finitely many } F_n\}
  \]
  has full $\mu_1$-measure. If $x\in F$, then $\psi_{t_n}(x)$ belongs to
  finitely many sets $K_n^\complement$, so there exists a subsequence
  $n_k = n_k(x)$ such that $\psi_{t_{n_k}}(x)$ converges in $M_2$.
    
  Since the set $V$, defined in~\eqref{eq:defV}, has positive measure, and $F$
  has full measure, we can assume $V\subset F$. Moreover, there exists a point
  $x^* \in V$ such that, if $W(x^*,\delta)$ is the local stable leaf of $x^*$,
  then $V\cap W(x^*,\delta)$ has positive measure, with respect to the natural
  Riemannian volume on the leaf itself. As $x^*\in F$, there exists a
  subsequence $n_k = n_k(x^*)$ such that $\psi_{t_{n_k}}(x^*)$ converges in
  $M_2$. Let
  \[
    \Omega = \{x\ub_1^t, t\in\R, x\in V\cap W(x^*,\delta) \}.
  \]
  By construction, the set $\Omega$ is invariant under the flow defined by
  $\ub_1^t$, and has positive $\mu_1$-measure. By ergodicity, $\mu_1(\Omega)=1$.

  Thanks to \cref{lemma:conmutationU} and \cref{lemma:conmutationUbarA}, we can
  define $\zeta(x) = \lim_{k\to\infty} \psi_{t_{n_k}} (x)$, for every
  $x\in\Omega$, which satisfies $\zeta(x\ub_1^t) = \zeta(x)\ub_2^t$ for all
  $t\in\R$ and $x\in\Omega$.
\end{proof}

We can now complete the proof of \cref{thm:sl2xG}.

\begin{proof}[Proof of \cref{thm:sl2xG}]
  Let $\Omega$ and $n_k$ be given by \cref{cor:convergencepsi}. We assume, for
  simplicity, that $\Omega=M_1$ and $n_k = n$. Then, for every $x\in M_1$ we can
  define the function
  \[
    \zeta(x):=\lim_{n\to\infty} \psi_{t_n}(x),
  \]
  which satisfies
  \[
    \zeta(x\ub_1^t)=\zeta(x)\ub_2^t,
  \]
  for all $x\in M_1$ and all times $t\in\R$. In other words,
  $\zeta\colon M_1 \to M_2$ is a measurable conjugacy between the two unipotent
  flows $\phi_{U_1}^t$ and $\phi_{U_2}^t$. By the Rigidity Theorem
  in~\cite{Ratner:Raghunathan2}, there exists an isomorphism
  $\varpi\colon G_1\to G_2$, with
  $\varpi(\Gamma_1)\subset \overline{\gb}^{-1}\Gamma_2\overline{\gb}$, such that
  $\zeta(\Gamma_1 \xb) = \Gamma_2 \overline{\gb} \varpi(\xb)$, for $\mu_1$-a.e.\
  $\Gamma_1 \xb$. From here, the proof follows verbatim the one of
  \cref{thm:main}.
\end{proof}


%% file: Chevalley.tex
\section{A consequence of Chevalley's Lemma}\label{sec:appendix_Chevalley} In
this short appendix we prove a more precise version of
\cref{lemma:consequence_Chevalley}. We begin by recalling the setup.

Let $H < G$ be a Zariski closed subgroup of $G$ isomorphic to $\SL_2(\R)$, and
let $\mathfrak{s} < \mathfrak{g}$ be its Lie algebra. Let us write $\mathfrak{g}
= \mathfrak{m} \oplus \mathfrak{s}$, where $\mathfrak{m}$ is a
$\mathfrak{s}$-submodule of $\mathfrak{g}$. Then, any $\gb \in G$ with $d(\eb,
\gb)<1$ can be written in a unique way as $\gb =\exp(Y)\hb \in G$, with $\hb \in
H$ and $Y \in \mathfrak{m}$ such that $d(\eb, \hb)<1$ and $d(\eb, \exp(Y))<1$.
In the following lemma, we prove the converse result, namely we will show that
if an element $\gb \in G$ sufficiently close to the identity can be written as
$\gb =\exp(Y)\hb$ with $\hb \in H$ and $Y \in \mathfrak{m}$, then automatically
we have that $\exp(Y)$ and $\hb$ are close to the identity in $G$.

\begin{lemma}\label{lemma:consequence_Chevalley2} For every $\varepsilon>0$
there exists $\delta >0$ such that if $\gb=\exp(Y)\hb \in G$, with $\hb \in H$
and $Y \in \mathfrak{m}$, is such that $d(\eb, \gb) < \delta$, then $d(\eb, \hb)
< \varepsilon$ and $d(\eb, \exp(Y)) < \varepsilon$.
\end{lemma}

In the proof of this result, we will use Chevalley's Lemma, whose proof can be
found in, e.g., \cite[7.9]{Borel:IntroGAri} or \cite[3.1.4]{Zimmer:Erg_th}.

\begin{lemma}[Chevalley's Lemma]\label{lemma:Chevalley} Let $G$ be a linear
algebraic group over $\C$ and let $H <G$ be a Zariski closed subgroup. Assume
that $H$ is semisimple. There exist a rational representation $\rho \colon G \to
\GL(V)$, where $V$ is a finite-dimensional vector space, and a vector $v \in V$
such that $H= \Stab_G(v)$. 
\end{lemma}

\subsubsection{Proof of \cref{lemma:consequence_Chevalley2}}

Let $V,v$, and $\rho$ be given by Chevalley's Lemma. In the following, we will
identify the tangent space $T_vV$ at $v$ with $V$, so that we can simply write
the exponential map $\exp_{T_vV} \colon T_v V \simeq V \to V$ as $\exp_{T_vV}(w)
= v+w$. Let us also fix a norm $\| \cdot \|$ on $V$, and for simplicity assume
$\|v\| = 1$. 

Call $e_v \colon \GL(V) \to V$ the evaluation map $e_v(L) = Lv$. Note that we
can write its differential $\diff e_v \colon \mathfrak{gl}(V) \to V$ as $\diff
e_v (A)= \exp_{\mathfrak{gl}(V)}(A)v - v$. Indeed, if $A \in \mathfrak{gl}(V) $,
then we have
$$
\diff e_v(A) + v = (\exp_{T_vV}\circ \diff e_v)(A) 
= (e_v \circ \exp_{\mathfrak{gl}(V)})(A) = \exp_{\mathfrak{gl}(V)}(A)v, 
$$
which proves the claim.

Define $\Phi = e_v \circ \rho \colon G \to V$. Then, by the chain rule, we can
write its differential $\diff \Phi \colon \mathfrak{g} \to T_vV\simeq V$ as 
$$
\diff \Phi \colon X \mapsto \exp_{\mathfrak{gl}(V)}(\diff \rho(X))v - v 
= \rho(\exp(X)).v-v.
$$
It is immediate to see that $\diff \Phi(X) = 0$ if and only if
$\rho(\exp(X)).v=v$; that is, if and only if $\exp(X) \in \Stab_G(v)=H$. This
shows that $\ker \diff \Phi = \mathfrak{s}$. Identifying $\mathfrak{m}$ with
$\mathfrak{g} / \mathfrak{s}$, we obtain a linear isomorphism (which, by a
little abuse of notation, we still denote by $\diff \Phi $) between
$\mathfrak{m}$ and the image $\diff \Phi (\mathfrak{g})$. 

Let now $\varepsilon >0$ be fixed. There exists $\delta_0 >0$ such that if
$\|X\|_{\mathfrak{g}}<\delta_0$ then $d(\eb, \exp(X)) < \varepsilon /2$. Since
the inverse $(\diff \Phi)^{-1}$ is a linear, and hence continuous, map, there
exists $\delta_1 >0$ such that for all $w \in \diff \Phi (\mathfrak{g})$ with
$\|w\|\leq \delta_1$, we have $\|(\diff \Phi)^{-1}w\|_{\mathfrak{g}}\leq
\delta_0$. Finally, by continuity of $\rho$, let $\delta>0$ be such that if $\gb
\in G$ is such that $d(\eb, \gb) < \delta$, then $\|\rho(\gb) - \Id\| <
\delta_1$. Without loss of generality, we can assume that $\delta < \varepsilon
/2$.

Let $\gb=\exp(Y)\hb \in G$ be such that $d(\eb, \gb) < \delta$. Then, we have
$$
\| \rho(\exp(Y)).v - v\| = \| \rho(\exp(Y)) \rho(\hb) .v - v\| = 
\| \rho(\gb).v - v\| \leq \| \rho(\gb) - \Id\| < \delta_1, 
$$
and, clearly, $\rho(\exp(Y)).v - v = \diff \Phi(Y) \in \diff \Phi
(\mathfrak{g})$. Since $Y \in \mathfrak{m}$, we deduce that $Y = (\diff
\Phi)^{-1} (\diff \Phi(Y))$, and therefore 
$$
\|Y\|_{\mathfrak{g}} = \| (\diff \Phi)^{-1}(\rho(\exp(Y)).v - v)\| < \delta_0. 
$$
We conclude that $d(\eb,\exp(Y)) < \varepsilon /2$ and $d(\eb,\hb) =
d(\eb,\exp(-Y)\gb) \leq \delta +\varepsilon /2 < \varepsilon $, hence the proof
is complete.